\documentclass[11pt, reqno]{amsart}

\usepackage{amsmath}
\usepackage{amsthm}
\usepackage{amssymb}
\usepackage{autobreak}
\usepackage{color}
\usepackage[top=30mm, bottom=30mm, left=30mm, right=30mm]{geometry}
\usepackage[hidelinks]{hyperref}
\usepackage{comment}
\usepackage[shortlabels]{enumitem}

\numberwithin{equation}{section}

\newtheorem{thm}{Theorem}[section]
\newtheorem{prop}[thm]{Proposition}
\newtheorem{lem}[thm]{Lemma}
\newtheorem{cor}[thm]{Corollary}
\newtheorem{defi}[thm]{Definition}

\newtheorem{rem}[thm]{Remark}

\newtheorem*{prop*}{Proposition}

\newcommand{\nc}{\newcommand}

\nc{\ntg}{\notag}
\nc{\fr}[2]{\frac{#1}{#2}}

\nc{\br}{\mathbb{R}}
\nc{\bc}{\mathbb{C}}
\nc{\bz}{\mathbb{Z}}
\nc{\bn}{\mathbb{N}}
\nc{\bs}{\mathbb{S}}

\nc{\Hs}{H^s(\br)}

\nc{\abs}[1]{\left| #1 \right|}
\nc{\jb}[1]{\left\langle #1 \right\rangle }
\nc{\nrm}[1]{\left\| #1 \right\|}
\nc{\snrm}[1]{\| #1 \|}
\nc{\st}{_{S_T}}
\nc{\tx}[2]{_{L_T^{#1}L_x^{#2}}}
\nc{\txx}[1]{_{L_{T, x}^{#1}}}
\nc{\thx}[2]{_{L_T^{#1}H_x^{#2}}}
\nc{\txs}{_{L_T^{\infty}X^s}}
\nc{\lx}[1]{_{L_{x}^{#1}}}
\nc{\lt}[1]{_{L_{T}^{#1}}}
\nc{\cth}[1]{C\left([0, T]; H^{#1}(\br)\right)}
\nc{\ctx}[1]{C\left([0, T]; X^{#1}(\br)\right)}
\nc{\ct}[1]{C\left([0, T]; {#1}\right)}

\nc{\vep}{\varepsilon}
\nc{\del}{\delta}

\nc{\limm}[2]{\lim_{{#1}\rightarrow{#2}}}
\nc{\ra}{\rightarrow}
\nc{\minfty}{{-\infty}}

\nc{\mt}{\mapsto}

\nc{\cl}{\mathcal{L}}
\nc{\cf}{\mathcal{F}}
\nc{\cs}{\mathcal{S}}

\nc{\pt}{\partial_t}
\nc{\px}{\partial_x}

\nc{\hh}[1]{_{H^{#1}}}
\nc{\xx}[1]{_{X^{#1}}}

\nc{\re}{\operatorname{Re}}
\nc{\im}{\operatorname{Im}}

\nc{\I}{\mathrm{I}}
\nc{\II}{\mathrm{I\hspace{-1pt}I}}
\nc{\III}{\mathrm{I\hspace{-1pt}I\hspace{-1pt}I}}
\nc{\IV}{\mathrm{I\hspace{-1pt}V}}
\nc{\V}{\mathrm{V}}
\nc{\VI}{\mathrm{V\hspace{-1pt}I}}
\nc{\VII}{\mathrm{V\hspace{-1pt}I\hspace{-1pt}I}}
\nc{\VIII}{\mathrm{V\hspace{-1pt}I\hspace{-1pt}I\hspace{-1pt}I}}
\nc{\IX}{\mathrm{I\hspace{-1pt}X}}
\nc{\X}{\mathrm{X}}

\nc{\red}[1]{\textcolor{red}{#1}}

\title[WP for QDNLS]{Well-posedness for a nonlinear Schr\"odinger equation with quadratic derivative nonlinearities for bounded primitive initial data}

\keywords{Schr\"odinger equation;
derivative nonlinearity;
Cauchy problem;
well-posedness;
gauge transformation}
\allowdisplaybreaks[4]
\begin{document}
\subjclass[2020]{35Q55, 35B30}
\maketitle
\begin{center}
%\smallskip{Kohei Akase$^1$\footnote{corresponding author}} \\
\smallskip{Kohei Akase$^1$} \\
\smallskip{\footnotesize ${}^1$Department of Mathematics, Graduate School of Science, Osaka University \\ Toyonaka, Osaka, 560-0043, Japan} \\
\smallskip{\footnotesize Email: u306026b@ecs.osaka-u.ac.jp} 

\end{center}
\begin{abstract}
We consider the Cauchy problem for a quadratic derivative nonlinear Schr\"odinger equation whose nonlinearity is a linear combination of  $\px(u^2)$  and $\px(|u|^2)$. We prove the local well-posedness in the $L^2$-based Sobolev space $H^s(\br)$ for $s\ge 0$ with bounded primitives.
Moreover, we prove the global well-posedness in $H^s(\br)$ for $s\ge 1$ and a special case of the coefficients of nonlinearities.
\end{abstract} 
\setcounter{tocdepth}{3}
\section{Introduction}
We consider the Cauchy problem for the following quadratic derivative nonlinear Schr\"{o}dinger equation
\begin{equation}\label{nls}
\left\{
\begin{aligned}
i\pt u + \fr{1}{2}\px^2 u &=\px \left(\lambda u^2 + \mu\abs{u}^2\right),\quad t>0,\ x\in\br,\\
u(0)&=\phi,
\end{aligned}
\right.
\end{equation}
where $u=u(t, x)$ is $\bc$-valued unknown function and $\lambda, \mu\in\bc$.

When $\mu =0$, Christ \cite{Christ} proved that \eqref{nls} is ill-posed in any $L^2$-based Sobolev spaces $H^s(\br)$. 
Thus, some additional condition for the initial data is needed to show the well-posedness in $H^s(\br)$.
In \cite{Ozawa1998}, Ozawa proved the local well-posedness in $X^s(\br)$ with $s\ge 1$ for the following nonlinear Schr\"odinger equation
\begin{equation}\label{Ozawa eq}
i\pt v + \fr{1}{2}\px^2 v =\lambda v\px v + \mu \bar v\px  v + P(v ,\bar v),
\end{equation}
where $\lambda, \mu \in \bc$, and  $P$ is a polynomial that does not have either constant or linear term. 
Here, $X^s(\br)$ is defined by
\begin{equation}\label{Xs}
\begin{gathered}
X^s(\br) :=\left\{f\in H^s(\br)\Biggm|\sup_{x\in\br}\abs{\int_{\minfty}^x f(y)\,dy}<\infty\right\}, \\
\nrm{f}_{X^s}:=\nrm{f}_{H^s} + \sup_{x\in\br}\abs{\int_\minfty^xf(y)\,dy},
\end{gathered}
\end{equation}
for $s>\frac{1}{2}$, where the integral is defined by the improper integral $\lim_{a\to\minfty}\int_a^xf(y)\,dy$.
He used the gauge transformation.
See \cite{Chihara, HO1994} for the gauge transformation.

In this paper, we prove the following theorem by using the gauge transformation and modifying the definition of $X^s(\br)$.
The modified definition of $X^s(\br)$ is given in Definition \ref{def of $X^s$}.

\begin{thm}\label{WP in $X^s$}
Let $s\ge 0$.
Then, the Cauchy problem \eqref{nls} is locally well-posed in $X^s(\br)$.
More precisely, we obtain the following:
\begin{enumerate}[(i)]
\item \it{ Let $s>\frac{1}{2}$ and $R>0$. 
Then, for any $\phi\in X^s(\br)$ with $\|\phi\|_{X^s}\le R$, there exist $T=T(s, R)>0$ and a solution $u\in C([0, T]; H^s(\br))$ to \eqref{nls}.
Also, such a solution is unique in $C([0, T]; X^s(\br))$.
Moreover, the flow map from $\{\phi\in X^s(\br) \mid\nrm{\phi}_{X^s}\le R\}$ to $C([0, T]; X^s(\br))$ is Lipschitz continuous.
}
\item \it{ When $0\le s\le \frac{1}{2}$, for any $R>0$, there exists $T=T(s, R)>0$ such that the flow map $S_T^1: X^{1}(\br)\cap \{\phi\in X^s(\br) \mid\nrm{\phi}_{X^s}\le R\}\to C\left([0, T]; X^{1}(\br)\right)$ extends to the Lipschitz continuous map $S_T^s: \{\phi\in X^s(\br) \mid\nrm{\phi}_{X^s}\le R\}\to C\left([0, T]; X^s(\br)\right)\cap L^4([0, T]; L^\infty(\br))$.
Moreover, the function $S_T^s(\phi)$ solves \eqref{nls} for any $\phi\in X^s(\br)$ with $\|\phi\|_{X^s}\le R$.
}
\end{enumerate}
\end{thm}
The definition of the solution is given in Definition \ref{defi sol}.
In the case of $s$ being small, it is hard to justify the differentiation of the gauge transformation.
Thus, we proved the uniqueness in a weaker sense.
The justification for the differentiation of the gauge transformation is discussed in Section 3.

When we have $2\lambda + \bar\mu = 0$, the $L^2$-norm of the solution to \eqref{nls} is conserved. 
As we see in Section 3, the $L^\infty$-norm of the primitive of the solution is controlled by the $X^0$-norm of the initial data and the $L^2$-norm of $u$. Therefore, we obtain the following:

\begin{thm}\label{special case $X^s$}
Let $s\ge 0$. 
If $2\lambda + \bar\mu =0$, then \eqref{nls} is globally well-posed in $X^s(\br)$.
\end{thm}

Moreover, we can prove the global well-posedness in $H^s(\br)$ under this coefficient case.

\begin{thm}\label{special case $H^s$}
Let $s\ge 1$ and $2\lambda + \bar\mu =0$. 
Then, the Cauchy problem \eqref{nls} is globally well-posed in $H^s(\br)$.
More precisely, for any $\phi\in H^s(\br)$ and $T>0$, there exists a solution to \eqref{nls} which belongs to $C([0, T]; H^s(\br))$.
Also, such a solution is unique in $\{u\in C([0, T]; H^s(\br))\mid\px u\in L^{1}([0, T]; L^{\infty}(\br))\}$. 
Moreover, the flow map from $H^s(\br)$ to $C([0, T]; H^s(\br))$ is continuous.
\end{thm}

There are many works about the Cauchy problem for the following Schr\"odinger equation
\begin{equation}\label{g-nls}
i\pt u + \Delta u = P(u, \bar u, \nabla u, \nabla \bar u), \quad t>0,\  x\in\br^n,
\end{equation}
where $P$ is a polynomial that does not have either constant or linear terms.
When the nonlinearity $P(u, \bar u, \nabla u, \nabla \bar u)$ consists of cubic or more order terms, Kenig, Ponce, and Vega \cite{KPV1993NLS} proved the small data local well-posedness of \eqref{g-nls} in smooth Sobolev spaces with general dimensions. 
Also, Pornnopparath \cite{Pornnopparath} improved their result in the one-dimensional case.
They use the contraction mapping argument.

As \cite{KPV1993NLS, Pornnopparath}, we can uniformly treat the nonlinearity whose order is greater than or equal to 3 in the Sobolev spaces.
On the other hand, when $P(u, \bar u, \nabla u, \nabla \bar u)$ consists of quadratic nonlinearity with derivative, the situation largely differs.
When $P(u, \bar u, \nabla u, \nabla \bar u)=\bar u \partial_{x_k}\bar u$ or $u \partial_{x_k}\bar u$, the Fourier restriction norm method works well to prove the well-posedness in the Sobolev spaces.
See \cite{Grunrock, Hirayama2014} for $\bar u \partial_{x_k}\bar u$ and \cite{IO} for $u \partial_{x_k}\bar u$, respectively.
As these results, we can apply an iteration argument if the quadratic term has a derivative in $\bar u$.
However, the situation becomes more complicated when the quadratic term has a derivative in $u$.
When we consider the case $P(u, \bar u, \nabla u, \nabla \bar u)= \bar u\partial_{x_k} u$, the flow map fails to be twice differentiable in $H^s(\br^n)$ for any $s\in \br$ (see, for example, \cite{IO}), we cannot use an iteration argument in $H^s(\br^n)$.
When $P(u, \bar u, \px u, \px \bar u)= u\px u$, Molinet, Saut, and Tzvetkov \cite{MST} showed that the flow map fails to be twice differentiable in $H^s(\br)$ for any $s\in \br$. Moreover, as mentioned above, Christ \cite{Christ} proved the ill-posedness in $H^s(\br)$ for any $s\in\br$.
We note that we can utilize an iteration argument in weighted Sobolev spaces.
See \cite{KPV1993NLS}.

In \cite{Ozawa1998}, Ozawa transformed \eqref{Ozawa eq} to the system of equations of $v$ and 
\[
\tilde v:=\exp\left(\lambda\int_\minfty^x v(t, y)\,dy + \mu\int_\minfty^x \bar v(t, y)\,dy\right)\px v.
\]
We note that we can use such a transformation since solutions are considered in $X^s(\br)$.
Using this transformation, he obtained a system of equations of $v, \tilde v$ which has no derivative in the nonlinearity.
Thus, an iteration argument works well.

We explain the idea of the proof of our result.
The standard iteration argument in $H^s(\br)$ does not work well since \eqref{nls} has the nonlinearities $u\px u$ and $\px(|u|^2)$.
Also, \eqref{nls} does not have an energy structure.
On the other hand, the gauge transformation works well.
In this paper, we consider the transformation 
\[
\tilde u:=\exp\left(\lambda\int_\minfty^x u(t, y)\,dy + \mu\int_\minfty^x \bar u(t, y)\,dy\right)u.
\] 
Then, the transformed equation has no derivative in the nonlinearity.
Thus, we can prove the well-posedness in $X^s(\br)$ for $s\ge0$.
We note that we cannot apply the argument like \cite{Ozawa1998}.
This is because $\tilde u$ contains no derivative, and thus the system of equations of $u, \tilde u$ has a derivative in the nonlinearity.
Therefore, we first prove the existence of a solution to \eqref{nls} in a smooth space.
To construct a solution to \eqref{nls}, we consider the regularized equation such as \eqref{ep-nls}, and we use the gauge transformation in this equation.
We note that we can justify the gauge transformation since the regularized equation has a solution.
Combining the gauge transformation in the regularized equation and a Bona-Smith type approximation in $X^s(\br)$, we can obtain a solution to \eqref{nls}.
Once we prove the existence of a solution to \eqref{nls} in $X^s(\br)$, then we can justify the gauge transformation for the solution to \eqref{nls} and prove the uniqueness and the Lipschitz continuity of the flow map.

We note that we used the function space $X^s(\br)$ since the gauge transformation is not bounded in $L^2(\br)$.
In the case of the Benjamin-Ono equation
\[
\pt u + \mathcal{H}\px^2 u = u\px u,
\]
where $\mathcal{H}$ denotes the Hilbert transform,
Tao \cite{Tao2004} and Ifrim and Tataru \cite{IT2019} used the gauge transformation that has a form of $\exp(i\int_\minfty^x u(t, y)\,dy) u$.
In the Benjamin-Ono equation, the unknown function is real-valued. 
Thus, there is no problem with the boundedness of the gauge transformation.
On the other hand, the exponent is not necessarily pure imaginary in our case since the unknown function is complex-valued.
Hence, we need to use $X^s(\br)$.
However, as in Theorem \ref{special case $H^s$}, the special coefficient condition makes the exponent pure imaginary.
This enables us to obtain the well-posedness in Sobolev spaces.

The difficulty in proving the well-posedness in the Sobolev space is that we cannot utilize the gauge transformation when considering the difference.
The gauge transformation works well to obtain the a priori estimate.
On the other hand, it does not work when we consider the estimate for the difference in $H^s(\br)$.
Fortunately, the special condition has good structures.
In particular, the $L^2$-norm of the difference between two solutions is controlled by the $L^2$-norm of the difference between two initial data.
Hence, we can use the frequency envelope method that is made by Tao \cite{Tao2001}.

The merit of using the frequency envelope is that once we obtain the estimate for the difference in some weaker space, we can prove the well-posedness by using an a priori estimate which uses the frequency envelope.
Thus, this method works well when $2\lambda + \bar\mu=0$.
We need to use the $L^1([0, t]; L_x^\infty(\br))$-norm of the derivative of the solutions when obtaining the estimate for the $L^2$-norm of the difference between two solutions.
Therefore, we proved the well-posedness in $H^s(\br)$ for $s\ge 1$.

The remainder of this paper is organized as follows. In Section 2, we introduce the notation and some lemmas.
In Section 3, we show the well-posedness for regularized equations and the existence of a solution to \eqref{nls} in $X^s(\br)$ for sufficiently large $s$. 
In Section 4, we prove Theorem \ref{WP in $X^s$}.
In Section 5, we show Theorem \ref{special case $X^s$} and \ref{special case $H^s$}.

\section{Preliminaries}
Throughout this paper, we define the solution to \eqref{nls} by the following definition. 
\begin{defi}\label{defi sol}
Let $s\ge0$, $T>0$.
We say $u\in C([0, T]; H^{s}(\br))$ is a solution to \eqref{nls} when $u$ satisfies
\begin{equation}\label{sol}
u(t) = \phi + \fr i 2 \px^2\left(\int_0^tu(t')\,dt'\right) -i\px\left(\int_0^t (\lambda u^2 + \mu \abs{u}^2)(t')\,dt'\right)
\end{equation}
in $H^{s-2}(\br)$ for all $t\in [0, T]$.
\end{defi}
\noindent Here, we note that if $s\ge 0$ and $f\in H^s(\br)$, then we have $f^2, \abs{f}^2\in H^{s-1}(\br)$. Therefore, the third part on the right-hand side of \eqref{sol} is well-defined.

\subsection{Notation}
Let $\cf \phi (\xi):= \fr{1}{\sqrt{2\pi}}\int_{\br}\phi(x)e^{-ix\xi}\,dx$ be the Fourier transform of $\phi$. We also use the notation $\hat \phi$ to express the Fourier transform of $\phi$. We write $A\lesssim B$ to denote $A\le CB$ for some $C>0$.
Also, we write $A\sim B$ when we have $A\lesssim B$ and $B\lesssim A$.
Let $p$ be an even bump function such that
\[
p(\xi)=1\ \mathrm{on}\ [0, 1],\quad  p(\xi)=0\ \mathrm{if}\ \xi\ge 2.
\]
Let $\bn_0 := \bn\cup\{0\}$. For $k\in \bn_0$, we denote $p_{\le k}(\xi)= p(\xi/2^{-k})$.
Also, we define the inhomogeneous Littlewood-Paley operators $P_{\le k}$ for $k\ge0$ by
\[
\cf(P_{\le k}\phi)(\xi) = p_{\le k}(\xi)\hat \phi (\xi).
\]
If $k\le -1$, we define $P_{\le k}=0$. For $k\in \bz$, let $P_k := P_{\le k} - P_{\le k-1}$. We also define $p_k$ by the symbol of $P_{k}$.
In order to simplify the notation, we denote
\[
P_{[j, k]} := \sum_{\ell = j}^{k}P_{\ell}
\]
for $j, k\in \bz$ with $j\le k$. 
We write $P_{\ge j} := I - P_{\le j-1}$.

For $p, q\in[1, \infty]$ and $T>0$, we denote
\[
L^q_TL^p_x := L^q_t([0, T];L^p_x(\br)).
\]
We also define $L_T^qH_x^s$ and so on in the same manner.
\subsection{Some lemmas}
In this paper, we use the Strichartz estimate. See, for example,  \cite[Lemma 2.1]{GV1992}. 
In the following, we denote $e^{\frac{i}{2}t\px^2}$ by the Schr\"{o}dinger propagator.
\begin{lem}\label{strest}
Let $\phi\in L^2(\br)$ and $f\in L_T^1L_x^2$. 
If $u$ satisfies
\[
u(t)=e^{\frac{i}{2}t\px^2}\phi -i\int_0^te^{\frac{i}{2}(t-t')\px^2}f(t')\,dt' \ \mathrm{in}\ L^2(\br),
\]
then we have
\[
\nrm{u}_{L_T^\infty L_x^2\cap L_T^4L_x^\infty}\lesssim \nrm{\phi}_{L^2} + \nrm{f}_{L_T^1L_x^2}.
\]
\end{lem}

In \cite{Ozawa1998}, the function space $X^s(\br)$ is defined until $s>\fr 1 2$ by \eqref{Xs}.
We extend this space until $s>-\fr 1 2$.  
For $a, x\in \br$ and smooth $\phi$, it holds that
\[
\int_a^x \phi(y)\, dy 
= \cf[\mathbf{1}_{[a, x]}\phi](0)
= \int_\br\fr{e^{ix\eta}-e^{ia\eta}}{i\eta}\hat\phi(\eta)\, d\eta,
\]
where $\mathbf{1}_{[a, x]}$ denotes the characteristic function of the interval $[a, x]$.
Since $\eta^{-1-s}\in L^2([1, \infty))$ if $s>-\fr{1}{2}$, we define $X^s(\br)$ as follows:
\begin{defi}\label{def of $X^s$}
For $s>-\fr 1 2$, we define $X^s(\br)$ by
\begin{gather*}
X^s(\br):= \left\{\phi\in H^s(\br) \Biggm|  \sup_{x\in\br}\abs{ \lim_{a\to\minfty} \int_\br\fr{e^{ix\eta}-e^{ia\eta}}{i\eta}\hat\phi(\eta)\, d\eta}<\infty\right\}, \\
\nrm{\phi}_{X^s} := \nrm{\phi}_{H^s} + \sup_{x\in\br}\abs{ \lim_{a\to\minfty} \int_\br\fr{e^{ix\eta}-e^{ia\eta}}{i\eta}\hat\phi(\eta)\, d\eta}.
\end{gather*}
\end{defi}
%When we define $X^s(\br)$ as above, it is a Banach space.
\begin{lem}\label{Xs is Banach}
When $s>-\fr 1 2$, $X^s(\br)$ is a Banach space.
\end{lem}
\begin{proof}
Let $\{\phi_n\}\subset X^s(\br)$ be a Cauchy sequence in $X^s(\br)$.
Then, for all $n\in \bn$ and $x\in\br$, there exists $\lim_{a\to\minfty} \int_\br\fr{e^{ix\eta}-e^{ia\eta}}{i\eta}\hat\phi_n(\eta)\, d\eta$. We define it by $\Phi_n(x)$. 
From $|e^{ix\eta}-e^{ix_0\eta}|\le \min(|(x-x_0)\eta|, 2)$ and Lebesgue's dominated convergence theorem, $\Phi_n$ is continuous. 
Hence, there exist $\phi\in H^s(\br)$ and $\Phi\in C(\br)\cap L^\infty(\br)$ such that 
$\phi_n \to \phi$ in $H^s(\br)$ and $\Phi_n \to \Phi$ in $C(\br)\cap L^\infty(\br)$.
Considering $e^{ix\eta}-e^{ia\eta}=e^{ix\eta}-1 +1-e^{ia\eta}$, we obtain $\lim_{x\to \minfty}\Phi_n(x)=0$ for any $n$. 
Therefore, we have $\lim_{x\to\minfty}\Phi(x)=0$.
For any $b, x\in\br$, it holds that
\begin{align*}
\Phi(x) - \Phi(b)
&= \lim_{n\to\infty}\left(\Phi_n(x) - \Phi_n(b)\right) \\
&=  \lim_{n\to\infty} \int_\br\fr{e^{ix\eta}-e^{ib\eta}}{i\eta}\hat\phi_n(\eta)\, d\eta \\
&=  \int_\br\fr{e^{ix\eta}-e^{ib\eta}}{i\eta}\hat\phi(\eta)\, d\eta.
\end{align*}
Hence, by taking $b\to\minfty$, we obtain $\phi\in X^s(\br)$ and $\phi_n \to \phi$ in $X^s(\br)$.
\end{proof}
To simplify the notation, for $s>-\fr{1}{2}$ and $\phi\in X^s(\br)$, we define
\begin{equation}\label{def of primitive}
\int_\minfty^x\phi(y)\, dy := \lim_{a\to\minfty}\int_\br\fr{e^{ix\eta}-e^{ia\eta}}{i\eta}\hat\phi(\eta)\, d\eta.
\end{equation}
By the proof of Lemma \ref{Xs is Banach}, $\int_\minfty^x\phi(y)\,dy$ is a bounded continuous function. 
We note that this notation is consistent with the primitive if $s>\fr{1}{2}$. 
Also, we have
\[
\int_\minfty^x\left(a\phi(y)+b\psi(y)\right)\, dy =a\int_\minfty^x\phi(y)\, dy + b\int_\minfty^x\psi(y)\, dy.
\]
for $a, b\in\bc$ and $\phi, \psi\in X^s(\br)$. 
When $s>\fr{1}{2}$ and $\phi\in H^s(\br)$, the Riemann-Lebesgue lemma yields $\px\phi\in X^{s-1}(\br)$ and
\begin{equation}\label{primitive of derivative}
\int_\minfty^x\partial_y\phi(y)\, dy = \phi(x)
\end{equation}
for all $x\in\br$.
\section{Existence of the solution in smooth spaces}

In this section, we prove the existence of the solution to \eqref{nls} in $X^s(\br)$ with sufficiently large $s$.
\begin{prop}\label{existence of a smooth solution}
Let $s>\fr{3}{2}$. 
Then, for any $\phi\in X^s(\br)$, there exist $T=T(s, \nrm{\phi}_{X^s})\in (0 ,1)$ and $u\in C([0, T]; X^s(\br))$ is a solution to \eqref{nls}.
\end{prop}
To prove this proposition, we first consider the following regularized equation
\begin{equation}\label{ep-nls}
\left\{
\begin{aligned}
i\pt u + \fr{1}{2}(1-i\vep)\px^2 u &= \px \left(\lambda u^2 + \mu\abs{u}^2\right),\quad t>0,\ x\in\br,\\
u(0)&=\phi,
\end{aligned}
\right.
\end{equation}
where $\vep\in (0,1)$.

We define the solution to \eqref{ep-nls} by the mild solution. 
Namely, we say $u\in C([0, T]; X^s(\br))$ is a solution to \eqref{ep-nls} when $u$ satisfies the following equation
\begin{equation}\label{mild-sol}
u(t) = U_\vep(t)\phi - i\int_0^tU_\vep(t-t')\px\left(\lambda u^2 + \mu \abs{u}^2\right)(t')\,dt'
\end{equation}
in $X^{s}(\br)$ for all $t\in[0, T]$, where $U_\vep(t) = e^{\fr{i}{2}(1-i\vep)t\px^2}$.

\subsection{Local well-posedness of regularized equations}
We show the local well-posedness for \eqref{ep-nls} by the contraction mapping argument in $X^s(\br)$
\begin{prop}\label{wp of ep-nls}
Let $\lambda, \mu\in\bc$ and $\vep\in(0, 1)$. Then, the Cauchy problem \eqref{ep-nls} is locally well-posed in $X^s(\br)$ for $s>\frac{3}{2}$. 
Also, if $u\in C([0, T]; X^s(\br))$ is a solution to \eqref{ep-nls}, then we have
\begin{equation}\label{ep-sol}
u(t) = \phi + \fr i 2(1-i\vep) \px^2\left(\int_0^tu(t')\,dt'\right) -i\px\left(\int_0^t \left(\lambda u^2 + \mu \abs{u}^2\right)(t')\,dt'\right)
\end{equation}
in $H^{s-2}(\br)$.
\end{prop}
%Also, for $t\in\br$ let $U(t)=e^{\fr{i}{2}t\px^2}$.
We prepare some fundamental properties of $U_{\vep}(t)$.
We omit the proof of the following lemma since the proof is easy.
\begin{lem}\label{schauder}
Let $\vep>0$. For $0\le s_1\le s_2$ and $t>0$, we have
\[
\nrm{U_{\vep}(t)f}_{\dot H^{s_2}}
\lesssim (\vep t)^{\fr{s_1 - s_2}{2}}\nrm{f}_{\dot H^{s_1}},
\]
where the implicit constant depends only on $s_2 - s_1$.
In particular, we obatin
\[
\nrm{U_{\vep}(t)f}_{H^{s_2}}
\lesssim \left(1+\left(\vep t\right)^{\fr{s_1 - s_2}{2}}\right)\nrm{f}_{H^{s_1}}.
\]
\end{lem}
The second lemma is the estimate for free solutions in $X^s(\br$). 
In the following, we define $U_0(t)=U(t)=e^{\fr{i}{2}t\px^2}$ for $t\in\br$.
We note that the following two lemmas admit the case $\vep=0$.
\begin{lem}\label{est of free sol}
Let $0\le\vep \le 1$ and $s> -\fr{1}{2}$.
If $\phi\in X^s(\br)$, then we have $U_\vep(t)\phi\in X^s(\br)$ for $t>0$.
Also, it holds that
\begin{align*}
\sup_{x\in\br}\abs{\int_\minfty^xU_\vep(t)\phi(y)\, dy}
\lesssim \sup_{x\in\br}\abs{\int_\minfty^x\phi(y)\, dy} + (1+t)\nrm{\phi}_{H^{s}},
\end{align*}
where the implicit constant depends only on $s$.
In particular, we obtain
\[
\nrm{U_{\vep}(t)\phi}_{X^s}
\lesssim (1+t)\nrm{\phi}_{X^s}.
\]
Moreover, $U_\vep(t)\phi\in C([0,\infty);X^s(\br))$.
When $\vep=0$,  we have $U(t)\phi\in C(\br ;X^s(\br))$ and
\[
\nrm{U(t)\phi}_{X^s}\lesssim (1+\abs{t})\nrm{\phi}_{X^s}
\]
for $t\in\br$.
\end{lem}
\begin{proof}
First, by decomposing $e^{-\fr{1}{2}(i+\vep)t\eta^2}=(e^{-\fr{1}{2}(i+\vep)t\eta^2}-1) +1$, we obtain
\begin{align*}
&\lim_{a\to\minfty}\int_\br\fr{e^{ix\eta} - e^{ia\eta}}{i\eta}e^{-\fr{1}{2}(i+\vep)t\eta^2}\hat\phi(\eta)\,d\eta \\
&=\lim_{a\to\minfty}\int_\br\fr{e^{ix\eta} - e^{ia\eta}}{i\eta}\left(e^{-\fr{1}{2}(i+\vep)t\eta^2} - 1\right)\hat\phi(\eta)\,d\eta
+ \int_\minfty^x\phi(y)\, dy.
\end{align*}
Therefore, the Riemann-Lebesgue lemma yields $U_\vep(t)\phi\in X^s(\br)$ for any $t>0$.
From the H\"older inequality, we obtain 
the desired estimate.

Next, we prove $U_\vep(t)\phi\in C([0,\infty);X^s(\br))$.
Let $t_0, t\in [0, \infty)$. 
Then, we have from \eqref{def of primitive} that
\begin{align*}
&\sup_{x\in\br}\abs{\int_\minfty^x\left(U_{\vep}(t)-U_{\vep}(t_0)\right)\phi(y)\,dy} \\
&=\sup_{x\in\br}\abs{\lim_{a\to\minfty} \int_\br\fr{e^{ix\eta} - e^{ia\eta}}{i\eta}(e^{-\fr{1}{2}(i+\vep)t\eta^2}- e^{-\fr{1}{2}(i+\vep)t_0\eta^2}) \hat\phi(\eta)\,d\eta} \\
&\le \int_\br\fr{1}{\abs{\eta}}\abs{e^{-\fr{1}{2}(i+\vep)t\eta^2} - e^{-\fr{1}{2}(i+\vep)t_0\eta^2}}\abs{\hat \phi(\eta)}\,d\eta.
\end{align*}
From Lebesgue's dominated convergence theorem, we have $U_\vep(t)\phi\in C([0,\infty); X^s(\br))$.
The proof of the case $\vep=0$ is similar.
\begin{comment}
We prove $U_{\vep}(t)\phi\in C([0, \infty); X^s)$. it is sufficient to prove $\sup_{x\in\br}\abs{\int_\minfty^x(U_{\vep}(t)-U_{\vep}(t_0))\phi(y)\, dy}\to0(t\to t_0)$.

By a similar argument as above, we have
\begin{align*}
\abs{\int_\minfty^x(U_{\vep}(t) - U_{\vep}(t_0))\phi(y)\, dy }
&= \abs{\int_\br\fr{e^{-ix\eta}}{\eta}\left(e^{-\fr{1}{2}(i+\vep)t\eta^2} - e^{-\fr{1}{2}(i+\vep)t_0\eta^2}\right)\hat\phi(\eta)\,d\eta} \\
&\le \int_\br\fr{1}{\eta}\abs{e^{-\fr{1}{2}(i+\vep)t\eta^2} - e^{-\fr{1}{2}(i+\vep)t_0\eta^2}}\abs{\hat\phi(\eta)}\,d\eta.
\end{align*}
Therefore $\sup_{x\in\br}\abs{\int_\minfty^x(U_{\vep}(t)-U_{\vep}(t_0))\phi(y)\, dy}\to0\ (t\to t_0)$ holds by Lebesgue's dominated convergence theorem.
\end{comment}
\end{proof}
The third lemma is the estimate for the Duhamel term.
\begin{lem}\label{est of Duhamel}
Let $0\le\vep \le 1$ and $s> \fr{1}{2}$.
If $u\in C([0, T]; H^s(\br))$, then we have 
\[
\int_\minfty^x\left(\int_0^t U_{\vep}(t-t') \partial_y\left(\lambda u^2 +\mu\abs{u}^2\right)(t')\,dt'\right)(y)\, dy \in C([0, T]; L_x^\infty(\br))
\]
and
\[
\nrm{\int_\minfty^x\left(\int_0^t U_{\vep}(t-t') \partial_y\left(\lambda u^2 +\mu\abs{u}^2\right)(t')\,dt'\right)(y)\, dy}\txx{\infty}
\lesssim T^{\fr{1}{2}}\nrm{u}\tx{\infty}{2}^2,
\]
where the implicit constant is universal.
\end{lem}

\begin{proof}
From $s>\frac{1}{2}$ and $\int_0^t U_\vep(t-t') (\lambda u^2 +\mu\abs{u}^2)(t')\, dt'\in C([0, T]; H^s(\br))$, 
we obtain the first property by \eqref{primitive of derivative}. 
We consider the estimate for the $L_{T, x}^\infty$-norm.
It suffices to show the case $\vep=0$. From the dispersive estimate, we have
\begin{align*}
&\abs{\int_\minfty^x\left(\int_0^t U(t-t') \partial_y\left(\lambda u^2 +\mu\abs{u}^2\right)(t')\,dt'\right)(y)\, dy} \\
&= \abs{\int_0^t \left(U(t-t') \left(\lambda u^2 +\mu\abs{u}^2\right)(t')\right)(x)\,dt'} \\
&\lesssim \int_0^t(t-t')^{-\fr{1}{2}}\nrm{u^2(t')}_{L^1}\, dt' \\
&\lesssim T^{\fr{1}{2}}\nrm{u}\tx{\infty}{2}^2
\end{align*}
for all $t\in [0, T]$ and $x\in \br$. 
Thus, we obtain the desired estimate.
\end{proof}

\begin{proof}[Proof of Proposition \ref{wp of ep-nls}]
For $\phi\in X^s(\br)$ and $u\in \ctx{s}$, we define
\[
\Phi(u)(t) := U_{\vep}(t)\phi -i\int_0^t U_{\vep}(t-t')\px\left(\lambda u^2 + \mu \abs{u}^2\right)(t')\,dt'.
\]
From Lemmas \ref{est of free sol} and \ref{est of Duhamel}, $\Phi$ is a map on $\ctx{s}$.
We prove that there exists $K=K(s)>0$ such that if $\nrm{\phi}_{X^s}\le R$, then $\Phi$ is a contraction map on $B_{KR, T}^s:= \{f\in C\left([0, T]; X^s(\br)\right)\mid\nrm{f}_{L_T^{\infty}X^{s}}\le KR\}$ for some $T\in (0 ,1)$.

From Lemma \ref{schauder}, we have
\begin{equation}\label{pr wp of ep-nls1}
\begin{aligned}
&\nrm{\int_0^t U_{\vep}(t-t')\px\left(\lambda u^2 + \mu \abs{u}^2\right)(t')\,dt'}_{H^s} \\
&\lesssim\int_0^t\left(1+\left(\vep (t-t')\right)^{-\fr{1}{2}}\right)\nrm{\lambda u^2(t') + \mu \abs{u}^2(t')}_{H^s}\,dt' \\
&\lesssim \left(T + \vep^{-\fr{1}{2}} T^{\fr{1}{2}}\right)\nrm{u}\thx{\infty}{s}^2
\end{aligned}
\end{equation}
for all $t\in[0, T]$ .
This estimate and Lemmas \ref{est of free sol} and \ref{est of Duhamel} yield that there exists $\tilde K = \tilde K (s)>0$ such that
\[
\nrm{\Phi(u)}\txs
\le \tilde K\left( \nrm{\phi}_{X^s} + T^{\fr1 2}\vep^{-\fr{1}{2}}\nrm{u}\thx{\infty}{s}^2\right).
\]
Taking $K=\tilde 2K$ and $T\le \frac{\vep}{K^4 R^2}$, $\Phi$ is a map on $B_{KR, T}$.
Also, similar calculations as Lemma \ref{est of Duhamel} and \eqref{pr wp of ep-nls1} yield
\[
\nrm{\Phi(u)-\Phi(u)}\txs
\le \tilde K T^{\fr1 2}\vep^{-\fr{1}{2}}\left(\nrm{u}\thx{\infty}{s} + \nrm{v}\thx{\infty}{s}\right)\nrm{u-v}\thx{\infty}{s}.
\]
Taking $T$ as $T\le \frac{\vep}{4 K^4 R^2}$, $\Phi$ is a contraction map on $B_{KR, T}$.
Thus, there exists a solution to \eqref{ep-nls}.

If $u, v\in C([0, T]; X^s(\br))$ are solutions to \eqref{ep-nls}, then we have
\[
\nrm{u-v}_{L_{T'}^\infty X^s}
\le \tilde K {T'}^{\fr1 2}\vep^{-\fr{1}{2}}\left(\nrm{u}\thx{\infty}{s} + \nrm{v}\thx{\infty}{s}\right)\nrm{u-v}_{L_{T'}^\infty X^s}.
\]
for any $T'\in (0, T]$.
Therefore, taking $T'$ as $T'=\vep(2\tilde K(\nrm{u}\thx{\infty}{s}+\nrm{v}\thx{\infty}{s}))^{-2}$, we obtain $u=v$ in $C([0, T']; X^s(\br))$.
Then, the same argument yields $u(\cdot +T')=v(\cdot +T')$ in $C([0, T']; X^s(\br))$, especially, $u=v$ in $C([0, 2T']; X^s(\br))$.
Repeating this argument, we obatin $u=v$ in $C([0, T]; X^s(\br))$.

We show the continuity of the flow map.
Let $\phi, \phi_n\in X^s(\br)$ and $\phi_n\to\phi$ in $X^s(\br)$.
Then, it holds for any large $n$ that $\nrm{\phi_n}_{X^s}\le 2\nrm{\phi}_{X^s}=:\tilde R$.
Thus, taking $T>0$ as $T\le \vep(2K^2 \tilde R)^{-2}$, we may assume $u, u_n\in B_{K\tilde R, T}$ is a solution to \eqref{nls} with initial data $\phi, \phi_n$, respectively. 
Then, we have
\begin{align*}
\nrm{u-u_n}\txs
&\le \tilde K\left( \nrm{\phi-\phi_n}_{X^s} + T^{\fr1 2}\vep^{-\fr{1}{2}}\left(\nrm{u}\thx{\infty}{s} + \nrm{v}\thx{\infty}{s}\right)\nrm{u-v}\thx{\infty}{s}\right) \\
&\le \tilde K\left( \nrm{\phi-\phi_n}_{X^s} + 2T^{\fr1 2}\vep^{-\fr{1}{2}}K\tilde R\nrm{u-v}\thx{\infty}{s}\right).
\end{align*}
From $T\le \vep(2K^2 \tilde R)^{-2}$, we obatin
\[
\nrm{u-u_n}\txs
\le K\nrm{\phi-\phi_n}_{X^s}.
\]
Thus, the flow map is Lipschitz continuous.

Next, let $u\in C([0, T]; X^s(\br))$ be a solution to \eqref{ep-nls}.
To prove that $u$ satisfies \eqref{ep-sol}, we have to show
\begin{align*}
&\int_0^t U_\vep(t-t')\px(\lambda u^2 + \mu |u|^2)(t')\,dt' \\
&= \frac{i}{2}(1-i\vep)\px^2\int_0^t\left(\int_0^{t'} U_\vep(t'-\tau)\px(\lambda u^2 + \mu |u|^2)(\tau)\,d\tau\right)\,dt' + \int_0^t \px(\lambda u^2 + \mu |u|^2)(t')\,dt'
\end{align*}
in $H^{s-2}(\br)$.
From $u\in C([0, T]; H^s(\br))$, Lebesgue's dominated convergence theorem yields
\[
\partial_\tau U_\vep(\tau-t')\px(\lambda u^2 + \mu |u|^2)(t') = \frac{i}{2}(1-i\vep)\px^2U_\vep(\tau-t')\px(\lambda u^2 + \mu |u|^2)(t')
\]
in $H^{s-2}(\br)$ for $t'<\tau \le t$.
Since $u\in C([0, T]; H^s(\br))$, Lemma \ref{schauder} yields
\begin{equation}\label{pr wp of ep-nls2}
\|\px^2U_\vep(\tau-t')\px(\lambda u^2 + \mu |u|^2)(t')\|_{L_\tau^1((t', t]; H_x^{s-2})}
\le (\vep(t-t'))^{-\frac{1}{2}}\nrm{u}_{L_T^\infty H_x^s}^2.
\end{equation}
Thus, we have
\begin{align*}
&U_\vep(t-t')\px(\lambda u^2 + \mu |u|^2)(t') \\
&=\int_{t'}^t\frac{i}{2}(1-i\vep)\px^2U_\vep(\tau-t')\px(\lambda u^2 + \mu |u|^2)(t')\,d\tau + \px(\lambda u^2 + \mu |u|^2)(t')
\end{align*}
in $H^{s-2}(\br)$.
Integrating on $t'\in [0, t)$, we obtain
\begin{equation}\label{pr wp of ep-nls3}
\begin{aligned}
&\int_0^tU_\vep(t-t')\px(\lambda u^2 + \mu |u|^2)(t')\,dt' \\
&=\frac{i}{2}(1-i\vep)\px^2\int_0^t\int_0^\tau U_\vep(\tau-t')\px(\lambda u^2 + \mu |u|^2)(t')\,dt'\,d\tau + \int_0^t\px(\lambda u^2 + \mu |u|^2)(t')\,dt'.
\end{aligned}
\end{equation}
Here, we obtain from \eqref{pr wp of ep-nls2} that $\|\px^2U_\vep(\tau-t')\px(\lambda u^2 + \mu |u|^2)(t')\|_{L_\tau^1((t', t]; H_x^{s-2})}\in L_{t'}^1([0, t))$, and thus we changed the order of the integration.
Swapping $t'$ and $\tau$ in the first term of the right-hand side of \eqref{pr wp of ep-nls3}, we have the desired equality.
\end{proof}

\subsection{Justification for the gauge transformation}
In this subsection, we justify the gauge transformation.
Throughout this paper,  we denote
$\cl_\vep = i\pt + \fr{1}{2}(1-i\vep)\px^2$ for $0\le \vep\le 1$.
Also, for the solution $u$ to \eqref{ep-nls} or \eqref{nls}, we denote 
\begin{equation}\label{def of gauge}
\Lambda(t, x):=2\lambda\int_\minfty^xu(t, y)\,dy + \mu\int_\minfty^x\bar u(t, y)\, dy.
\end{equation}

Noting that $\cl_\vep \bar u = -\overline{\cl_\vep u} +\px^2 \bar u$, we formally obtain
\begin{align}\label{ep-gauge trans}
&e^{\Lambda}\cl_\vep\left(e^{-\Lambda}u\right) \ntg\\
&=\cl_\vep u +\left(-\cl_\vep\Lambda+\fr{1}{2}\left(1-i\vep\right)\left(\px\Lambda\right)^2\right)u
-\left(1-i\vep\right)(\px\Lambda)\px u \ntg\\
&=\px \left(\lambda u^2 + \mu\abs{u}^2\right) \ntg\\
&\quad +\left(-2\lambda\int_\minfty^x\cl_\vep u\,dy +\mu\int_\minfty^x\overline{\cl_\vep u}\,dy -\mu\px\bar u +\fr{1}{2}(1-i\vep)\left(2\lambda u+\mu\bar u\right)^2\right)u \ntg\\
&\quad -\left(1-i\vep\right)\left(2\lambda u+\mu \bar u\right)\px u \ntg\\
&=\left(-2\lambda\left(\lambda u^2 + \mu \abs{u}^2\right)+ \mu\left(\bar\lambda \bar u^2 + \bar\mu \abs{u}^2\right)
+\fr{1}{2}(1-i\vep)\left(2\lambda u+\mu\bar u\right)^2\right)u \ntg\\
&\quad +i\vep\left(2\lambda u +\mu\bar u\right)\px u \ntg\\
&=-2i\lambda^2\vep u^3 + \left(\abs\mu^2 - 2i\lambda\mu\vep\right)\abs u^2 u + (\mu\bar\lambda + \fr{1}{2}\left(1-i\vep\right)\mu^2)\abs u^2 \bar u +i\vep\left(2\lambda u + \mu \bar u\right)\px u.
\end{align}
In particular, we have
\begin{equation}\label{gauge trans}
e^{\Lambda}\cl\left(e^{-\Lambda}u\right)
= \abs{u}^2\left(\abs{\mu}^2u + \fr{1}{2}\mu(\mu+\bar \lambda)\bar u\right)
\end{equation}
if $\vep=0$.
We denote
\begin{gather}\label{N3}
N^{(3)}_\vep(u):= -2i\lambda^2\vep u^3 + \left(\abs\mu^2 - 2i\lambda\mu\vep\right)\abs u^2 u + \left(\mu\bar\lambda + \fr{1}{2}\left(1-i\vep\right)\mu^2\right)\abs u^2 \bar u , \\
\label{N2}
N^{(2)}(u):= i\left(2\lambda u + \mu \bar u\right)\px u.
\end{gather}

In order to obtain some estimates to prove the well-posedness, we need to justify the calculation \eqref{ep-gauge trans} and \eqref{gauge trans}.
If $s> \fr{1}{2}$, we can justify these calculation in the Sobolev space $H^{s-2}(\br)$.

\begin{prop}\label{prop gauge}
Let $s> \fr{1}{2}$, $0< \vep\le 1$, and $u\in \ctx{s}$ satisfy \eqref{ep-sol}. 
Then, we have \eqref{ep-gauge trans} in $H^{s-2}(\br)$.
When $\vep=0$, the same argument holds by substituting \eqref{ep-sol} and \eqref{ep-gauge trans} for \eqref{nls} and \eqref{gauge trans}, respectively.
\end{prop}
To prove this proposition, we prepare some lemmas.
\begin{lem}\label{diff}
Let $s\ge 0$, $0\le \vep\le 1$, and $u\in \ctx{s}$ satisfy \eqref{ep-sol}. Then, for $t\in[0, T]$, we have
\begin{align*}
&\int_\minfty^x u(t, y)\,dy \\
&= \int_\minfty^x \phi(y)\, dy + \fr{i}{2}(1-i\vep)\px\left(\int_0^tu(t')\, dt'\right)(x) -i \left(\int_0^t\left(\lambda u^2 + \mu\abs{u}^2\right)(t')\, dt'\right)(x)
\end{align*}
in the sense of the tempered distribution with respect to $x$. 
In particular, for all $t\in[0, T]$, we have $\int_\minfty^x u(t + h, y)\,dy - \int_\minfty^x u(t, y)\,dy\in H_x^{s-1}(\br)$ and
\[
\lim_{h\to 0}\fr{\int_\minfty^x u(t + h, y)\,dy - \int_\minfty^x u(t, y)\,dy}{h} = \fr{i}{2}(1-i\vep)\px u(t,x) -i \left(\lambda u^2 +\mu \abs{u}^2\right)(t, x)
\]
in $H_x^{s-1}(\br)$.
\end{lem}
\begin{proof}
We only prove the case $\vep=0$ since the proof of the case $0< \vep\le 1$ is similar.
Let $\psi$ be a test function. Then, Lebesgue's dominated convergence theorem yields
\begin{align*}
\left\langle\int_\minfty^xu(t, y)\,dy, \psi(x)\right\rangle
&= \lim_{a\to\minfty}\int_{\br}\left(\int_\br\fr{e^{ix\eta}-e^{ia\eta}}{i\eta}\hat u(t, \eta)\, d\eta\right)\psi(x)\,dx \\
&= \lim_{a\to\minfty}\int_\br\fr{\check\psi(\eta) - e^{ia\eta}\check\psi(0)}{i\eta}\left(p_0(\eta)\hat u(t, \eta) + p_{\ge 1}(\eta)\hat u(t, \eta)\right)\,d\eta,
\end{align*}
where $p_{\ge 1}$ is the symbol of $P_{\ge 1}$.
From \eqref{sol}, we have
\begin{align*}
&\lim_{a\to\minfty}\int_\br\fr{\check\psi(\eta) - e^{ia\eta}\check\psi(0)}{i\eta}p_0(\eta)\hat u(t, \eta)\,d\eta \\
&=\lim_{a\to\minfty}\int_\br\fr{\check\psi(\eta) - e^{ia\eta}\check\psi(0)}{i\eta}p_0(\eta) \\
&\quad\quad\quad\quad \times\left(\hat \phi(\eta) +\fr{i}{2}(i\eta)^2\int_0^t\hat u(t', \eta)\,dt' -i(i\eta)\int_0^t\cf\left(\lambda u^2 +\mu\abs{u}^2\right)(t', \eta)\,dt'\right)\,d\eta \\
&= \left\langle\int_\minfty^xP_0\phi(y)\,dy, \psi(x)\right\rangle 
+ \fr{i}{2}\left\langle \px P_0\left(\int_0^tu(t', x)\,dt'\right), \psi(x)\right\rangle \\
&\quad-i \left\langle P_0\left(\int_0^t\left(\lambda u^2+\mu\abs{u}^2\right)(t', x)\,dt'\right), \psi(x)\right\rangle,
\end{align*}
where the last equality is obtained by transforming in the same way as the first equality of this proof, but in the reverse procedure.
The high frequency part is expressed by the Riemann-Lebesgue lemma and \eqref{sol} as follows: 
\begin{align*}
&\lim_{a\to\minfty}\int_\br\fr{\check\psi(\eta) - e^{ia\eta}\check\psi(0)}{i\eta}p_{\ge 1}(\eta)\hat u(t, \eta)\,d\eta \\
&=\lim_{a\to\minfty}\int_\br\fr{\check\psi(\eta)}{i\eta}p_{\ge 1}(\eta)\hat u(t, \eta)\,d\eta \\
&=\lim_{a\to\minfty}\int_\br\fr{\check\psi(\eta)}{i\eta}p_{\ge 1}(\eta) \\
&\quad\quad\quad\quad \times\left(\hat \phi(\eta) +\fr{i}{2}(i\eta)^2\int_0^t\hat u(t', \eta)\,dt' -i(i\eta)\int_0^t\cf\left(\lambda u^2 +\mu\abs{u}^2\right)(t', \eta)\,dt'\right)\,d\eta \\
&= \left\langle\int_\minfty^xP_{\ge 1}\phi(y)\,dy, \psi(x)\right\rangle 
+ \fr{i}{2}\left\langle \px P_{\ge 1}\left(\int_0^tu(t', x)\,dt'\right), \psi(x)\right\rangle \\
&\quad-i \left\langle P_{\ge 1}\left(\int_0^t\left(\lambda u^2+\mu\abs{u}^2\right)(t', x)\,dt'\right), \psi(x)\right\rangle.
\end{align*}
Therefore, we obtain the desired equality.
\end{proof}

The following lemma is useful to consider the product of the function in $L^\infty(\br)\cap \dot H^{s}(\br)$ and Sobolev spaces. 
See \cite[Lemma 5]{HKO2020}.
\begin{lem}\label{seki}
For $s\ge 0$, we have
\begin{equation*}\label{est of sobolev}
\nrm{fg}_{H^s}
\lesssim \nrm{f}_{H^s}\nrm{g}_{L^\infty}+ \nrm{f}_{H^{s-[s]}}\nrm{g}_{\dot{H}^{[s]+1}},
\end{equation*}
where $[s]$ means the largest integer less than or equal to $s$.
In particular, we have
\[
\nrm{fg}_{H^s}
\lesssim \nrm{f}_{H^s}\nrm{g}_{L^\infty\cap \dot H^{[s]+1}}.
\]
\end{lem}

From \eqref{def of gauge} and a direct calculation, we have
\[
\nrm{e^{\pm\Lambda(t)}}_{\dot H^k}
\lesssim \nrm{e^{\pm\Lambda(t)}}\lx{\infty}\nrm{u(t)}_{H^{k-1}}\left(1+\nrm{u(t)}_{L^2\cap H^{k-2}}^{k-1}\right).
\]
Thus, Lemma \ref{seki} yields
\begin{equation}\label{est of exp f}
\nrm{e^{\pm\Lambda(t)} f}_{H^s}
\lesssim  \nrm{e^{\pm\Lambda(t)}}\lx{\infty}\left(\nrm{f}_{H^s} + \nrm{u(t)}_{H^{[s]}}\left(1+\nrm{u(t)}_{L^2\cap H^{[s]-1}}^{[s]}\right)\nrm{f}_{H^{s-[s]}}\right),
\end{equation}
especially, 
\begin{equation}\label{est of exp f2}
\nrm{e^{\pm\Lambda(t)} f}_{H^s}
\lesssim \nrm{e^{\pm\Lambda(t)}}\lx{\infty}\left(1 + \nrm{u(t)}_{H^{[s]}}^{[s]+1}\right)\nrm{f}_{H^s}
\end{equation}
for $s\ge 0$.
\begin{proof}[Proof of Proposition \ref{prop gauge}]
We only prove the case $\vep=0$ since the proof of the case $0<\vep<1$ is similar. 
We note that Lemma \ref{diff} and \eqref{def of gauge} yields  $\Lambda(t+h)-\Lambda(t)\in H^{s-1}(\br)$ and
\[
\lim_{h\to 0}\frac{\Lambda(t+h)-\Lambda(t)}{h}
= 2\lambda\left(\fr{i}{2}\px u -i \left(\lambda u^2 +\mu \abs{u}^2\right)\right)
+ \mu \left(-\fr{i}{2}\px \bar u +i \left(\bar\lambda \bar u^2 +\bar\mu \abs{u}^2\right)\right)
\]
in $H^{s-1}(\br)$.
First, we show that 
\begin{equation}\label{pr prop gauge1}
\lim_{h\to0}\frac{e^{-(\Lambda(t+h)-\Lambda(t))}-1}{h} = -\pt\Lambda(t)\ \mathrm{in}\ H^{s-1}(\br).
\end{equation}
From $e^{-(\Lambda(t+h)-\Lambda(t))}-1=\sum_{k=1}^\infty \frac{(-1)^k}{k!}(\Lambda(t+h)-\Lambda(t))^k$, it suffices to show that for $s>\frac{1}{2}$, there exists $C>0$ such that
\[
\nrm{\left(\Lambda(t+h)-\Lambda(t)\right)^k}_{H^{s-1}}
\le C^k\nrm{\Lambda(t+h)-\Lambda(t)}_{L^\infty\cap\dot H^{[s]+1}}^k\nrm{\Lambda(t+h)-\Lambda(t)}_{H^{s-1}}
\]
for all $k\in\bn$.
We note that from \eqref{def of gauge} and $s> \frac{1}{2}$, we have $\px\Lambda(t)=2\lambda u(t)+\mu\bar u(t)\in H^{s}(\br)$, especially, $\Lambda(t)\in C([0, T]; L^\infty(\br)\cap\dot H^{s+1}(\br))$.
From Lemma \ref{seki}, it is easy to show the case $s\ge 1$. Thus, we prove the case $\fr{1}{2}< s<1$.
For any $g\in H^{-(s-1)}(\br)$, it holds from Lemma \ref{seki} that
\[
\nrm{\left(\Lambda(t+h)-\Lambda(t)\right)g}_{H^{-(s-1)}}
\le C\nrm{\Lambda(t+h)-\Lambda(t)}_{L^\infty\cap\dot H^{1}}\nrm{g}_{H^{-(s-1)}}.
\]
Thus, the duality argument yields
\[
\nrm{\left(\Lambda(t+h)-\Lambda(t)\right)^2}_{H^{s-1}}
\le C\nrm{\Lambda(t+h)-\Lambda(t)}_{L^\infty\cap\dot H^{1}}\nrm{\Lambda(t+h)-\Lambda(t)}_{H^{s-1}}.
\]
We can repeat this argument when $k\ge 3$.
Hence, we obtain \eqref{pr prop gauge1}.

Next, we prove
\begin{equation}\label{pr prop gauge2}
\nrm{e^{-\Lambda(t)}f}_{H^{s-2}}
\lesssim C(s, \nrm{u(t)}_{X^s})\nrm{f}_{H^{s-2}}
\end{equation}
for all $f\in H^{s-2}(\br)$.
It is easy when $s\ge 2$. 
Therefore, we only prove the case $\fr{1}{2}< s< 2$. 

When $1<s<2$, \eqref{est of exp f2} yields
\begin{align*}
\nrm{e^{-\Lambda(t)}g}_{H^{-(s-2)}}
&\lesssim \nrm{e^{-\Lambda(t)}}\lx{\infty}\left(1 + \nrm{u(t)}_{H^{-(s-2)}}^{[-(s-2)]+1}\right)\nrm{g}_{H^{-(s-2)}} \\
&\lesssim \nrm{e^{-\Lambda(t)}}\lx{\infty}\left(1 + \nrm{u(t)}_{H^{s}}^{[3-s]}\right)\nrm{g}_{H^{-(s-2)}}
\end{align*}
for any $g\in H^{-(s-2)}(\br)$.
Therefore, the duality argument yields \eqref{pr prop gauge2}.
We consider the case $\frac{1}{2}< s\le 1$.
Then, we have
\[
\nrm{e^{-\Lambda(t)}g}_{H^{-(s-2)}}
\sim \nrm{e^{-\Lambda(t)}g}_{L^2} + \nrm{e^{-\Lambda(t)}\px g}_{H^{-(s-1)}}
+ \nrm{e^{-\Lambda}\px \Lambda(t) \cdot g}_{H^{-(s-1)}}.
\]
The inequality \eqref{est of exp f2} and $s>\fr{1}{2}$ yields
\begin{align*}
\nrm{e^{-\Lambda(t)}\px \Lambda(t) \cdot g}_{H^{-(s-1)}}
&\lesssim \nrm{e^{-\Lambda(t)}}\lx{\infty}\left(1 + \nrm{u(t)}_{H^{-(s-1)}}^{[2-s]}\right)\nrm{\px \Lambda(t) \cdot g}_{H^{-(s-1)}} \\
&\lesssim \nrm{e^{-\Lambda(t)}}\lx{\infty}\left(1 + \nrm{u(t)}_{H^{s}}^{[2-s]}\right)\nrm{\px \Lambda(t) \cdot g}_{H^{-(s-1)}} \\
&\lesssim \nrm{e^{-\Lambda(t)}}\lx{\infty}\left(1 + \nrm{u(t)}_{H^{s}}^{[2-s]}\right)\nrm{u(t)}_{H^{s}}\nrm{g}_{H^{-(s-2)}},
\end{align*}
where we used the bilinear estimate $\|FG\|_{H^{-(s-1)}}\lesssim \|F\|_{H^s}\|G\|_{H^{-(s-2)}}$ for $s> \frac{1}{2}$.
We also obtain the estimate for $e^{-\Lambda(t)}\px g$ in the same way as $e^{-\Lambda}\px \Lambda(t) \cdot g$.
Thus, we have \eqref{pr prop gauge2}.

From \eqref{pr prop gauge1}, \eqref{pr prop gauge2}, and the equality
\[
\frac{e^{-\Lambda(t+h)}u(t+h)-e^{-\Lambda(t)}u(t)}{h}
= \frac{e^{-(\Lambda(t+h)-\Lambda(t))}-1}{h}\cdot e^{-\Lambda(t)}u(t+h) -e^{-\Lambda(t)}\cdot\frac{u(t+h)-u(t)}{h},
\]
we obtain
\[
\pt \left(e^{-\Lambda}u\right)(t)
= -\pt\Lambda(t)\cdot e^{-\Lambda(t)}u(t) +e^{-\Lambda(t)}\pt u(t)
\ \mathrm{in}\ H^{s-2}(\br).
\]
This completes the proof.
\end{proof}

\subsection{Uniform a priori estimate for solutions of regularized problems}
In this subsection, we obtain the a priori estimate. 
The main purpose of this section is to prove the existence of a solution to \eqref{nls} in a smooth space.
If we find a solution to \eqref{nls} in a smooth space, we can justify the transformation as \eqref{gauge trans}.
Hence, it suffeces to consider \eqref{ep-nls} for $s>\fr{3}{2}$.

\begin{prop}\label{smooth apriori}
Let $s>\fr{3}{2}$.
Then, there exists $C>0$ as follows.
For any $\phi\in X^s(\br)$, there exists $T^* = T^*(s, \nrm{\phi}_{X^s})\in(0, 1)$ such that for any $\vep\in(0, 1)$, \eqref{ep-nls} has a solution $u\in C([0, T^*]; X^s(\br))$ with
\[
\nrm{u}_{L_{T^*}^\infty X^s}
\le Ce^{C\nrm{\phi}_{X^0}}\left(\nrm{\phi}_{X^s} + \nrm{\phi}_{X^s}^{3[s]+2}\right).
\]
\end{prop}

To prove Proposition \ref{smooth apriori}, we first consider the estimate for  $e^{-\Lambda}u$.
\begin{lem}{\label{smooth apriori step1}}
Let $s>\fr{3}{2}$.
Then, there exists $C_1>0$ such that
if $\vep\in(0, 1)$,  $T\in(0, 1)$, and $u\in C\left([0, T]; X^s(\br)\right)$ satisfies \eqref{ep-sol}, then we have
\begin{equation}\label{ineq smooth apriori step1}
\begin{aligned}
\nrm{e^{-\Lambda}u}\thx{\infty}{s}
&\le C_1\nrm{e^{-\Lambda(0)}}_{L^{ \infty}}\left(\nrm{\phi}_{H^{s-1}}+ \nrm{\phi}_{H^{s-1}}^{[s]+1}\right)\nrm{\phi}_{H^s} \\
&\quad+C_1T^{\frac{1}{2}}\nrm{e^{-\Lambda}}\txx{\infty}\left(\nrm{u}_{L_T^\infty H_x^{s-1}}+ \nrm{u}_{L_T^\infty H_x^{s-1}}^{[s]+3}\right)\nrm{u}\thx{\infty}{s}.
\end{aligned}
\end{equation}
In particular, we have
\begin{align*}
\nrm{e^{-\Lambda}u}\thx{\infty}{s}
&\le C_1\nrm{e^{-\Lambda(0)}}_{L^{ \infty}}\left(\nrm{\phi}_{H^s}^2 + \nrm{\phi}_{H^s}^{[s]+2}\right) \\
&\quad +C_1T^{\frac{1}{2}}\nrm{e^{-\Lambda}}\txx{\infty}\left(\nrm{u}\thx{\infty}{s}^2 + \nrm{u}\thx{\infty}{s}^{[s]+4}\right).
\end{align*}
\end{lem}

\begin{proof}
From Proposition \ref{prop gauge}, we obtain
\[
e^{-\Lambda(t)}u(t) 
= U_{\vep}(t)\left(e^{-\Lambda(0)}\phi\right) -i\int_0^tU_{\vep}(t-t')\left(e^{-\Lambda}\left(N^{(3)}_\vep(u)+ \vep N^{(2)}(u)\right)\right)(t')\,dt'
\]
in $H^s(\br)$, where $N_{\vep}^{(3)}$ and $N^{(2)}$ are defined by \eqref{N3} and \eqref{N2}, respectively.
The inequality \eqref{est of exp f} yields
\begin{align*}
\nrm{U_{\vep}(t)\left(e^{-\Lambda(0)} \phi\right)}_{H^s} 
&\lesssim \nrm{e^{-\Lambda(0)}}_{L^{ \infty}}\left(\nrm{\phi}_{H^s} + \nrm{\phi}_{H^{[s]}}\left(1+\nrm{\phi}_{L^2\cap H^{[s]-1}}^{[s]}\right)\nrm{\phi}_{H^{s-[s]}}\right) \\
&\lesssim \nrm{e^{-\Lambda(0)}}_{L^{ \infty}}\left(\nrm{\phi}_{H^{s-1}}+ \nrm{\phi}_{H^{s-1}}^{[s]+1}\right)\nrm{\phi}_{H^s}.
\end{align*}

Next, we consider the term $N^{(3)}_\vep(u)$. It suffices to consider the term $e^{-\Lambda}u^3$. 
This term is estimated by \eqref{est of exp f} as follows:

\begin{align*}
&\nrm{\int_0^tU_{\vep}(t-t')e^{-\Lambda(t')}u^3(t')\,dt'}\thx{\infty}{s} \\
&\lesssim T\nrm{e^{-\Lambda}u^3}\thx{\infty}{s} \\
&\lesssim T\nrm{e^{-\Lambda}}\txx{\infty}\left(\nrm{u^3}_{L_T^\infty H_x^s} + \nrm{u}_{L_T^\infty H_x^{[s]}}\left(1+\nrm{u}_{L_T^\infty L_x^2\cap L_T^\infty H_x^{[s]-1}}^{[s]}\right)\nrm{u^3}_{L_T^\infty H_x^{s-[s]}}\right) \\
&\lesssim T\nrm{e^{-\Lambda}}\txx{\infty}\left(\nrm{u}_{ L_T^\infty H_x^{s-1}}^2 + \nrm{u}_{L_T^\infty H_x^{s-1}}^{[s] + 3}\right)\nrm{u}\thx{\infty}{s}.
\end{align*}
Here, we used $s-1>\fr{1}{2}$ and $\nrm{u^3}_{L_T^\infty H_x^s}\lesssim \nrm{u}_{L_T^\infty H_x^{s-1}} ^2\nrm{u}_{L_T^\infty H_x^s}$.

On the term $N^{(2)}(u)$, it suffices to consider the term $e^{-\Lambda}\bar u\px u$. Lemma \ref{schauder} and \eqref{est of exp f} yield
\begin{align*}
&\vep\nrm{\int_0^tU_{\vep}(t-t')e^{-\Lambda(t')}\bar u\px u(t')\,dt'}\thx{\infty}{s} \\
&\lesssim \vep\left(T+ \vep^{-\fr{1}{2}}T^{\fr{1}{2}}\right)\nrm{e^{-\Lambda}\bar u\px u}\thx{\infty}{s-1} \\
&\lesssim T^{\frac{1}{2}}\nrm{e^{-\Lambda}}\txx{\infty} \left(\nrm{\bar u\px u}\thx{\infty}{s-1}
+  \nrm{u}_{L_T^\infty H_x^{[s]-1}}\left(1+\nrm{u}_{L_T^\infty L_x^2\cap L_T^\infty H_x^{[s]-2}}^{[s]-1}\right)\nrm{\bar u\px u}_{L_T^\infty H_x^{s-[s]}}\right) \\
&\lesssim T^{\frac{1}{2}}\nrm{e^{-\Lambda}}\txx{\infty}\left(\nrm{u}_{ L_T^\infty H_x^{s-1}}+ \nrm{u}_{ L_T^\infty H_x^{s-1}}^{[s]+1}\right)\nrm{u}\thx{\infty}{s}. 
\end{align*}
Therefore, we obtain the desired estimate.
\end{proof}

\begin{comment}
From Lemma \ref{strest} and \eqref{ep-nls}, we immediately obtain the following lemma. We omit the proof.
\begin{lem}\label{smooth apriori step2}
Let $\vep\in(0, 1)$, $s>\fr{3}{2}$, and $u\in \ctx{s}$ be a solution to \eqref{ep-nls}.
Then, there exists $C=C(s)>0$ such that
\begin{align}{\label{ineq smooth apriori step2}}
\nrm{u}\thx{\infty}{s-1}\le \nrm{\phi}_{H^s} + CT\nrm{u}\thx{\infty}{s}^2.
\end{align}
\end{lem}
\end{comment}

We can obtain the estimate for the $H^s$-norm of the solution from Lemma \ref{smooth apriori step1}.

\begin{lem}{\label{smooth apriori step3}}
Let $s>\fr{3}{2}$.
Then, there exists $C_2>0$ such that
if $\vep\in(0, 1)$,  $T\in(0, 1)$, and $u\in C\left([0, T]; X^s(\br)\right)$ is a solution to \eqref{ep-sol}, then we have
\begin{align*}
\nrm{e^{-\Lambda}}\txx{\infty}^{-1}\nrm{u}\thx{\infty}{s}
&\le C_2\nrm{e^{-\Lambda(0)}}_{L^{ \infty}}\left(\nrm{\phi}_{H^s}+\nrm{\phi}_{H^s}^{3[s]+2}\right) \\
&\quad + C_2T^{\frac{1}{2}}\nrm{e^{-\Lambda}}\txx{\infty}\left(\nrm{u}_{L_T^{ \infty }H_x^{ s }}^2 + \nrm{u}_{L_T^{ \infty }H_x^{ s }}^{4[s]+4}\right).
\end{align*}
\end{lem}

\begin{proof}
Since $u$ satisfies \eqref{ep-sol}, Lemma \ref{strest} yields
\begin{equation}\label{pr apriori step3 ineq1}
\nrm{u}_{L_T^\infty H_x^{s-1}}
\lesssim \nrm{\phi}_{H^{s-1}} + T\nrm{u}_{L_T^\infty H_x^{s}}^2.
\end{equation}
We have
\begin{equation}\label{decomp}
\begin{aligned}
\nrm{ u }_{L_T^{ \infty }H_x^{ s }}
&= \nrm{e^{\Lambda}e^{-\Lambda}u }_{L_T^{ \infty }H_x^{ s }} \\
&\lesssim \nrm{u}_{L_T^{ \infty }L_x^{ 2 }} + \nrm{u\px\Lambda}_{L_T^{ \infty }H_x^{ s-1 }} + \nrm{e^{\Lambda}\px\left(e^{-\Lambda}u\right)}_{L_T^{ \infty }H_x^{ s-1 }} \\
&:= \I + \II + \III.
\end{aligned}
\end{equation}
The estimate for $\I$ and $\II$ is obtained by \eqref{pr apriori step3 ineq1} and $s-1>\fr{1}{2}$. Also, $\III$ is estimated by using \eqref{est of exp f2}, \eqref{pr apriori step3 ineq1}, and Lemma \ref{smooth apriori step1} :
\begin{align*}
&\nrm{e^{\Lambda}}\txx{\infty}^{-1}\times\III \\
&\lesssim \left(1 + \nrm{u}\thx{\infty}{s-1}^{[s]}\right)\nrm{e^{-\Lambda}u}\thx{\infty}{s}  \\
&\lesssim \left(1 + \nrm{\phi}_{H^s}^{[s]} + T^{[s]}\nrm{u}\thx{\infty}{s}^{2[s]}\right)\\
&\quad \times \left(\nrm{e^{-\Lambda(0)}}_{L^{\infty}}\left(\nrm{\phi}_{H^s} + \nrm{\phi}_{H^s}^{2[s]+2}\right)+ T^{\fr{1}{2}} \nrm{e^{-\Lambda}}\txx{\infty}\left(\nrm{u}\thx{\infty}{s}^2 + \nrm{u}\thx{\infty}{s}^{[s]+4}\right)\right) \\
&\lesssim \nrm{e^{-\Lambda(0)}}_{L^{\infty}}\left(\nrm{\phi}_{H^s}+\nrm{\phi}_{H^s}^{3[s]+2}\right) 
+ T^{\fr{1}{2}} \nrm{e^{-\Lambda}}\txx{\infty}\left(\nrm{u}_{L_T^{ \infty }H_x^{ s }}^2 + \nrm{u}_{L_T^{ \infty }H_x^{ s }}^{4[s]+4}\right),
\end{align*}
where we used $\max(4[s]+2, 3[s]+4)\le 4[s]+4$.
Therefore, we obtain the desired estimate.
\end{proof}
If $s>\frac{1}{2}$, then the $L_{T, x}^\infty$-norm of the primitive of the solution to \eqref{nls} and \eqref{ep-nls} is bounded by Lemmas \ref{est of free sol} and \ref{est of Duhamel}. 
\begin{lem}\label{smooth apriori step4}
There exists $C_3>0$ such that
if $s>\fr{1}{2}$, $\vep\in[0, 1)$,  $T\in(0, 1)$, and $u\in C\left([0, T]; X^s(\br)\right)$ satisfies \eqref{ep-sol}, then we have
\begin{equation*}
\nrm{\int_\minfty^xu(t, y)\,dy}\txx{\infty}
\le C_3\left(\nrm{\phi}_{X^0} + T^\fr{1}{2}\nrm{u}_{L_T^{\infty}L_x^2}^2\right).
\end{equation*}
In particular, the same estimate holds by replacing $\int_\minfty^xu(t, y)\,dy$ with $\Lambda$ which is defined by \eqref{def of gauge}.
\end{lem}
\begin{proof}
From $s>\fr{1}{2}$ and $u\in \ctx{s}$, we have $\px(\lambda u^2 + \mu |u|^2)\in\ctx{s-1}$.
Thus, $u$ satisfies \eqref{mild-sol} in $H^{s-1}(\br)$.
Since $s-1>-\frac{1}{2}$, we obtain 
\[
\int_\minfty^xu(t, y)\,dy
=\int_\minfty^xU_\vep(t)\phi(y)\,dy -i\int_\minfty^x\left(\int_0^t U_\vep(t-t') \partial_y\left(\lambda u^2 +\mu\abs{u}^2\right)(t')\,dt'\right)(y)\, dy.
\]
Thus, Lemmas \ref{est of free sol} and \ref{est of Duhamel} yields the desired estimate.
\end{proof}

\begin{proof}[Proof of Proposition \ref{smooth apriori}]
From Proposition \ref{wp of ep-nls}, there exist $T=T(\vep, s, \|\phi\|_{X^s})\in(0, 1)$ and a solution $u\in C([0, T]; X^s(\br))$ to \eqref{ep-nls}.
Let $C_1, C_2, C_3$ be constants taken in Lemmas \ref{smooth apriori step1} 
--\ref{smooth apriori step4} and $C=\max(C_1, C_2, C_3)$.
We take 
\begin{equation}\label{smooth T^*}
M=2Ce^{2C\nrm{\phi}_{X^0}}(\nrm{\phi}_{X^s}+\nrm{\phi}_{X^s}^{3[s]+2}), \quad
T^*=\min\left(\frac{\nrm{\phi}_{X^s}^2}{M^4}, \frac{M^2}{\left(M+M^{4[s]+4}\right)^2}\right).
\end{equation}
We note $M$ and $T^*$ depend only on $s$ and $\nrm{\phi}_{X^s}$.
We define ${T_0^*}=\min(T, T^*)$.
Then, Lemmas \ref{smooth apriori step3} and \ref{smooth apriori step4} yield
\begin{align*}
\nrm{u}_{L_{T_0^*}^\infty X^s}
&\le Ce^{C(\nrm{\phi}_{X^0} + {T_0^*}^{\fr{1}{2}}\nrm{u}_{L_{T_0^*}^{\infty}L_x^2}^2)}\left(\nrm{\phi}_{X^s} + \nrm{\phi}_{H^s}^{3[s]+2} +{T_0^*}^{\fr{1}{2}} \left(\nrm{u}_{L_{T_0^*}^{\infty}H_x^s}^2 + \nrm{u}_{L_{T_0^*}^{\infty}H_x^s}^{4[s]+4}\right)\right).
\end{align*}
Therefore, a continuity argument yields
\[
\nrm{u}_{L_{T_0^*}^\infty X^s}
\le M.
\]
Thus, Proposition \ref{wp of ep-nls} yields that there exist $\delta=\delta(\vep, s, M)\in (0, 1)$ and a solution $u\in C([0, T + \delta]; X^s(\br))$ to \eqref{ep-nls}.
Then, Lemmas \ref{smooth apriori step3} and \ref{smooth apriori step4} yield
\begin{align*}
\nrm{u}_{L_{T_1^*}^\infty X^s}
&\le Ce^{C(\nrm{\phi}_{X^0} + {T_1^*}^{\fr{1}{2}}\nrm{u}_{L_{T_1^*}^{\infty}L_x^2}^2)}\left(\nrm{\phi}_{X^s} + \nrm{\phi}_{H^s}^{3[s]+2} +{T_1^*}^{\fr{1}{2}} \left(\nrm{u}_{L_{T_1^*}^{\infty}H_x^s}^2 + \nrm{u}_{L_{T_1^*}^{\infty}H_x^s}^{4[s]+4}\right)\right),
\end{align*}
where ${T_1^*} = \min(T_\vep + \delta, T^*)$.
Therefore, a continuity argument yields
\[
\nrm{u}_{L_{T_1^*}^\infty X^s}
\le M.
\]
Repeating this argument $k$ times, we obtain a solution $u\in C([0, T_k^*]; X^s(\br))$ with
\[
\nrm{u}_{L_{T_k^*}^\infty X^s}
\le M,
\]
where ${T_k^*} = \min(T_\vep + k\delta, T^*)$.
Taking $k$ sufficiently large, we finish the proof.
\end{proof}
Since we obtain an a priori estimate in $X^s(\br)$, we can obtain an a priori estimate in $X^{s+1}(\br)$ whose time depends only on $s$ and the $X^s$-norm of the initial data.
\begin{prop}\label{smooth apriori2}
Let $s>\fr{3}{2}$.
Then, for any $\phi\in X^{s+1}(\br)$, there exist $C=C(s, \nrm{\phi}_{X^s})>0$ and $T^{**}=T^{**}(s, \nrm{\phi}_{X^s})\in (0, 1)$ such that for any $\vep\in(0, 1)$, \eqref{ep-nls} has a solution $u\in C([0, T^{**}]; X^{s+1}(\br))$ with 
\begin{equation*}
\nrm{u}_{L_{T^{**}}^\infty X^{s+1}}
\le C\left(1+\nrm{\phi}_{H^{s+1}}\right).
\end{equation*}
\end{prop}
To prove Proposition \ref{smooth apriori2}, we use the following lemma.
\begin{lem}\label{smooth apriori3}
Let $s>\fr{3}{2}$, $\phi\in X^{s+1}(\br)$, $\vep\in(0, 1)$, and $u\in C([0, T]; X^{s+1}(\br))$ be a solution to \eqref{ep-nls}.
Then, there exist $\tilde C=\tilde C(s, \nrm{u}_{L_T^\infty X^s})$ which is independent of $\vep$ such that
\begin{equation*}
\nrm{u}_{L_{T}^\infty X^{s+1}}
\le \tilde C\left(1+\nrm{\phi}_{H^{s+1}} + T^{\frac{1}{2}}\nrm{u}_{L_{T}^\infty X^{s+1}}\right).
\end{equation*}
\end{lem}
\begin{proof}
It suffices to consider the $L_{T}^\infty H_x^{s+1}$-norm.
We use the decomposition \eqref{decomp}, where $s$ is replaced with $s+1$.
Terms $\I$ and $\II$ are easily estimated as follows:
\[
\I + \II
\lesssim \nrm{u}_{L_{T}^\infty L_x^2} + \nrm{u}_{L_{T}^\infty H_x^s}^2.
\]

We consider the term $\III$.
The inequality \eqref{est of exp f} yields
\begin{align*}
\III
&\lesssim \nrm{e^{\Lambda}}_{L_{T, x}^\infty} \left(\nrm{\px\left(e^{-\Lambda}u\right)}_{L_{T}^\infty H_x^s}+ \nrm{u}_{L_{T}^\infty H_x^{[s]}}\left(1+\nrm{u}_{L_{T}^\infty H_x^{[s]-1}}^{[s]}\right)\nrm{\px\left(e^{-\Lambda}u\right)}_{L_{T}^\infty H_x^{s-[s]}}\right).
\end{align*}
From \eqref{est of exp f} and Proposition \ref{smooth apriori}, we obtain
\[
\nrm{e^{\Lambda}}_{L_{T, x}^\infty}\nrm{u}_{L_{T}^\infty H_x^{[s]}}\left(1+\nrm{u}_{L_{T}^\infty H_x^{[s]-1}}^{[s]}\right)\nrm{\px\left(e^{-\Lambda}u\right)}_{L_{T}^\infty H_x^{s-[s]}}
\le \tilde C(s, \nrm{u}_{L_{T}^\infty X^s}).
\]
Also, Proposition \ref{smooth apriori} and Lemma \ref{smooth apriori step1} yield that
\begin{align*}
\nrm{e^{-\Lambda}u}_{L_{T}^\infty H_x^{s+1}}
&\lesssim \nrm{e^{-\Lambda(0)}}_{L^{\infty}}\left(\nrm{\phi}_{H^{s}}+ \nrm{\phi}_{H^{s}}^{[s]+2}\right)\nrm{\phi}_{H^{s+1}} \\
&\quad +{T}^{\frac{1}{2}}\nrm{e^{-\Lambda}}_{L_{T, x}^\infty}\left(\nrm{u}_{L_{T}^\infty H_x^{s}}+ \nrm{u}_{L_{T}^\infty H_x^{s}}^{[s]+4}\right)\nrm{u}_{L_{T}^\infty H_x^{s+1}} \\
&\le \tilde C(s, \nrm{u}_{L_{T}^\infty X^s})\left(\nrm{\phi}_{H^{s+1}} + {T}^{\frac{1}{2}}\nrm{u}_{L_{T}^\infty H_x^{s+1}}\right).
\end{align*}
Thus, we obtain the desired estimate. 
\end{proof}
\begin{proof}[Proof of Proposition \ref{smooth apriori2}]
From Proposition \ref{wp of ep-nls}, there exist $T=T(\vep, s, \|\phi\|_{X^{s+1}})\in(0, 1)$ and a solution $u\in C([0, T]; X^{s+1}(\br))$ to \eqref{ep-nls}.
Let $C_1, C_2, C_3$ be constants taken in Lemmas \ref{smooth apriori step1} 
--\ref{smooth apriori step4} and $C=\max(C_1, C_2, C_3)$.
We take $M$ and $T^*$ as \eqref{smooth T^*}.
Also, we define $T^{**}$ and $T_0^{**}$ by ${T^{**}}=\min(T^*, (2\tilde C(s, M))^{-2})$ and $T_0^{**} = \min(T, T^{**})$, where $\tilde C$ is taken in Lemma \ref{smooth apriori3}.
Then, Proposition \ref{smooth apriori} and Lemma \ref{smooth apriori3} yield
\[
\nrm{u}_{L_{T_0^{**}}^\infty X^{s+1}}
\le 2\tilde C(s, M)(1+\nrm{\phi}_{H^{s+1}}).
\]
Thus, Proposition \ref{wp of ep-nls} yields that there exist $\delta=\delta(\vep, s, \nrm{\phi}_{X^{s+1}})\in (0, 1)$ and a solution $u\in C([T, T + \delta]; X^{s+1}(\br))$ to \eqref{ep-nls}.
Then, Proposition \ref{smooth apriori} and Lemma \ref{smooth apriori3} yield
\[
\nrm{u}_{L^\infty ([T, T+\delta];X^{s+1}(\br))}
\le 2\tilde C(s, M)(1+\nrm{\phi}_{H^{s+1}}).
\]
In particular, we have the same estimate when we replace $L^\infty ([T, T+\delta];X^{s+1}(\br))$ with $L_{T_1^{**}}^\infty X^{s+1}$ where ${T_1^{**}} = \min(T + \delta, T^{**})$.
Repeating Propositions \ref{wp of ep-nls} and \ref{smooth apriori} and Lemma \ref{smooth apriori3}, we have
\[
\nrm{u}_{L^\infty ([T+(k-1)\delta, T+k\delta];X^{s+1}(\br))}
\le 2\tilde C(s, M)(1+\nrm{\phi}_{H^{s+1}}).
\]
In particular, we have the same estimate when we replace $L^\infty ([T+(k-1)\delta, T+k\delta];X^{s+1}(\br))$ with $L_{T_k^{**}}^\infty X^{s+1}$ where ${T_k^{**}} = \min(T + k\delta, T^{**})$.
Taking $k$ sufficiently large, we finish the proof.
\end{proof}
\subsection{Estimate for difference}
In this subsection, we consider the estimate for the difference between two solutions.
For $\vep_1, \vep_2\in(0, 1)$ and $\phi_1,\phi_2$,  we define $\tilde\phi$ as
\[
\tilde\phi
=\left\{
\begin{aligned}
&\phi_1 \ (\vep_1\le \vep_2) \\
&\phi_2 \ (\vep_1> \vep_2).
\end{aligned}
\right.
\]
\begin{prop}\label{smooth est of diff}
Let $s>\fr{3}{2}$ and $j=\{1, 2\}$.
\begin{enumerate}[(i)]
\item \it {Let $\phi_j\in X^{s+1}(\br)$ and $\vep_j\in (0, 1)$.
Also, let $u_{\vep_j, \phi_j}\in C([0, T_j^{**}]; X^{s+1}(\br))$ be a solution to \eqref{ep-nls} obtained in Proposition \ref{smooth apriori2}.
Then, there exist $C=C(s, \|\phi_1\|_{X^s}, \|\phi_2\|_{X^s})$ and $T^{**}= T^{**}(s, \nrm{\phi}_{X^s})\in (0, \min(T_1^{**}, T_2^{**})]$ which are independent of $\vep_1, \vep_2$ such that
\[
\nrm{u_{\vep_1, \phi_1} - u_{\vep_2, \phi_2}}_{L_{T^{**}}^\infty X^s} 
\le C\left(\nrm{\phi_1 - \phi_2}_{X^s} + \abs{\vep_1 - \vep_2}^{\fr{1}{2}}(1+\|\tilde\phi\|_{H^{s+1}})\right).
\]}

\item \it {Let $\phi_j\in X^{s}(\br)$ and $\vep_1=\vep_2\in (0, 1)$.
Also, let $u_{\vep_j, \phi_j}\in C([0, T_j^{*}; X^{s}(\br))$ be a solution to \eqref{ep-nls} obtained in Proposition \ref{smooth apriori}.
Then, there exist $C=C(s, \|\phi_1\|_{X^s}, \|\phi_2\|_{X^s})$ and $T^{*}= T^{*}(s, \nrm{\phi}_{X^s})\in (0, \min(T_1^{*}, T_2^{*})]$ which are independent of $\vep_1, \vep_2$ such that
\[
\nrm{u_{\vep_1, \phi_1} - u_{\vep_2, \phi_2}}_{L_{T^{*}}^\infty X^s} 
\le C\nrm{\phi_1 - \phi_2}_{X^s}.
\]}
\end{enumerate}
\end{prop}

In the following, we denote $u_1 = u_{\vep_1, \phi_1}$, $u_2 = u_{\vep_2, \phi_2}$, and
\begin{equation}\label{primitives}
\begin{gathered}
\Lambda_1(t, x) = 2\lambda\int_\minfty^xu_1(t, y)\,dy + \mu\int_\minfty^x\bar u_1(t, y)\,dy, \\
\Lambda_2 (t, x)= 2\lambda\int_\minfty^xu_2(t, y)\,dy + \mu\int_\minfty^x\bar u_2(t, y)\,dy. 
\end{gathered}
\end{equation}
First, we consider the estimate for $\Lambda_1-\Lambda_2$.
\begin{lem}\label{smooth est of diff step1}
We assume the same condition as Proposition \ref{smooth est of diff} (i).
Then, there exist $C=C(s, \|\phi_1\|_{X^s}, \|\phi_2\|_{X^s})>0$ which is independent of $\vep_1, \vep_2$ such that
\begin{align*}
&\nrm{\Lambda_1 - \Lambda_2}_{L_{T^{**}, x}^{ \infty }}  \\
&\le C\left(\nrm{\phi_1 - \phi_2}_{X^0} + \abs{\vep_1 - \vep_2}^{\frac{1}{2}}\|\tilde \phi\|_{H^1}
+{T^{**}}^{\frac{1}{2}}(\nrm{u_1 - u_2}_{L_{T^{**}}^{\infty }L_x^{2 }} + \abs{\vep_1 - \vep_2}^{\fr{1}{2}})\right).
\end{align*}
for all $T^{**}\in (0, \min(T_1^{**}, T_2^{**})]$.
If we assume the same condition as Proposition \ref{smooth est of diff} (ii), then the same estimate holds for all $T^{*}\in (0, \min(T_1^{*}, T_2^{*})]$ by replacing $T^{**}$ with $T^{*}$.
\end{lem}
To prove Lemma \ref{smooth est of diff step1}, we use the following lemma.
\begin{lem}\label{est of diff free sol}
Let $s\in\br$, $\vep_1, \vep_2\in[0, 1)$, $t>0$, and $a\in(0, 2)$.
Then the following estimate holds:
\begin{align*}
\nrm{U_{\vep_1}(t)\phi_1 - U_{\vep_2}(t)\phi_2}_{H^s}
\lesssim \nrm{\phi_1 - \phi_2}_{H^s} +\abs{\vep_1 - \vep_2}^{\fr{a}{2}}t^{\fr{a}{2}}\|\tilde \phi\|_{H^{s+a}},
\end{align*}
where the implicit constant depends only on $a$.
\end{lem}
\begin{proof}
We may assume $\vep_1\ge \vep_2$.
From the triangle inequality, we have
\begin{equation}\label{pr smooth diff free sol}
\nrm{U_{\vep_1}(t)\phi_1 - U_{\vep_2}(t)\phi_2}_{H^s}
\le \nrm{U_{\vep_1}(t)(\phi_1 - \phi_2)}_{H^s} + \nrm{\left(U_{\vep_1}(t) - U_{\vep_2}(t)\right)\phi_2}_{H^s}.
\end{equation}
The first term on the right-hand side of \eqref{pr smooth diff free sol} is bounded by $\nrm{\phi_1-\phi_2}_{H^s}$.
Since we have
\[
\abs{e^{-\frac{i}{2}(1-i\vep_1)t\xi^2} - e^{-\frac{i}{2}(1-i\vep_2)t\xi^2}}
\lesssim \frac{1}{a}\left(\vep_1^{\frac{a}{2}}-\vep_2^{\frac{a}{2}}\right)\left(t\xi^2\right)^{\frac{a}{2}}
\le \frac{1}{a}(\vep_1-\vep_2)^{\frac{a}{2}}\left(t\xi^2\right)^{\frac{a}{2}},
\]
the second term on the right-hand side of \eqref{pr smooth diff free sol} is bounded by $\abs{\vep_1 - \vep_2}^{\fr{a}{2}}t^{\fr{a}{2}}\|\phi_2\|_{H^{s+a}}$.
Thus, we obtain the desired estimate.
\end{proof}

\begin{proof}[Proof of Lemma \ref{smooth est of diff step1}]
In the following, we note that $T^{**}\le 1$.
Since $u_1, u_2$ are solutions to \eqref{ep-nls}, we have
\begin{equation*}%\label{pr smooth est of diff step1}
\begin{aligned}
&\int_\minfty^x   \left(u_1 - u_2\right)(t, y)\,dy \\
&= \int_\minfty^x   U_{\vep_1}(t)\phi_1(y)\,dy - i\int_\minfty^x  \left(\int_0^t   U_{\vep_1}(t-t')\partial_y(\lambda u_1^2 + \mu \abs{u_1}^2)(t')\,dt'\right)(y)dy \\
&\quad- \int_\minfty^x   U_{\vep_2}(t)\phi_2(y)\,dy + i\int_\minfty^x  \left(\int_0^t  U_{\vep_2}(t-t')\partial_y(\lambda u_2^2 + \mu \abs{u_2}^2)(t')\,dt'\right)(y)dy \\
&= \int_\minfty^x  U_{\vep_1}(t)\left(\phi_1- \phi_2\right)(y)\,dy 
 + \int_\minfty^x   \left(U_{\vep_1}(t) -U_{\vep_2}(t)\right)\phi_2(y)\,dy \\
&\quad -i\left(\int_0^t  U_{\vep_1}(t-t')\left(\left(\lambda u_1^2 + \mu \abs{u_1}^2\right) - \left(\lambda u_2^2 + \mu \abs{u_2}^2\right)\right)(t')\,dt'\right)(x) \\
&\quad - i  \left(\int_0^t  \left(U_{\vep_1}(t-t') - U_{\vep_2}(t-t')\right)(\lambda u_2^2 + \mu \abs{u_2}^2)(t')\,dt'\right)(x). 
\end{aligned}
\end{equation*}
From Lemma \ref{est of free sol}, we obtain
\begin{align*}
\nrm{\int_\minfty^x  U_{\vep_1}(t)\left(\phi_1- \phi_2\right)(y)\,dy }_{L_{T^{**}, x}^\infty}
\lesssim  \nrm{\phi_1- \phi_2}_{X^0}.
\end{align*}
Also, a similar argument as Lemma \ref{est of Duhamel} yields
\begin{equation*}%\label{est of diff of Duhamel}
\begin{aligned}
&\nrm{\int_0^t  U_{\vep_1}(t-t')\left(\left(\lambda u_1^2 + \mu \abs{u_1}^2\right) - \left(\lambda u_2^2 + \mu \abs{u_2}^2\right)\right)(t')\,dt'}_{L_{T^{**}, x}^\infty} \\
&\lesssim {T^{**}}^{\fr{1}{2}}\nrm{\left(\lambda u_1^2 + \mu \abs{u_1}^2\right) - \left(\lambda u_2^2 + \mu \abs{u_2}^2\right)}_{L_{T^{**}}^{ \infty }L_x^{ 1 }} \\
&\lesssim {T^{**}}^{\fr{1}{2}}\left(\nrm{u_1}_{L_{T^{**}}^{ \infty }L_x^{ 2 }} + \nrm{u_2}_{L_{T^{**}}^{ \infty }L_x^{ 2 }}\right)\nrm{u_1- u_2}_{L_{T^{**}}^{ \infty }L_x^{ 2 }}.
\end{aligned}
\end{equation*}
A direct calculation as Lemma \ref{est of diff free sol} yields
\begin{equation}\label{ineq for appendix}
\begin{aligned}
&\abs{\int_\minfty^x \left(U_{\vep_1}(t) -U_{\vep_2}(t)\right)\phi_2(y)\,dy} \\
&= \lim_{a\ra\minfty}\abs{\int_\br\fr{e^{ix\eta}-e^{ia\eta}}{\eta}\left(e^{-\fr{1}{2}(i+\vep_1)t\eta^2} - e^{-\fr{1}{2}(i+\vep_2)t\eta^2}\right)\hat\phi_2(\eta) \,d\eta} \\
&\lesssim \abs{\vep_1 - \vep_2}^{\frac{1}{2}}{T^{**}}^{\frac{1}{2}}\|\hat\phi_2\|_{L^1} \\
&\lesssim \abs{\vep_1 - \vep_2}^{\frac{1}{2}}{T^{**}}^{\frac{1}{2}}\nrm{\phi_2}_{H^1}
\end{aligned}
\end{equation}
for any $t \in [0, T^{**}]$.
Moreover, for $\delta\in(0, \fr{1}{2})$, we have from Lemma \ref{est of diff free sol} that
\begin{align*}
&\nrm{\int_0^t  \left(U_{\vep_1}(t-t') - U_{\vep_2}(t-t')\right)u_2^2 (t')\,dt'}_{L_x^\infty} \\
&\lesssim \int_0^t \nrm{\left(U_{\vep_1}(t-t') - U_{\vep_2}(t-t')\right)u_2^2(t')}_{H^{\fr{1}{2}+\delta}}  \,dt' \\
&\lesssim \abs{\vep_1 - \vep_2}^{\fr{1}{4}-\fr{\delta}{2}}\int_0^t (t-t')^{\fr{1}{4}-\fr{\delta}{2}}\nrm{u_2^2(t')}_{H^1}  \,dt' \\
&\lesssim \abs{\vep_1 - \vep_2}^{\fr{1}{4}-\fr{\delta}{2}}{T^{**}}^{\fr{5}{4}-\fr{\delta}{2}}\nrm{u_2}_{L_{T^{**}}^{ \infty }H_x^{ 1 }}^2
\end{align*}
for $t\in[0, {T^{**}}]$.
A similar estimate follows when we replace $u_2^2$ with $|u_2|^2$.
Thus, Proposition \ref{smooth apriori} yields
\begin{align*}
&\nrm{\Lambda_1 - \Lambda_2}_{L_{T^{**}, x}^{ \infty }} \\
&\lesssim \nrm{\phi_1-\phi_2}_{X^0} 
+\abs{\vep_1-\vep_2}^{\frac{1}{2}}{T^{**}}^{\frac{1}{2}}\nrm{\phi_2}_{H^1} \\
&\quad +{T^{**}}^{\frac{1}{2}}\left(\nrm{u_1}_{L_{T^{**}}^\infty L_x^2}+\nrm{u_2}_{L_{T^{**}}^\infty L_x^2}\right)\nrm{u_1-u_2}_{L_{T^{**}}^\infty L_x^2}
+ \abs{\vep_1-\vep_2}^{\frac{9}{8}}{T^{**}}^{\frac{9}{8}}\nrm{u_2}_{L_{T^{**}}^\infty H_x^1}^2 \\
&\le \nrm{\phi_1-\phi_2}_{X^0} 
+\abs{\vep_1-\vep_2}^{\frac{1}{2}}{T^{**}}^{\frac{1}{2}}\nrm{\phi_2}_{H^1} \\
&\quad +C(s, \|\phi_1\|_{X^s}, \|\phi_2\|_{X^s}){T^{**}}^{\frac{1}{2}}\left(\nrm{u_1-u_2}_{L_{T^{**}}^\infty L_x^2}
+ \abs{\vep_1-\vep_2}^{\frac{1}{2}}\right),
\end{align*}
where we took $\delta=\frac{1}{4}$.
This completes the proof.
\end{proof}

To obtain the estimate for the $L_T^\infty H_x^s$-norm of the difference between two solutions, we first consider the estimate for the difference between gauge-transformed solutions.
\begin{lem}\label{smooth est of diff step2}
We assume the same condition as Proposition \ref{smooth est of diff} (i). 
Then, there exists $C=C(s, \|\phi_1\|_{X^s}, \|\phi_2\|_{X^s})>0$ which is independent of $\vep_1, \vep_2$ such that
\begin{equation*}
\begin{aligned}
&\nrm{e^{-\Lambda_1}u_1 - e^{-\Lambda_2}u_2}_{L_{T^{**}}^{ \infty }H_x^{s}} \\
&\le C\left(\nrm{\phi_1 - \phi_2}_{X^s} + \abs{\vep_1 - \vep_2}^{\fr{1}{2}}\left(1+\|\tilde\phi\|_{H^{s+1}}\right) 
+ {T^{**}}^{\fr{1}{2}}\nrm{u_1 - u_2}_{L_{T^{**}}^{ \infty }X^{s}}\right)
\end{aligned}
\end{equation*}
for all $T^{**}\in (0, \min(T_1^{**}, T_2^{**})]$.
If we assume the same condition as Proposition \ref{smooth est of diff} (ii), then the same estimate holds for all $T^{*}\in (0, \min(T_1^{*}, T_2^{*})]$ by replacing $T^{**}$ with $T^{*}$.
\end{lem}

In the following, we use the inequalities
\begin{equation*}
\nrm{e^{\pm(\Lambda_1(t)-\Lambda_2(t))}-1}_{L^\infty}
\le C(\|\Lambda_1(t)\|_{L^\infty}, \snrm{\Lambda_2(t)}_{L^\infty})\nrm{\Lambda_1(t)-\Lambda_2(t)}_{L^\infty} \\
\end{equation*}
and
\begin{equation*}
\begin{aligned}
&\nrm{e^{\pm(\Lambda_1(t)-\Lambda_2(t))}-1}_{\dot H^k} \\
&\lesssim \nrm{e^{\pm(\Lambda_1(t)-\Lambda_2(t))}}_{L^\infty}\nrm{u_1(t)-u_2(t)}_{H^{k-1}}\left(1+\nrm{u_1(t)-u_2(t)}_{L^2\cap H^{k-2}}^{k-1}\right) \\
&\le C(k, \snrm{u_1(t)}_{X^{k-1}}, \nrm{u_2(t)}_{X^{k-1}})\snrm{u_1(t)-u_2(t)}_{H^{k-1}}
\end{aligned}
\end{equation*}
for $k\in \bn$.
In particular, we have
\begin{equation}\label{pr diff2 ineq}
\nrm{e^{\pm(\Lambda_1(t)-\Lambda_2(t))}-1}_{L^\infty\cap \dot H^k}
\le C(k, \snrm{u_1(t)}_{X^{k-1}}, \snrm{u_2(t)}_{X^{k-1}})\nrm{u_1(t)-u_2(t)}_{X^{k-1}}.
\end{equation}

\begin{proof}[Proof of Lemma \ref{smooth est of diff step2}]
We may assume $\vep_1\ge \vep_2$. 
From \eqref{ep-gauge trans}, we have
\begin{align*}
e^{-\Lambda_1}u_1 - e^{-\Lambda_2}u_2
&=U_{\vep_1}(t)\left(e^{-\Lambda_1(0)}\phi_1\right) - U_{\vep_2}(t)\left(e^{-\Lambda_2(0)}\phi_2\right) \\
&\quad -i \int_0^t U_{\vep_1}(t-t')\left(e^{-\Lambda_1}\left(N_{\vep_1}^{(3)}(u_1)+\vep_1N^{(2)}(u_1)\right)\right)(t')  \,dt' \\
&\quad +i \int_0^t U_{\vep_2}(t-t')\left(e^{-\Lambda_2}\left(N_{\vep_2}^{(3)}(u_2)+\vep_2N^{(2)}(u_2)\right)\right)(t')  \,dt' \\
&=U_{\vep_1}(t)\left(e^{-\Lambda_1(0)}\phi_1 - e^{-\Lambda_2(0)}\phi_2\right) + \left(U_{\vep_1}(t)-U_{\vep_2}(t)\right) \left(e^{-\Lambda_2(0)}\phi_2\right) \\
&\quad -i \int_0^t U_{\vep_1}(t-t')\left(e^{-\Lambda_1}\left(N_{\vep_1}^{(3)}(u_1)-N_{\vep_2}^{(3)}(u_1)\right)\right)(t')  \,dt'\\
&\quad -i \int_0^t U_{\vep_1}(t-t')\left(e^{-\Lambda_1}\left(N_{\vep_2}^{(3)}(u_1)-N_{\vep_2}^{(3)}(u_2)\right)\right)(t')  \,dt'\\
&\quad -i \int_0^t U_{\vep_1}(t-t')\left(\left(e^{-\Lambda_1}- e^{-\Lambda_2}\right)N_{\vep_2}^{(3)}(u_2)\right)(t')  \,dt'\\
&\quad - i \int_0^t \left(U_{\vep_1}(t-t') - U_{\vep_2}(t-t')\right)\left(e^{-\Lambda_2}N_{\vep_2}^{(3)}(u_2)\right)(t')  \,dt' \\
&\quad -i \vep_1\int_0^t U_{\vep_1}(t-t')\left(e^{-\Lambda_1}\left(N^{(2)}(u_1)-N^{(2)}(u_2)\right)\right)(t')  \,dt'\\
&\quad -i \left(\vep_1-\vep_2\right)\int_0^t U_{\vep_1}(t-t')\left(e^{-\Lambda_1}N^{(2)}(u_2)\right)(t')  \,dt'\\
&\quad -i \vep_2\int_0^t U_{\vep_1}(t-t')\left(\left(e^{-\Lambda_1} - e^{-\Lambda_2}\right)N^{(2)}(u_2)\right)(t')  \,dt'\\
&\quad - i\vep_2 \int_0^t \left(U_{\vep_1}(t-t') - U_{\vep_2}(t-t')\right)\left(e^{-\Lambda_2}N^{(2)}(u_2)\right)(t')  \,dt' \\
&=: \I_1 + \cdots + \I_{10},
\end{align*}
where $N_{\vep}^{(3)}$ and $N^{(2)}$ are defined in \eqref{N3} and \eqref{N2}, respectively.
From Lemma \ref{seki}, \eqref{est of exp f2}, and \eqref{pr diff2 ineq}, we obtain
\begin{align*}
\nrm{\I_1}_{L_{T^{**}}^\infty H_x^s}
&\le \nrm{e^{-\Lambda_1(0)}\left(\phi_1 - \phi_2\right)}_{H^s} + \nrm{\left(1- e^{\Lambda_1(0) - \Lambda_2(0)}\right)e^{-\Lambda_1(0)}\phi_2}_{H^s} \\
&\le C(s, \|\phi_1\|_{X^s}, \|\phi_2\|_{X^s})\left(\nrm{\phi_1 - \phi_2}_{H^s} + \nrm{1- e^{\Lambda_1(0) - \Lambda_2(0)}}_{L^\infty \cap \dot H^{[s]+1}}\right) \\
&\le  C(s, \|\phi_1\|_{X^s}, \|\phi_2\|_{X^s})\nrm{\phi_1 - \phi_2}_{X^s}.
\end{align*}
The term $\I_2$ is estimated by \eqref{est of exp f2} and Lemma \ref{est of diff free sol} with $a=1$ follows:
\begin{align*}
\nrm{\I_2}_{L_{T^{**}}^\infty H_x^s}
&\lesssim {T^{**}}^{\fr{1}{2}}\abs{\vep_1 - \vep_2}^{\fr{1}{2}}\nrm{e^{-\Lambda_2(0)}\phi_2}_{H^{s+1}} \\
&\le C(s, \|\phi_1\|_{X^s}, \|\phi_2\|_{X^s}){T^{**}}^{\fr{1}{2}}\abs{\vep_1 - \vep_2}^{\fr{1}{2}}\nrm{\phi_2}_{H^{s+1}}.
\end{align*}
From \eqref{est of exp f2}, we obtain
\begin{align*}
\nrm{\I_3}_{L_{T^{**}}^\infty H_x^s}
&= {T^{**}}\nrm{e^{-\Lambda_1}\left(N_{\vep_1}^{(3)}(u_1)-N_{\vep_2}^{(3)}(u_1)\right)}_{L_{T^*}^\infty H_x^s}\\
&\le C\left(s, \nrm{u_1}_{L_{T^{**}}^\infty X^s}\right){T^{**}}\abs{\vep_1-\vep_2}
\end{align*}
and
\begin{align*}
\nrm{\I_4}_{L_{T^{**}}^\infty H_x^s}
&= {T^{**}}\nrm{e^{-\Lambda_1}\left(N_{\vep_2}^{(3)}(u_1)-N_{\vep_2}^{(3)}(u_2)\right)}_{L_{T^{**}}^\infty H_x^s}\\
&\le C\left(s, \snrm{u_1}_{L_{T^{**}}^\infty X^s}, \snrm{u_2}_{L_{T^{**}}^\infty X^s}\right){T^{**}}\nrm{u_1-u_2}_{L_{T^{**}}^\infty H_x^s}.
\end{align*}
Lemma \ref{seki}, \eqref{est of exp f2}, and \eqref{pr diff2 ineq} yield
\begin{align*}
\nrm{\I_5}_{L_{T^{**}}^\infty H_x^s}
&= {T^{**}}\nrm{\left(e^{-(\Lambda_1-\Lambda_2)}-1\right)e^{-\Lambda_2}N_{\vep_2}^{(3)}(u_2)}_{L_{T^{**}}^\infty H_x^s}\\
&\le C\left(s, \snrm{u_1}_{L_{T^{**}}^\infty X^s}, \snrm{u_2}_{L_{T^{**}}^\infty X^s}\right){T^{**}}\nrm{u_1-u_2}_{L_{T^{**}}^\infty X^s}.
\end{align*}
For $\I_7$ and $\I_8$, Lemma \ref{schauder} and \eqref{est of exp f2} yield
\begin{align*}
\nrm{\I_7}_{L_{T^{**}}^\infty H_x^s}
&=\vep_1\sup_{t\in[0, {T^{**}}]}\int_0^t\left(1+\left(\vep_1(t-t')\right)^{-\fr{1}{2}}\right)\nrm{e^{-\Lambda_1}\left(N^{(2)}(u_1)-N^{(2)}(u_2)\right)}_{L_{T^{**}}^\infty H_x^{s-1}}\, dt' \\
&\le C\left(s, \snrm{u_1}_{L_{T^{**}}^\infty X^s}, \snrm{u_2}_{L_{T^{**}}^\infty X^s}\right)\left(\vep_1{T^{**}}\right)^{\fr{1}{2}}\nrm{u_1-u_2}_{L_{T^{**}}^\infty X^s}
\end{align*}
and
\begin{align*}
\nrm{\I_8}_{L_{T^{**}}^\infty H_x^s}
&=\abs{\vep_1-\vep_2}\sup_{t\in[0, {T^{**}}]}\int_0^t\left(1+\left(\vep_1(t-t')\right)^{-\fr{1}{2}}\right)\nrm{e^{-\Lambda_1}N^{(2)}(u_2)}_{L_{T^{**}}^\infty H_x^{s-1}}\, dt' \\
&\le C\left(s, \snrm{u_1}_{L_{T^{**}}^\infty X^s}, \snrm{u_2}_{L_{T^{**}}^\infty X^s}\right)\left(\abs{\vep_1-\vep_2}{T^{**}}\right)^{\fr{1}{2}}.
\end{align*}
The term $\I_9$ is estimated by Lemmas \ref{schauder} and \ref{seki}, \eqref{est of exp f2}, and \eqref{pr diff2 ineq}:
\begin{align*}
\nrm{I_9}_{L_{T^{**}}^\infty H_x^s}
&\lesssim \vep_2\sup_{t\in[0, {T^{**}}]}\int_0^t\left(1+\left(\vep_1(t-t')\right)^{-\fr{1}{2}}\right)\nrm{\left(e^{-(\Lambda_1-\Lambda_2)}-1\right)e^{-\Lambda_2}N^{(2)}(u_2)}_{L_{T^{**}}^\infty H_x^{s-1}}\, dt' \\
&\le C\left(s, \snrm{u_1}_{L_{T^{**}}^\infty X^s}, \snrm{u_2}_{L_{T^{**}}^\infty X^s}\right)\left(\vep_2{T^{**}}\right)^{\fr{1}{2}}\nrm{u_1-u_2}_{L_{T^{**}}^\infty X^s},
\end{align*}
where we used $\vep_1\ge \vep_2$.
For $\I_6$, \eqref{est of exp f} and Lemma \ref{est of diff free sol} yield
\begin{align*}
\nrm{I_6}_{L_{T^{**}}^\infty H_x^{s}}
&\lesssim \sup_{t\in[0, {T^{**}}]}\int_0^t\abs{\vep_1-\vep_2}^{\fr{1}{2}}(t-t')^{\fr{1}{2}}\nrm{e^{-\Lambda_2}N_{\vep_2}^{(3)}(u_2)}_{L_{T^{**}}^\infty H_x^{s+1}}\, dt' \\
&\lesssim \abs{\vep_1-\vep_2}^{\fr{1}{2}}{T^{**}}^{\frac{3}{2}}\nrm{e^{-\Lambda_2}}_{L_{T^*, x}^\infty} \\
&\quad\times\left(\nrm{N_{\vep_2}^{(3)}(u_2)}_{L_{T^{**}}^\infty H_x^{s+1}} + \nrm{u_2}_{L_{T^{**}}^\infty H_x^{s+1}}\left(1+\nrm{u_2}_{L_{T^{**}}^\infty H_x^{s}}^{[s]+1}\right)\nrm{N_{\vep_2}^{(3)}(u_2)}_{L_{T^{**}}^\infty H_x^{s-[s]}}\right) \\
&\le C\left(s, \snrm{u_2}_{L_{T^{**}}^\infty X^s}\right)\abs{\vep_1-\vep_2}^{\fr{1}{2}}{T^{**}}^{\frac{3}{2}}\left(\nrm{N_{\vep_2}^{(3)}(u_2)}_{L_{T^{**}}^\infty H_x^{s+1}} + \nrm{u_2}_{L_{T^{**}}^\infty H_x^{[s]+1}}\right) \\
&\le C\left(s, \snrm{u_2}_{L_{T^{**}}^\infty X^s}\right)\abs{\vep_1-\vep_2}^{\fr{1}{2}}{T^{**}}^{\frac{3}{2}}\nrm{u_2}_{L_{T^{**}}^\infty H_x^{s+1}},
\end{align*}
where we used $\|N_{\vep_2}^{(3)}(u_2)\|_{L_{T^{**}}^\infty H_x^{s+1}}\lesssim \|u_2\|_{L_{T^{**}}^\infty H_x^{s}}^2\|u_2\|_{L_{T^{**}}^\infty H_x^{s+1}}$.
Finally, we estimate $\I_{10}$.
Considering $U_{\vep_1}(t-t')-U_{\vep_2}(t-t')=U_{\vep_2}(t-t')(U_{\vep_1-\vep_2}(t-t')-1)$,
Lemmas \ref{schauder} and \ref{est of diff free sol} and \eqref{est of exp f2} yield
\begin{align*}
\nrm{I_{10}}_{L_{T^{**}}^\infty H_x^{s}}
&\lesssim \vep_2\sup_{t\in[0, {T^{**}}]}\int_0^t\left(1+(\vep_2(t-t'))^{-\frac{1}{2}}\right)\abs{\vep_1-\vep_2}^{\fr{1}{2}}(t-t')^{\fr{1}{2}}\nrm{e^{-\Lambda_2}N^{(2)}(u_2)}_{L_{T^{**}}^\infty H_x^{s}}\, dt' \\
&\le C\left(s, \snrm{u_2}_{L_{T^{**}}^\infty X^s}\right)\abs{\vep_1-\vep_2}^{\fr{1}{2}}{T^{**}}\nrm{u_2}_{L_{T^{**}}^\infty H_x^{s+1}}.
\end{align*}
Therefore, we obtain
\begin{align*}
&\nrm{e^{-\Lambda_1}u_1 - e^{-\Lambda_2}u_2}_{L_{T^{**}}^{ \infty }H_x^{s}} \\
&\le C\left(s, \snrm{u_1}_{L_{T^{**}}^\infty X^s}, \snrm{u_2}_{L_{T^{**}}^\infty X^s}\right) \\
&\quad\times\left(\nrm{\phi_1 - \phi_2}_{X^s} + \abs{\vep_1 - \vep_2}^{\fr{1}{2}}\left(1+\nrm{\phi_2}_{H^{s+1}}\right) 
+ {T^{**}}^{\fr{1}{2}}\nrm{u_1 - u_2}_{L_{T^{**}}^{\infty}X^{s}}\right) \\
&\quad+C\left(s, \snrm{u_2}_{L_{T^{**}}^\infty X^s}\right){T^{**}}\nrm{u_2}_{L_{T^{**}}^{\infty}H_x^{s+1}}.
\end{align*}
Thus, Proposition \ref{smooth apriori} and \ref{smooth apriori2} yield the desired estimate.

If $\vep_1=\vep_2$, then we have $\I_2=\I_3=\I_6=\I_8=\I_{10}=0$. 
Thus, we obtain the same estimate.
\end{proof}
By a similar argument as the proof of Proposition \ref{smooth est of diff step2}, we obtain the following lemma.
\begin{lem}\label{smooth est of diff step3}
We assume the same condition as Proposition \ref{smooth est of diff} (i). 
Then, there exists $C=C(s, \|\phi_1\|_{X^s}, \|\phi_2\|_{X^s})>0$ which is independent of $\vep_1, \vep_2$ such that
\[
\nrm{u_1 - u_2}_{L_{T^{**}}^\infty H_x^{s-1}} 
\le C\left(\nrm{\phi_1 - \phi_2}_{H^{s-1}} + \abs{\vep_1 - \vep_2}^{\fr{1}{2}}\left(1+\|\tilde\phi\|_{H^{s+1}}\right) +{T^{**}} \nrm{u_1-u_2}_{L_{T^{**}}^\infty H_x^{s}}\right)
\]
for all $T^{**}\in (0, \min(T_1^{**}, T_2^{**})]$.
If we assume the same condition as Proposition \ref{smooth est of diff} (ii), then the same estimate holds for all $T^{**}\in (0, \min(T_1^{*}, T_2^{*})]$.
\end{lem}
\begin{proof}
We may assume $\vep_1\ge \vep_2$.
Since $u_1$ and $u_2$ are solutions to \eqref{ep-nls}, we have
\begin{align*}
&u_1(t)-u_2(t) \\
&= U_{\vep_1}(t)\left(\phi_1-\phi_2\right) + \left(U_{\vep_1}(t)-U_{\vep_2}(t)\right)\phi_2 \\
&\quad - i\int_0^tU_{\vep_1}(t-t')\px\left(\lambda (u_1^2-u_2^2) + \mu (\abs{u_1}^2-\abs{u_2}^2)\right)(t')\,dt' \\
&\quad - i\int_0^t\left(U_{\vep_1}(t-t')-U_{\vep_2}(t-t')\right)\px\left(\lambda u_2^2 + \mu \abs{u_2}^2\right)(t')\,dt'.
\end{align*}
Thus, we obtain the desired estimate by a similar argument as in the proof of Proposition \ref{smooth est of diff step2}.
\end{proof}
\begin{proof}[Proof of Proposition \ref{smooth est of diff}]

From Lemma \ref{seki} and \eqref{pr diff2 ineq}, we have
\[
\nrm{u_1\left(1-e^{-\Lambda_1+\Lambda_2}\right)}_{L_{T^{**}}^{\infty}H_x^{s}}
\le C(s, \|u_1\|_{L_{T^{**}}^{\infty}X^{s}}, \|u_2\|_{L_{T^{**}}^{\infty}X^{s}})\nrm{u_1-u_2}_{L_{T^{**}}^{\infty}X^{[s]}}.
\]
Here, we obtain
\[
\nrm{u_1-u_2}_{L_{T^{**}}^{ \infty }H_x^{[s]}}
\le \nrm{u_1\left(1-e^{-\Lambda_1+\Lambda_2}\right)}_{L_{T^{**}}^{\infty}H_x^{[s]}}
+ \nrm{e^{-\Lambda_1}u_1 - e^{-\Lambda_2}u_2}_{L_{T^{**}}^{ \infty }H_x^{[s]}},
\]
and a direct calculation yields
\[
\nrm{u_1\left(1-e^{-\Lambda_1+\Lambda_2}\right)}_{L_{T^{**}}^{\infty}H_x^{[s]}}
\le C(s, \|u_1\|_{L_{T^{**}}^{\infty}X^{s}}, \|u_2\|_{L_{T^{**}}^{\infty}X^{s}})\nrm{u_1-u_2}_{L_{T^{**}}^{\infty}X^{[s]-1}}.
\]
Therefore, \eqref{est of exp f2} and Lemmas \ref{smooth est of diff step1}, \ref{smooth est of diff step2}, and \ref{smooth est of diff step3} yield
\begin{align*}
&\nrm{u_1 - u_2}_{L_{T^{**}}^{ \infty }H_x^{ s }} \\
&\le \nrm{u_1\left(1 - e^{-\Lambda_1 +\Lambda_2}\right)}_{L_{T^{**}}^{ \infty }H_x^{ s }} + \nrm{e^{\Lambda_2}\left(e^{-\Lambda_1}u_1 - e^{-\Lambda_2}u_2\right)}_{L_{T^{**}}^{ \infty }H_x^{ s }} \\
&\le C\left(s, \snrm{\phi_1}_{X^s}, \snrm{\phi_2}_{X^s}\right)\left(\nrm{u_1-u_2}_{L_{T^{**}}^{\infty}X^{[s]-1}} + \nrm{e^{-\Lambda_1}u_1 - e^{-\Lambda_2}u_2}_{L_{T^{**}}^{\infty}H^{s}}\right) \\
&\le C\left(s, \snrm{\phi_1}_{X^s}, \snrm{\phi_2}_{X^s}\right)\left(\nrm{\phi_1 - \phi_2}_{X^s} + \abs{\vep_1 - \vep_2}^{\fr{1}{2}}\left(1+\|\tilde \phi\|_{H^{s+1}}\right) 
+ {T^{**}}^{\fr{1}{2}}\nrm{u_1 - u_2}_{L_{T^{**}}^{ \infty }X^{s}}\right).
\end{align*}
Thus, Lemma \ref{smooth est of diff step1} yields
\begin{align*}
&\nrm{u_1 - u_2}_{L_{T^{**}}^{ \infty }X^{ s }} \\
&\le C\left(s, \snrm{\phi_1}_{X^s}, \snrm{\phi_2}_{X^s}\right)\left(\nrm{\phi_1 - \phi_2}_{X^s} + \abs{\vep_1 - \vep_2}^{\fr{1}{2}}\left(1+\|\tilde \phi\|_{H^{s+1}}\right) 
+ {T^{**}}^{\fr{1}{2}}\nrm{u_1 - u_2}_{L_{T^{**}}^{ \infty }X^{s}}\right).
\end{align*}
Taking ${T^{**}}^{\frac{1}{2}}C(s, \|\phi_1\|_{X^s}, \|\phi_2\|_{X^s})<\frac{1}{2}$ if necessary, we obtain the desired estimate.
\end{proof}

\subsection{Proof of Proposition \ref{existence of a smooth solution}}
In this subsection, we prove Proposition \ref{existence of a smooth solution}.
We prove the uniqueness and the continuity of the flow map in Section 4, since we can prove these properties including the case $s\ge 0$.
\begin{comment}
We prove the following theorem:
\begin{thm}
Let $s\ge 1$ and $R>0$. For any $\phi\in X_R^s$, there exist $T=T(s, R)>0$ and an unique solution $u$ of \eqref{nls} in $C\left([0, T]; X^s\right)\cap C^1\left([0, T]; H^{s-2}\right)$. Also, the flow map is continuous from $X_R^s$ to $C\left([0, T]; X^s\right)\cap C^1\left([0, T]; H^{s-2}\right)$.
\end{thm}
\end{comment}

We obtained the estimate for the $X^s$-norm of the difference between two solutions by partially using the $H^{s+1}$-norm of the initial data. 
We treat this problem by using a Bona-Smith type approximation.
For $s\ge 1$ and $\eta\in(0,1]$, we define $J_{\eta, s}f := e^{-\eta\abs{\px}^s}f$.
\begin{lem}\label{BS}
Let $s\ge 1$, $\eta\in(0,1]$, $j\ge 0$, and $f\in X^s(\br)$.
Then, the following properties hold:
\begin{align*}
&(\mathrm{i})\quad J_{\eta, s}f\in X^{\infty}, \\
&(\mathrm{ii})\quad \nrm{J_{\eta, s}f}_{X^s}\lesssim \nrm{f}_{X^s}, \\
&(\mathrm{iii})\quad \lim_{\eta\to+0}\nrm{J_{\eta, s}f  - f}_{X^s}=0, \\
&(\mathrm{iv})\quad \nrm{J_{\eta, s}f}_{H^{s+j}}\lesssim \eta^{-\fr{j}{s}}\nrm{f}_{H^s}.
\end{align*}
\end{lem}
\begin{proof}
We only prove (iv) since the proofs of (i), (ii), and (iii) are similar arguments as Lemma \ref{est of free sol}.
We have
\begin{gather*}
\eta^{\fr{j}{s}}\nrm{J_{\eta, s}f}_{\dot H^{{s+j}}}
= \nrm{\abs{\xi}^s\left(\eta\abs{\xi}^s\right)^{\fr{j}{s}}e^{-\eta\abs{\xi}^s}\hat f}_{L^2}
\lesssim \nrm{f}_{\dot H^s}, \\
\eta^{\fr{j}{s}}\nrm{J_{\eta, s}f}_{L^2} \le \nrm{f}_{L^2}.
\end{gather*}
Thus, we obtain the desired estimate.
\end{proof}
\begin{proof}[Proof of Proposition \ref{existence of a smooth solution}]
Let $\phi\in X^s(\br)$ and $\nrm{\phi}_{X^s}\le R$. For $\eta\in(0, 1)$, we denote $\phi_\eta:=J_{\eta, s}\phi$.
From Lemma \ref{BS}, we have $\nrm{\phi_\eta}_{X^s}\le 2R$ if $\eta$ is sufficiently small.
Therefore, Proposition \ref{smooth apriori2} yields that there exists $T^{**}=T^{**}(s, R)\in (0, 1)$ such that for each $\vep$ and sufficiently small $\eta$, there exists a solution $u_{\vep, \eta}\in C([0, T^{**}]; X^{s+1}(\br))$ to \eqref{ep-nls} with regularization parameter $\vep$ and initial data $\phi_\eta$.

We denote $u_{\eta}:=u_{\eta^3, \eta}$.
Then, Proposition \ref{smooth est of diff} yields that there exist $\tilde T^{**}={\tilde T^{**}}(s, R)\in (0, T^{**}]$ and $C=C(s, R)>0$ such that 
\begin{align*}
\nrm{u_{\eta_1} - u_{\eta_2}}_{L_{\tilde T^{**}}^{ \infty }X^{ s }} 
&\le C \left(\nrm{\phi_{\eta_1} - \phi_{\eta_2}}_{X^s} + \abs{\eta_1^3 - \eta_2^3}^{\fr{1}{2}}\left(\nrm{\phi_{\eta_2}}_{H^{s+1}}+1\right)\right)
\end{align*}
\begin{comment}
\begin{align*}
&\nrm{u_{\eta_1} - u_{\eta_2}}_{L_{T^*}^{ \infty }X^{ s }} \\
&\le e^{C\left(\nrm{\phi_{\eta_1}}_{X^s} + \nrm{\phi_{\eta_2}}_{X^s}\right)}P_s\left(\nrm{\phi_{\eta_1}}_{X^s}, \nrm{\phi_{\eta_2}}_{X^s}\right)\left(\nrm{\phi_{\eta_1} - \phi_{\eta_2}}_{X^s} + \abs{\eta_1^3 - \eta_2^3}^{\fr{1}{2}}\left(\nrm{\phi_{\eta_2}}_{H^{s+1}}+1\right)\right) \\
&\le P_s\left(\nrm{\phi_{\eta_1}}_{X^s}, \nrm{\phi_{\eta_2}}_{X^s}\right)\left(\nrm{\phi_{\eta_1} - \phi_{\eta_2}}_{X^s}  + \eta_2^{\fr{3}{2}}\left(\eta_2^{-\fr{1}{s}}\nrm{\phi_{\eta_2}}_{H^{s}}+1\right)\right) \\
&\ra 0\quad\left(\eta_1, \eta_2\ra +0\right)
\end{align*}
\end{comment}
for sufficiently small $\eta_1\ge \eta_2> 0$.
Here, Lemma \ref{BS} (iv) yields that
\[
\abs{\eta_1^3 - \eta_2^3}^{\fr{1}{2}}\nrm{\phi_{\eta_2}}_{H^{s+1}}
\lesssim \eta_2^{\fr{3}{2}}\eta_2^{-\fr{1}{s}}\nrm{\phi_{\eta_2}}_{H^{s}}.
\]
Hence, $\{u_{\eta}\}_{\eta\in(0, 1)}$ is a Cauchy net in $C([0, {\tilde T^{**}}]; X^s(\br))$. Thus, there exists a limit of $u_{\eta}$ in $C([0, {\tilde T^{**}}]; X^s(\br))$.
In particular, $\px^2 u_{\eta}\ra \px^2 u$, $\px \left(u_{\eta}^2\right)\ra \px (u^2)$, and $\px (\abs{u_\eta}^2)\to \px (\abs{u}^2)$ hold in $C([0, {\tilde T^{**}}]; H^{s-2}(\br))$.
Taking $\eta\ra +0$ in \eqref{ep-sol}, we obtain that $u$ is a solution to \eqref{nls} with initial data $\phi$.
\end{proof}

\begin{rem}\label{remark} \rm
In the proof of the existence, we only consider the approximation with $\vep = \eta^3$. 
However, we can show that $\lim_{(\vep, \eta)\to (0, 0)}u_{\vep, \eta}=u$ in $C([0, {\tilde T^{**}}]; X^s(\br))$. 
We prove this property in Appendix \ref{appendixA}.
\end{rem}

\section{Well-posedness in $X^s(\br)$ with $s\ge0$}
In this section, we prove the local well-posedness of \eqref{nls} in $X^s(\br)$ with $s\ge0$. 
By using the gauge transformation, we obtained the equation \eqref{gauge trans} which has no derivatives in the nonlinearities.
In the following, we use this equation in order to prove the well-posedness.

\subsection{The existence time of smooth solutions}
We first prove that the existence time of the solution depends only on the $X^0$-norm of the initial data.

\begin{prop}\label{existence time}
Let $s'>\frac{3}{2}$ and $\phi\in X^{s'}(\br)$. 
Then, there exist $T=T(\nrm{\phi}_{X^0})\in (0, 1)$ and a solution $u\in C([0, T]; X^{s'}(\br))$ to \eqref{nls}.
\end{prop}

To prove Proposition \ref{existence time}, we prepare some lemmas.
The first lemma is the fractional Leibniz rule.
See \cite[Lemma X4]{KP}.
\begin{lem}\label{fracLeib}
If $s>0$ and $1<p<\infty$, then we have
\[
\nrm{(1-\px^2)^{\frac{s}{2}}(fg)}_{L^p}
\lesssim \nrm{(1-\px^2)^{\frac{s}{2}}f}_{L^p}\nrm{g}_{L^\infty} + \nrm{f}_{\infty}\nrm{(1-\px^2)^{\frac{s}{2}} g}_{L^p} .
\]
\end{lem}

The second lemma is the estimate for the $L_T^\infty H_x^k$-norm which uses $L_T^\infty H_x^{k-1}$-norm, where $k\in \bn$. 
\begin{lem}\label{apriori-1}
Let $s'\ge 1$.
Then, for any $k\in\bn$ with $k\le [s']$, there exists $C_4=C_4(k)>0$ such that
\begin{align*}
\nrm{u}_{L_{T}^\infty X^k}
&\le C_4e^{C_4(\nrm{\phi}_{X^0}+T^{\frac{1}{2}}\nrm{u}_{L_T^\infty L_x^2}^2)}\left(1+\nrm{u}_{L_T^\infty H_x^{k-1}}^k\right) \\
&\quad \times\left(\left(1 + \nrm{\phi}_{H^{k-1}}^{k}\right)\nrm{\phi}_{X^{k}} + T^\frac{1}{2}\left(1 + \nrm{u}_{L_T^\infty H_x^{k-1}}^{k}\right)\left(1+\nrm{u}_{L_T^4 L_x^\infty}\right)^2\nrm{u}_{L_T^\infty H_x^{k}}\right)
\end{align*}
for any solution $u\in C([0, T]; X^{s'}(\br))$ with $T\in (0, 1)$ to \eqref{nls}.
\end{lem}
\begin{proof}
A direct calculation yields
\begin{equation}\label{apriori-1 pr1}
\nrm{e^{\pm \Lambda(t)}f}_{H^{k}}
\lesssim \nrm{e^{\pm \Lambda(t)}}_{L^\infty}\left(1 + \nrm{u(t)}_{H^{k-1}}^k\right)\nrm{f}_{H^k}.
\end{equation}
Thus, we obtain
\[
\nrm{e^{- \Lambda(0)}\phi}_{H^{k}}
\lesssim \nrm{e^{- \Lambda}}_{L_{T, x}^\infty}\left(1 + \nrm{\phi}_{H^{k-1}}^k\right)\nrm{\phi}_{H^k}.
\]
Similarly, we have
\begin{align*}
&\nrm{e^{- \Lambda}|u|^2u}_{L_T^1 H_x^{k}} + \nrm{e^{- \Lambda}|u|^2\bar u}_{L_T^1 H_x^{k}} \\
&\lesssim \nrm{e^{- \Lambda}}_{L_{T, x}^\infty}\left(1 + \nrm{u}_{L_T^\infty H_x^{k-1}}^k\right)\left(\nrm{|u|^2 u}_{L_T^1 H_x^{k}} + \nrm{|u|^2 \bar u}_{L_T^1 H_x^{k}}\right) \\
&\lesssim T^{\frac{1}{2}}\nrm{e^{- \Lambda}}_{L_{T, x}^\infty}\left(1 + \nrm{u}_{L_T^\infty H_x^{k-1}}^k\right)\nrm{u}_{L_T^4 L_x^\infty}^2\nrm{u}_{L_T^\infty H_x^{k}},
\end{align*}
where the last inequality is obtained by using Lemma \ref{fracLeib}.
Hence, Proposition \ref{prop gauge} and \eqref{apriori-1 pr1} yield
\begin{align*}
\nrm{u}_{L_{T}^\infty H_x^k} 
&=\nrm{e^{\Lambda}e^{-\Lambda}u}_{L_{T}^\infty H_x^k} \\
&\lesssim \nrm{e^{\Lambda}}_{L_{T, x}^\infty}\left(1+\nrm{u}_{L_T^\infty H_x^{k-1}}^k\right)\nrm{e^{-\Lambda}u}_{L_T^\infty H_x^k} \\
&\lesssim \nrm{e^{\Lambda}}_{L_{T, x}^\infty}\nrm{e^{-\Lambda}}_{L_{T, x}^\infty}\left(1+\nrm{u}_{L_T^\infty H_x^{k-1}}^k\right) \\
&\quad \times\left(\left(1 + \nrm{\phi}_{H^{k-1}}^{k}\right)\nrm{\phi}_{H^{k}} + T^{\frac{1}{2}}\left(1 + \nrm{u}_{L_T^\infty H_x^{k-1}}^k\right)\nrm{u}_{L_T^4 L_x^\infty}^2\nrm{u}_{L_T^\infty H_x^{k}}\right).
\end{align*}
Combining Lemma \ref{smooth apriori step4}, we obtain the desired estimate.
\end{proof}

The third lemma is the estimate for the $L_T^\infty H_x^s$-norm which uses $L_T^\infty H_x^{[s]}$-norm. We note that $s$ is not necessarily in $\bn$. 
\begin{lem}\label{apriori0}
Let $s'>\frac{1}{2}$ and $0\le s\le s'$.
Then, there exist $C_5=C_5(s)>0$ such that
\begin{align*}
&\nrm{u}_{L_{T}^\infty X^s\cap L_{T}^4 L_x^\infty} \\
&\le C_5e^{C_5(\nrm{\phi}_{X^0}+T^{\frac{1}{2}}\nrm{u}_{L_T^\infty L_x^2}^2)}\left(1+\nrm{u}_{L_T^\infty H_x^{[s]}}^{[s]+1}\right) \\
&\quad \times\left(\left(1+\nrm{\phi}_{H^{[s]}}^{[s]+1}\right)\nrm{\phi}_{X^s} + T^{\frac{1}{2}}\left(1+\nrm{u}_{L_T^\infty H_x^{[s]}}^{[s]+1}\right)\left(1 + \nrm{u}_{L_T^4 L_x^\infty}\right)^2\nrm{u}_{L_T^\infty H_x^s}\right)
\end{align*}
for any solution $u\in C([0, T]; X^{s'}(\br))$ with $T\in (0, 1)$ to \eqref{nls} .
\end{lem}

\begin{proof}
For the $L_T^4L_x^\infty$-norm, Lemma \ref{strest} and Proposition \ref{prop gauge} yield
\begin{align*}
\nrm{e^{-\Lambda}u}_{L_T^4L_x^\infty}
&\lesssim \nrm{e^{-\Lambda(0)}\phi}_{L_2} + \nrm{e^{-\Lambda}\abs{u}^2u}_{L_T^1L_x^2} + \nrm{e^{-\Lambda}\abs{u}^2\bar u}_{L_T^1L_x^2} \\
&\lesssim \nrm{e^{-\Lambda}}_{L_{T, x}^\infty}\left(\nrm{\phi}_{L^2} + T^\frac{1}{2}\nrm{u}_{L_T^4L_x^\infty}^2\nrm{u}_{L_T^\infty L_x^2}\right).
\end{align*}
Thus, we have
\begin{equation}\label{pr apriori01}
\nrm{u}_{L_T^4L_x^\infty}
\le \nrm{e^{\Lambda}}_{L_{T, x}^\infty}\nrm{e^{-\Lambda}u}_{L_T^4L_x^\infty}
\lesssim   \nrm{e^{\Lambda}}_{L_{T, x}^\infty}\nrm{e^{-\Lambda}}_{L_{T, x}^\infty}\left(\nrm{\phi}_{L^2} + T^\frac{1}{2}\nrm{u}_{L_T^4L_x^\infty}^2\nrm{u}_{L_T^\infty L_x^2}\right).
\end{equation}
Hence, Lemma \ref{smooth apriori step4} yields the desired estimate for the $L_T^4L_x^\infty$-norm.

Next, we consider the $L_T^\infty H_x^s$- norm.
From \eqref{est of exp f2}, we have
\[
\nrm{e^{-\Lambda(0)}\phi}_{H^{s}}
\lesssim \nrm{e^{-\Lambda(0)}}_{L^\infty}\left(1+\nrm{\phi}_{H^{[s]}}^{[s]+1}\right)\nrm{\phi}_{H^s}.
\]
Also, Lemma \ref{fracLeib} and \eqref{est of exp f2} yield
\begin{align*}
\nrm{e^{-\Lambda}\abs{u}^2u}_{L_T^1 H_x^s}
&\lesssim \nrm{e^{-\Lambda}}_{L_{T, x}^\infty}\left(1+\nrm{u}_{L_T^\infty H_x^{[s]}}^{[s]+1}\right)\nrm{\abs{u}^2u}_{L_T^1 H_x^s} \\
&\lesssim T^{\frac{1}{2}}\nrm{e^{-\Lambda}}_{L_{T, x}^\infty}\left(1+\nrm{u}_{L_T^\infty H_x^{[s]}}^{[s]+1}\right)\nrm{u}_{L_T^4 L_x^\infty}^2\nrm{u}_{L_T^\infty H_x^s},
\end{align*}
where we used Lemma \ref{fracLeib} to obtain the last inequality.
We can obtain the estimate for $\|e^{-\Lambda}\abs{u}^2\bar u\|_{L_T^1 H_x^s}$ similarly.
Thus, Proposition \ref{prop gauge} yields
\begin{equation}\label{gauge step1}
\begin{aligned}
\nrm{e^{-\Lambda}}_{L_{T, x}^\infty}^{-1}\nrm{e^{-\Lambda}u}_{L_T^\infty H_x^s}
\lesssim \left(1+\nrm{\phi}_{H^{[s]}}^{[s]+1}\right)\nrm{\phi}_{H^s} + T^{\frac{1}{2}}\left(1+\nrm{u}_{L_T^\infty H_x^{[s]}}^{[s]+1}\right)\nrm{u}_{L_T^4 L_x^\infty}^2\nrm{u}_{L_T^\infty H_x^s}.
\end{aligned}
\end{equation}
Hence, we obtain from \eqref{est of exp f2} that
\begin{align*}
\nrm{u}_{L_T^\infty H_x^s}
&=\nrm{e^{\Lambda}e^{-\Lambda}u}_{L_T^\infty H_x^s} \\
&\lesssim \nrm{e^{\Lambda}}_{L_{T, x}^\infty}\left(1+\nrm{u}_{L_T^\infty H_x^{[s]}}^{[s]+1}\right)\nrm{e^{-\Lambda}u}_{L_T^\infty H_x^s} \\
&\lesssim \nrm{e^{\Lambda}}_{L_{T, x}^\infty}\nrm{e^{-\Lambda}}_{L_{T, x}^\infty}\left(1+\nrm{u}_{L_T^\infty H_x^{[s]}}^{[s]+1}\right) \\
&\quad\times\left(\left(1+\nrm{\phi}_{H^{[s]}}^{[s]+1}\right)\nrm{\phi}_{H^s} + T^{\frac{1}{2}}\left(1+\nrm{u}_{L_T^\infty H_x^{[s]}}^{[s]+1}\right)\nrm{u}_{L_T^4 L_x^\infty}^2\nrm{u}_{L_T^\infty H_x^s}\right).
\end{align*}
From Lemma \ref{smooth apriori step4}, we obtain the desired estimate.
\end{proof}

The fourth lemma is as follows.
To prove Proposition \ref{existence time}, only the case $s=0$ is used.
However, we also use this lemma when we get an a priori estimate for the solution.
\begin{lem}\label{apriori1}
Let $s'>\fr{1}{2}$ and $0\le s\le s'$.
Then, there exists $C_6=C_6(s)>0$ such that
\begin{align*}
\nrm{u}_{L_{T}^\infty X^s\cap L_{T}^4 L_x^\infty}
&\le C_6e^{C_6(\nrm{\phi}_{X^0}+T^{\frac{1}{2}}\nrm{u}_{L_T^\infty L_x^2}^2)} \\
&\quad\times\left(\left(1+\nrm{\phi}_{H^s}^{2([s]+1)^2}\right)\nrm{\phi}_{X^s} + T^{\frac{1}{2}}\left(\nrm{u}_{L_T^\infty H_x^s\cap L_T^4 L_x^\infty} + \nrm{u}_{L_T^\infty H_x^s\cap L_T^4 L_x^\infty}^{4([s]+2)^4}\right)\right)
\end{align*}
for any solution $u\in C([0, T]; X^{s'}(\br))$ with $T\in (0, 1)$ to \eqref{nls}.
\end{lem}

\begin{proof}[Proof of Lemma \ref{apriori1}]
(\textit{(i) $0\le s< 1$})
In the same way as \eqref{pr apriori01}, we obtain
\[
\nrm{u}_{L_T^\infty L_x^2}
\lesssim   \nrm{e^{\Lambda}}_{L_{T, x}^\infty}\nrm{e^{-\Lambda}}_{L_{T, x}^\infty}\left(\nrm{\phi}_{L^2} + T^\frac{1}{2}\nrm{u}_{L_T^4L_x^\infty}^2\nrm{u}_{L_T^\infty L_x^2}\right).
\]
Thus, Lemmas \ref{smooth apriori step4} and \ref{apriori0} yield
\begin{equation}\label{pr apriori11}
\begin{aligned}
&\nrm{u}_{L_{T}^\infty X^s\cap L_{T}^4 L_x^\infty} \\
&\le Ce^{C(\nrm{\phi}_{X^0}+T^{\frac{1}{2}}\nrm{u}_{L_T^\infty L_x^2}^2)}\left(1+\nrm{u}_{L_T^\infty L_x^2}\right)^2 \\
&\quad \times\left(\left(1+\nrm{\phi}_{L^2}\right)\nrm{\phi}_{X^s} + T^{\frac{1}{2}}\left(1 + \nrm{u}_{L_T^4 L_x^\infty}^2\right)\nrm{u}_{L_T^\infty H_x^s}\right) \\
&\le Ce^{C(\nrm{\phi}_{X^0}+T^{\frac{1}{2}}\nrm{u}_{L_T^\infty L_x^2}^2)}\left(1+\nrm{\phi}_{L^2} + T^\frac{1}{2}\nrm{u}_{L_T^4L_x^\infty}^2\nrm{u}_{L_T^\infty L_x^2}\right)^2 \\
&\quad \times\left(\left(1+\nrm{\phi}_{L^2}\right)\nrm{\phi}_{X^s} + T^{\frac{1}{2}}\left(1 + \nrm{u}_{L_T^4 L_x^\infty}^2\right)\nrm{u}_{L_T^\infty H_x^s}\right) \\
&\le Ce^{C(\nrm{\phi}_{X^0}+T^{\frac{1}{2}}\nrm{u}_{L_T^\infty L_x^2}^2)}\left(\left(1+\nrm{\phi}_{L^2}\right)^2\nrm{\phi}_{X^s} + T^{\frac{1}{2}}\left(1 + \nrm{u}_{L_T^\infty L_x^2\cap L_T^4 L_x^\infty}^8\right)\nrm{u}_{L_T^\infty H^s}\right).
\end{aligned}
\end{equation}
Thus, we obtain the desired bound.

(\textit{(ii) $s\ge 1$})
We note that
we can obtain from Lemmas \ref{strest} and \ref{fracLeib} that
\begin{equation}\label{pr apriori12}
\nrm{u}_{L_T^\infty H_x^{s-1}}
\lesssim \nrm{\phi}_{H^{s-1}} + T^{\frac{3}{4}}\nrm{u}_{L_T^4 L_x^\infty}\nrm{u}_{L_T^\infty H_x^s},
\end{equation}
since we have \eqref{mild-sol} with $\vep=0$ in $H^{s'-1}(\br)$.
We consider the $L_T^\infty H_x^s$-norm.
From $s\ge 1$, we have
\[
\nrm{u}_{L_T^\infty H_x^s}
= \nrm{e^{\Lambda}e^{-\Lambda}u}_{L_T^\infty H_x^s}
\sim \nrm{u}_{L_T^\infty L_x^2} + \nrm{u\px\Lambda}_{L_T^\infty H_x^{s-1}} + \nrm{e^{\Lambda}\px(e^{-\Lambda}u)}_{L_T^\infty H_x^{s-1}}.
\]
The estimate for $\nrm{u}_{L_T^\infty L_x^2}$ is obtained by \eqref{pr apriori11}.
Also, Lemmas \ref{smooth apriori step4} and \ref{fracLeib}, \eqref{gauge step1}, \eqref{pr apriori12}, and the Sobolev embedding $H^{s}(\br)\hookrightarrow L^\infty(\br)$ yield
\begin{align*}
&\nrm{u\px\Lambda}_{L_T^\infty H_x^{s-1}} \\
&\le C\nrm{u}_{L_{T, x}^\infty}\nrm{u}_{L_T^\infty H_x^{s-1}} \\
&\le C\nrm{e^{\Lambda}}_{L_{T, x}^\infty}\nrm{e^{-\Lambda}u}_{L_{T, x}^\infty}\nrm{u}_{L_T^\infty H_x^{s-1}} \\
&\le Ce^{C(\nrm{\phi}_{X^0}+T^{\frac{1}{2}}\nrm{u}_{L_T^\infty L_x^2}^2)}\left(\left(1+\nrm{\phi}_{H^{[s]}}^{[s]+1}\right)\nrm{\phi}_{H^s} + T^{\frac{1}{2}}\left(1+\nrm{u}_{L_T^\infty H_x^{[s]}}^{[s]+1}\right)\nrm{u}_{L_T^4 L_x^\infty}^2\nrm{u}_{L_T^\infty H_x^s}\right) \\
&\quad \times\left(\nrm{\phi}_{H^{s-1}} + T^{\frac{3}{4}}\nrm{u}_{L_T^4 L_x^\infty}\nrm{u}_{L_T^\infty H_x^s}\right).
\end{align*}
Thus, we have
\begin{equation}\label{pr apriori13}
\begin{aligned}
\nrm{u\px\Lambda}_{L_T^\infty H_x^{s-1}}
&\le Ce^{C(\nrm{\phi}_{X^0}+T^{\frac{1}{2}}\nrm{u}_{L_T^\infty L_x^2}^2)} \\
&\quad\times\left(\left(\nrm{\phi}_{H^s}+\nrm{\phi}_{H^s}^{[s]+2}\right)\nrm{\phi}_{H^s} + T^{\frac{1}{2}}\left(\nrm{u}_{L_T^\infty H_x^s}^2+\nrm{u}_{L_T^\infty H_x^s}^{[s]+5}\right)\nrm{u}_{L_T^\infty H_x^s}\right).
\end{aligned}
\end{equation}

Finaly, we consider $\nrm{e^{\Lambda}\px(e^{-\Lambda}u)}_{L_T^\infty H_x^{s-1}}$.
From \eqref{est of exp f2}, we obtain
\[
\nrm{e^{\Lambda}\px(e^{-\Lambda}u)}_{L_T^\infty H_x^{s-1}}
\lesssim \nrm{e^{\Lambda}}_{L_{T, x}^\infty}\left(1+\nrm{u}_{L_T^\infty H_x^{[s]}}^{[s]+1}\right)\nrm{e^{-\Lambda}u}_{L_T^\infty H_x^s}.
\]
Here, Lemma \ref{apriori-1} and \eqref{pr apriori12} yield
\begin{align*}
&\nrm{u}_{L_T^\infty H_x^{[s]}} \\
&\le Ce^{C(\nrm{\phi}_{X^0}+T^{\frac{1}{2}}\nrm{u}_{L_T^\infty L_x^2}^2)}\left(\left(1+\nrm{\phi}_{H^s}^{2[s]}\right)\nrm{\phi}_{H^s} + T^{\frac{1}{2}}\left(1+\nrm{u}_{L_T^\infty H_x^s}^{3[s]+2}\right)\nrm{u}_{L_T^\infty H_x^s}\right).
\end{align*}
In particular, we have
\begin{align*}
1+\nrm{u}_{L_T^\infty H_x^{[s]}}^{[s]+1}
&\le Ce^{C(\nrm{\phi}_{X^0}+T^{\frac{1}{2}}\nrm{u}_{L_T^\infty L_x^2}^2)} \\
&\quad\times\left(1+\nrm{\phi}_{H^s}^{(2[s]+1)([s]+1)} + T^{\frac{1}{2}}\left(\nrm{u}_{L_T^\infty H_x^s}^{[s]+1}+\nrm{u}_{L_T^\infty H_x^s}^{3([s]+1)^2}\right)\right).
\end{align*}
Thus, Lemma \ref{smooth apriori step4}, \eqref{gauge step1}, and the Sobolev embedding $H^{s}(\br)\hookrightarrow L^\infty(\br)$ yield
\begin{equation}\label{pr apriori14}
\begin{aligned}
&\nrm{e^{\Lambda}\px(e^{-\Lambda}u)}_{L_T^\infty H_x^{s-1}} \\
&\le Ce^{C(\nrm{\phi}_{X^0}+T^{\frac{1}{2}}\nrm{u}_{L_T^\infty L_x^2}^2)}\left(1+\nrm{\phi}_{H^s}^{(2[s]+1)([s]+1)} + T^{\frac{1}{2}}\left(\nrm{u}_{L_T^\infty H_x^s}^{[s]+1}+\nrm{u}_{L_T^\infty H_x^s}^{3([s]+1)^2}\right)\right) \\
&\quad\times\left(\left(1+\nrm{\phi}_{H^{[s]}}^{[s]+1}\right)\nrm{\phi}_{H^s} + T^{\frac{1}{2}}\left(1+\nrm{u}_{L_T^\infty H_x^{[s]}}^{[s]+1}\right)\nrm{u}_{L_T^4 L_x^\infty}^2\nrm{u}_{L_T^\infty H_x^s}\right) \\
&\le Ce^{C(\nrm{\phi}_{X^0}+T^{\frac{1}{2}}\nrm{u}_{L_T^\infty L_x^2}^2)}\left(\left(1+\nrm{\phi}_{H^s}^{2([s]+1)^2}\right)\nrm{\phi}_{H^s} + T^{\frac{1}{2}}\left(\nrm{u}_{L_T^\infty H_x^s}^2 + \nrm{u}_{L_T^\infty H_x^s}^{4([s]+2)^4}\right)\right),
\end{aligned}
\end{equation}
where we used $[s]+4\le ([s]+2)^2$.
Therefore, \eqref{pr apriori11} with $s=0$, \eqref{pr apriori13}, and \eqref{pr apriori14} yield
\begin{equation}\label{pr apriori15}
\begin{aligned}
\nrm{u}_{L_T^\infty X^s}
&\le Ce^{C(\nrm{\phi}_{X^0}+T^{\frac{1}{2}}\nrm{u}_{L_T^\infty L_x^2}^2)} \\
&\quad\times\left(\left(1+\nrm{\phi}_{H^s}^{2([s]+1)^2}\right)\nrm{\phi}_{X^s} + T^{\frac{1}{2}}\left(\nrm{u}_{L_T^\infty H_x^s\cap L_T^4 L_x^\infty} + \nrm{u}_{L_T^\infty H_x^s\cap L_T^4 L_x^\infty}^{4([s]+2)^4}\right)\right).
\end{aligned}
\end{equation}
We note that $\nrm{u}_{L_T^\infty X^s\cap L_T^4 L_x^\infty}$ is also bounded by the right-hand side of \eqref{pr apriori15} when $0\le s< 1$.
This completes the proof.
\end{proof}

\begin{proof}[Proof of Proposition \ref{existence time}]
For simplicity, we only prove the case $\frac{3}{2}< s'< 2$. 
The case $s'\ge 2$ is shown similarly.
From Proposition \ref{existence of a smooth solution}, there exist $T_{s'} \in (0, 1)$ and the solution $u\in C([0, T_{s'}]; X^{s'}(\br))$.
We define $C = 3\max(C_4(1), C_6(0))$, $\tilde C = 3\max(C_4(1), C_5(s'), C_6(0))$, where $C_4(1), C_5(s')$, and $C_6(0)$ are constants taken in Lemmas \ref{apriori-1}--\ref{apriori1}.
Also, we define
\[
\tilde T_0 = \min\left(\frac{\nrm{\phi}_{X^0}^2}{{C}^4(\nrm{\phi}_{X^0}+\nrm{\phi}_{X^0}^3)^4}, \frac{(\nrm{\phi}_{X^0}+\nrm{\phi}_{X^0}^3)^2}{(C(\nrm{\phi}_{X^0}+\nrm{\phi}_{X^0}^3) + C^{64}(\nrm{\phi}_{X^0}+\nrm{\phi}_{X^0}^3)^{64})^2}\right),
\]
\begin{equation}\label{T0}
N_0= Ce^{C\nrm{\phi}_{X^0}}(\nrm{\phi}_{X^0}+\nrm{\phi}_{X^0}^3), \quad
T_0 = \min\left(\tilde T_0, \frac{\nrm{\phi}_{X^0}^2}{N_0^4}, \frac{1}{4C^2e^{2C\nrm{\phi}_{X^0}}(1+N_0)^8}\right),
\end{equation}
\[ 
N_1= Ce^{C\nrm{\phi}_{X^0}}(1+\nrm{\phi}_{X^0}^2)\nrm{\phi}_{X^1}, \quad
T_1 = \min\left(T_0, \frac{1}{4{\tilde C}^2e^{2\tilde CN_0}(1+N_1^2)^4(1+N_0)^4}\right).
\]
We note that $N_0$, and $T_0$ depend only on $\nrm{\phi}_{X^0}$.

From Proposition \ref{existence of a smooth solution} again, there exists $\delta=\delta(\nrm{u(T_{s'})}_{X^{s'}})>0$ and the solution $u\in C([0, T_{s'}+ \delta]; X^{s'}(\br))$.
We take
\[
\tilde T_0^{(0, 1)}= \min(T_{s'} +\delta, \tilde T_0), \quad
T_0^{(0, 1)} = \min(T_{s'} +\delta,  T_0), \quad
T_1^{(1)} = \min(\delta,  T_1).
\]

In the following, we write $\tilde \tau_0=\tilde T_0^{(0, 1)}$, $\tau_0=T_0^{(0, 1)}$, and $\tau_1= T_1^{(1)}$ in order to simplify the notation.
Then, Lemmas \ref{apriori-1}--\ref{apriori1} yield
\begin{align*}
\nrm{u}_{L_{\tilde \tau_0}^\infty X^0\cap L_{\tilde \tau_0}^4 L_x^\infty} 
&\le C_6e^{C_6(\nrm{\phi}_{X^0}+{\tilde \tau_0}^{\frac{1}{2}}\nrm{u}_{L_{\tilde \tau_0}^\infty L_x^2}^2)} \\
&\quad\times\left(\left(1+\nrm{\phi}_{L^2}^{2}\right)\nrm{\phi}_{X^0} + {\tilde \tau_0}^{\frac{1}{2}}\left(\nrm{u}_{L_{\tilde \tau_0}^4 L_x^\infty\cap L_{\tilde \tau_0}^\infty L_x^2} + \nrm{u}_{L_{\tilde \tau_0}^4 L_x^\infty\cap L_{\tilde \tau_0}^\infty L_x^2}^{64}\right)\right),
\end{align*}
\begin{align*}
\nrm{u}_{L_{\tau_0}^\infty X^1}
&\le C_4e^{C_4(\nrm{\phi}_{X^0}+{\tau_0}^{\frac{1}{2}}\nrm{u}_{L_{\tau_0}^\infty L_x^2}^2)}\left(1+\nrm{u}_{L_{\tau_0}^\infty L_x^2}\right) \\
&\quad \times\left(\left(1 + \nrm{\phi}_{L^2}\right)\nrm{\phi}_{X^1} + {\tau_0}^\frac{1}{2}\left(1 + \nrm{u}_{L_{\tau_0}^\infty L_x^2}\right)\left(1+\nrm{u}_{L_{\tau_0}^4 L_x^\infty}\right)^2\nrm{u}_{L_{\tau_0}^\infty H_x^1}\right),
\end{align*}
\begin{align*}
&\nrm{u(\cdot + T_{s'})}_{L_{\tau_1}^\infty X^{s'}} \\
&\le C_5e^{C_5(\nrm{u(T_{s'})}_{X^0}+{\tau_1}^{\frac{1}{2}}\nrm{u(\cdot + T_{s'})}_{L_{\tau_1}^\infty L_x^2}^2)}\left(1+\nrm{u(\cdot + T_{s'})}_{L_{\tau_1}^\infty H_x^1}^2\right) \\
&\quad \times\left(\left(1+\nrm{u(T_{s'})}_{H^1}^2\right)\nrm{u(T_{s'})}_{X^{s'}} \right. \\
&\quad\quad \left.+ {\tau_1}^{\frac{1}{2}}\left(1+\nrm{u(\cdot + T_{s'})}_{L_{\tau_1}^\infty H_x^1}^2\right)\left(1 + \nrm{u(\cdot + T_{s'})}_{L_{\tau_1}^4 L_x^\infty}\right)^2\nrm{u(\cdot + T_{s'})}_{L_{\tau_1}^\infty H_x^{s'}}\right).
\end{align*}
A continuity argument yields
\[
\nrm{u}_{L_{\tilde T_0^{(0, 1)}}^\infty X^0\cap L_{\tilde t}^4 L_x^\infty} 
\le N_0.
\]
Thus, we obtain
\[
\nrm{u}_{L_{T_0^{(0, 1)}}^\infty X^1}
\le N_1.
\]
Moreover, it holds that
\[
\nrm{u(\cdot + T_{s'})}_{L_{T_1^{(1)}}^\infty X^{s'}} 
\le \tilde Ce^{\tilde CN_0}(1+N_1^2)^2(1+\nrm{u(T_{s'})}_{X^{s'}})
=: N_{{s'}, 0}.
\]
Since we have $\nrm{u(T_{s'})}_{X^{s'}}\le M$, we may assume $\delta$ depends on $M$.
By taking
\[
\tilde T_0^{(0, 2)}= \min(T_{s'} +2\delta, \tilde T_0), \quad
T_0^{(0, 2)} = \min(T_{s'} +2\delta,  T_0), \quad
T_1^{(2)} = \min(2\delta,  T_1)
\]
and repeating the continuity argument, we obtain
\[
\nrm{u}_{L_{\tilde T_0^{(0, 2)}}^\infty X^0\cap L_{\tilde t}^4 L_x^\infty} 
\le N_0, \quad
\nrm{u}_{L_{T_0^{(0, 2)}}^\infty X^1}
\le N_1, \quad
\nrm{u(\cdot + T_{s'})}_{L_{T_1^{(2)}}^\infty X^{s'}} 
\le N_{s, 0}.
\]
Similarly, by taking
\[
\tilde T_0^{(0, k)}= \min(T_{s'} +k\delta, \tilde T_0), \quad
T_0^{(0, k)} = \min(T_{s'} +k\delta,  T_0), \quad
T_1^{(k)} = \min(k\delta,  T_1),
\]
we have
\[
\nrm{u}_{L_{\tilde T_0^{(0, k)}}^\infty X^0\cap L_{\tilde t}^4 L_x^\infty} 
\le N_0, \quad
\nrm{u}_{L_{T_0^{(0, k)}}^\infty X^1}
\le N_1, \quad
\nrm{u(\cdot + T_{s'})}_{L_{T_1^{(k)}}^\infty X^{s'}} 
\le N_{{s'}, 0}
\]
for any $k\in\bn$.
Taking $k$ as $k\delta\ge T_1$, we obtain
\[
\nrm{u}_{L_{\tilde T_0 + T_1}^\infty X^0\cap L_{\tilde t}^4 L_x^\infty} 
\le N_0, \quad
\nrm{u}_{L_{T_0 + T_1}^\infty X^1}
\le N_1, \quad
\nrm{u(\cdot + T_{s'})}_{L_{T_1}^\infty X^{s'}} 
\le N_{{s'}, 0}.
\]
A similar argument yields
\[
\nrm{u}_{L_{\tilde T_0^{(\ell, k)}}^\infty X^0\cap L_{\tilde t}^4 L_x^\infty} 
\le M, \quad
\nrm{u}_{L_{T_0^{(\ell, k)}}^\infty X^1}
\le N, \quad
\nrm{u(\cdot + T_{s'} + \ell T_1)}_{L_{T_1^{(k)}}^\infty X^{s'}} 
\le N_{{s'}, \ell},
\]
where $\tilde T_0^{(\ell, k)}$, $T_0^{(\ell, k)}$, and $N_{{s'}, \ell}$ is defined by
\[
\tilde T_0^{(\ell, k)}= \min(T_{s'} + \ell T_1 + k\delta, \tilde T_0), \quad
T_0^{(\ell, k)} = \min(T_{s'} + \ell T_1 + k\delta, T_0),
\]
\[
N_{{s'}, \ell} = \tilde Ce^{\tilde CN_0}(1+N_1^2)^2(1+\nrm{u(T_{s'} + \ell T_1)}_{X^{s'}}).
\]
Therefore, repeating this argument until $T_{s'} + \ell T_1\ge T_0$, we obtain
$u\in C([0, T_0]; X^{s'}(\br))$ is a solution to \eqref{nls}.
\end{proof}

We can also obtain an a priori estimate.
\begin{prop}\label{apriori}
Let $s'>\frac{3}{2}$, $0\le s\le s'$, $\phi\in X^{s'}(\br)$, and $u\in C([0, T_0]; X^{s'}(\br))$ be a solution to \eqref{nls} obtained in Proposition \ref{existence time}.
Then, there exist $T^*=T^*(s, \nrm{\phi}_{X^s})\in (0, T_0]$ and $C=C(s)>0$ such that
\[
\nrm{u}_{L_{T^{*}}^\infty X^s\cap L_{T^{*}}^4 L_x^\infty}
\le C e^{C \nrm{\phi}_{X^s}}\left(1+\nrm{\phi}_{X^s}^{2([s]+1)^2}\right)\nrm{\phi}_{X^s}.
\]
\end{prop}
\begin{proof}
We take
\[
N_s = C e^{C \nrm{\phi}_{X^s}}\left(1+\nrm{\phi}_{X^s}^{2([s]+1)^2}\right)\nrm{\phi}_{X^s}, \quad
T^* =\min\left(T_0, \frac{\nrm{\phi}_{X^s}^2}{N_s^4}, \frac{(1+\nrm{\phi}_{X^s}^{2([s]+1)^2})^2\nrm{\phi}_{X^s}^2}{(N_s + N_s^{4([s]+2)^4})^2}\right),
\]
where $C = 3C_6(s)$ and $C_6(s)$ is taken in Lemma \ref{apriori1}.
Then, we have from Lemma \ref{apriori1} that
\begin{align*}
&\nrm{u}_{L_{T^*}^\infty X^s\cap L_{T^*}^4 L_x^\infty} \\
&\le C_6e^{C_6(\nrm{\phi}_{X^0}+{T^*}^{\frac{1}{2}}\nrm{u}_{L_{T^*}^\infty L_x^2}^2)} \\
&\quad\times\left(\left(1+\nrm{\phi}_{H^s}^{2([s]+1)^2}\right)\nrm{\phi}_{X^s} + {T^*}^{\frac{1}{2}}\left(\nrm{u}_{L_{T^*}^4 L_x^\infty\cap L_{T^*}^\infty H_x^s} + \nrm{u}_{L_{T^*}^4 L_x^\infty\cap L_{T^*}^\infty H_x^s}^{4([s]+2)^4}\right)\right).
\end{align*}
Therefore, a continuity argument yields
\[
\nrm{u}_{L_{T^*}^\infty X^s\cap L_{T^*}^4 L_x^\infty}
\le N_s.
\]
This completes the proof.
\end{proof}

Next, we consider the estimate for the difference between two solutions.
\begin{lem}\label{diff2}
Let $s'>\frac{1}{2}$, $0\le s\le s'$, $T\in (0, 1)$, and $u_1, u_2\in C([0, T]; X^{s'}(\br))$ be solutions to \eqref{nls} with initial data $\phi_1$ and $\phi_2$, respectively.
Then, there exists a constant $C=C(s, \|u_1\|_{L_T^\infty X^s \cap L_T^4L_x^\infty}, \|u_2\|_{L_T^\infty X^s \cap L_T^4L_x^\infty})>0$ such that
\begin{align*}
\nrm{u_1-u_2}_{L_T^\infty X^s \cap L_T^4L_x^\infty}
\le C\left(\nrm{\phi_1-\phi_2}_{X^{s}} +T^{\frac{1}{2}}\nrm{u_1-u_2}_{L_T^\infty X^s \cap L_T^4L_x^\infty}\right) .
\end{align*} 
\end{lem}
\begin{proof}
For $u_1$ and $u_2$, we define $\Lambda_1$ and $\Lambda_2$ by \eqref{primitives}.
We only consider the case $0\le s\le \fr{1}{2}$, becasuse we can prove the case $s>\frac{1}{2}$ similarly.
From Proposition \ref{prop gauge}, we have
\begin{equation}\label{gauge diff}
\begin{aligned} 
&\cl \left(e^{-\Lambda_1}u_1 - e^{-\Lambda_2}u_2\right) \\
%&= e^{-\Lambda_1}\left(\abs{\mu}^2\abs{u_1}^2u_1 + \fr{1}{2}\mu\left(2\bar \lambda + \mu\right)\abs{u_1}^2\bar u_1\right) \\
%&\quad- e^{-\Lambda_2}\left(\abs{\mu}^2\abs{u_2}^2u_2 + \fr{1}{2}\mu\left(2\bar \lambda + \mu\right)\abs{u_2}^2\bar u_2\right) \\
&= e^{-\Lambda_1}\left(\abs\mu^2 \left(\abs{u_1}^2u_1 - \abs{u_2}^2u_2\right) + \fr{1}{2}\mu\left(2\bar \lambda + \mu\right) \left(\abs{u_1}^2\bar u_1 - \abs{u_2}^2\bar u_2\right)   \right) \\
&\quad +\left(1-e^{\Lambda_1 - \Lambda_2}\right)e^{-\Lambda_1}\left(\abs\mu^2 \abs {u_2}^2 u_2 + \fr{1}{2}\mu\left(2\bar \lambda + \mu\right)\abs {u_2}^2 \bar u_2\right).
\end{aligned}
\end{equation}

First, we prove the case $s=0$. 
Lemma \ref{strest}, \eqref{pr diff2 ineq}, and \eqref{gauge diff} yield
\begin{align*} 
&\nrm{e^{-\Lambda_1}u_1-e^{-\Lambda_2}u_2}_{L_{T}^\infty L_x^2\cap L_{{T}}^4L_x^\infty}  \\
&\le C(\snrm{\phi_1}_{X^0}, \snrm{\phi_2}_{X^0})\nrm{\phi_1-\phi_2}_{X^0}\\
&\quad+ C(\snrm{u_1}_{L_{T}^\infty X^0\cap L_{T}^4 L_x^\infty}, \snrm{u_2}_{L_{T}^\infty X^0\cap L_{T}^4 L_x^\infty}){T}^{\fr{1}{2}}\nrm{u_1-u_2}_{L_{T}^\infty X^0} \\
&\le C(\snrm{u_1}_{L_{T}^\infty X^0\cap L_{T}^4 L_x^\infty}, \snrm{u_2}_{L_{T}^\infty X^0\cap L_{T}^4 L_x^\infty})\left(\nrm{\phi_1-\phi_2}_{X^0} + {T}^{\fr{1}{2}}\nrm{u_1-u_2}_{L_{T}^\infty X^0}\right).
\end{align*} 
Also, a similar argument as Lemma \ref{smooth apriori step4} yields
\begin{equation}\label{pr diff20}
\begin{aligned}
&\nrm{\int_\minfty^x(u_1(t, y)-u_2(t, y))\,dy}\txx{\infty} + \nrm{\Lambda_1-\Lambda_2}\txx{\infty} \\
&\lesssim \nrm{\phi_1-\phi_2}_{X^0} + {T}^{\fr{1}{2}}\left(\nrm{u_1}\tx{\infty}{2}^2 +\nrm{u_2}\tx{\infty}{2}^2\right)\nrm{u_1-u_2}\tx{\infty}{2},
\end{aligned}
\end{equation}
especially,
\begin{align*} 
&\nrm{e^{\pm(\Lambda_1(t)-\Lambda_2(t))}-1}\txx{\infty} \\
&\le C(\snrm{u_1}_{L_{T}^\infty X^0\cap L_{T}^4 L_x^\infty}, \snrm{u_2}_{L_{T}^\infty X^0\cap L_{T}^4 L_x^\infty})\left(\nrm{\phi_1-\phi_2}_{X^0} + {T}^{\fr{1}{2}}\nrm{u_1-u_2}\tx{\infty}{2}\right).
\end{align*} 
Thus, we obtain
\begin{align*} 
&\nrm{u_1-u_2}_{L_{T}^\infty L_x^2\cap L_{T}^4L_x^\infty} \\
&\le \nrm{e^{\Lambda_1}}\txx{\infty}\nrm{e^{-\Lambda_1}u_1-e^{-\Lambda_2}u_2}_{L_{T}^\infty L_x^2\cap L_{T}^4L_x^\infty}
+ \nrm{u_2}_{L_{T}^\infty L_x^2\cap L_{T}^4L_x^\infty}\nrm{e^{\Lambda_1-\Lambda_2}-1}\txx{\infty} \\
&\le C(\snrm{u_1}_{L_{T}^\infty X^0\cap L_{T}^4L_x^\infty}, \snrm{u_2}_{L_{T}^\infty X^0\cap L_{T}^4L_x^\infty} ) 
\left(\nrm{\phi_1-\phi_2}_{X^0}+{T}^{\fr{1}{2}}\nrm{u_1-u_2}_{L_{T}^\infty X^0\cap L_{T}^4L_x^\infty}\right).
\end{align*}
Hence, \eqref{pr diff20} yields
\begin{equation}\label{pr diff21}
\begin{aligned}
&\nrm{u_1-u_2}_{L_{T}^\infty X^0\cap L_{T}^4L_x^\infty} \\
&\le C(\snrm{u_1}_{L_{T}^\infty X^0\cap L_{T}^4L_x^\infty}, \snrm{u_2}_{L_{T}^\infty X^0\cap L_{T}^4L_x^\infty})\left(\nrm{\phi_1-\phi_2}_{X^0}+{T}^{\fr{1}{2}}\nrm{u_1-u_2}_{L_{T}^\infty X^0\cap L_{T}^4L_x^\infty}\right).
\end{aligned}
\end{equation}
This completes the proof of the case $s=0$.

Next, we consider the case $0<s\le \fr{1}{2}$.
We consider the estimate for the $L_{T}^\infty H_x^s$-norm.
We have
\begin{equation}\label{pr diff22}
\nrm{u_1-u_2}_{L_{T}^\infty H_x^s}
\le \nrm{e^{\Lambda_1}\left(e^{-\Lambda_1}u_1-e^{-\Lambda_2}u_2\right)}_{L_{T}^\infty H_x^s} + \nrm{\left(e^{\Lambda_1-\Lambda_2}-1\right)u_2}_{L_{T}^\infty H_x^s}.
\end{equation}
Here, \eqref{est of exp f2} yields
\begin{equation}\label{pr diff23}
\nrm{e^{\Lambda_1}\left(e^{-\Lambda_1}u_1-e^{-\Lambda_2}u_2\right)}_{L_{T}^\infty H_x^s}
\le C(s, \|u_1\|_{L_{T}^\infty X^s})\nrm{e^{-\Lambda_1}u_1-e^{-\Lambda_2}u_2}_{L_{T}^\infty H_x^s}.
\end{equation}
Also, we obtain from Lemma \ref{seki} and \eqref{pr diff2 ineq} that
\begin{equation}\label{pr diff24}
\nrm{\left(e^{\Lambda_1-\Lambda_2}-1\right)u_2}_{L_{T}^\infty H_x^s}
\le C(s, \|u_1\|_{L_{T}^\infty X^s}, \|u_2\|_{L_{T}^\infty X^s})\nrm{u_1-u_2}_{L_{T}^\infty X^0}.
\end{equation}
We consider the ${L_{T}^\infty H_x^s}$-norm of
$e^{-\Lambda_1}u_1-e^{-\Lambda_2}u_2$.
From Lemma \ref{seki}, \eqref{est of exp f2}, and \eqref{pr diff2 ineq}, we obtain
\begin{align*}
\nrm{e^{-\Lambda_1(0)}\phi_1-e^{-\Lambda_2(0)}\phi_2}_{H^s}
&\le \nrm{\left(1-e^{\Lambda_1(0)-\Lambda_2(0)}\right)e^{-\Lambda_1(0)}\phi_1}_{H^s} 
+ \nrm{e^{-\Lambda_2(0)}(\phi_1-\phi_2)}_{H^s} \\
&\le C(s, \|\phi_1\|_{X^s}, \|\phi_2\|_{X^s})\nrm{\phi_1-\phi_2}_{X^s}
\end{align*}
and
\begin{align*} 
&\nrm{\left(1-e^{\Lambda_1-\Lambda_2}\right)e^{-\Lambda_1}\abs{u_2}^2u_2}_{L_{T}^2 H_x^s} \\
&\lesssim \nrm{\left(1-e^{\Lambda_1-\Lambda_2}\right)e^{-\Lambda_1}}_{L_{T, x}^\infty\cap L_{T}^\infty \dot H_x^{[s]+1}}\nrm{\abs{u_2}^2u_2}_{L_{T}^2 H_x^s} \\
&\le C(s, \|u_1\|_{L_{T}^\infty X^s}, \|u_2\|_{L_{T}^\infty X^s})\nrm{\abs{u_2}^2u_2}_{L_{T}^2 H_x^s}\nrm{u_1-u_2}_{L_{T}^\infty X^0}.
\end{align*} 
Also, \eqref{est of exp f2} yields
\[
\nrm{e^{-\Lambda_1}\left(\abs{u_1}^2u_1-\abs{u_2}^2u_2\right)}_{L_T^2 H_x^s}
\le C(s, \|u_1\|_{L_{T}^\infty X^s})\nrm{\abs{u_1}^2u_1-\abs{u_2}^2u_2}_{L_{T}^2 H_x^s}.
\]
From Lemma \ref{fracLeib}, we obtain
\[
\nrm{\abs{u_1}^2u_1-\abs{u_2}^2u_2}_{L_{T}^2 H_x^s}
\le C(s, \|u_1\|_{L_{T}^\infty H_x^s\cap L_{T}^4L_x^\infty}, \|u_2\|_{L_{T}^\infty H_x^s\cap L_{T}^4L_x^\infty})\nrm{u_1-u_2}_{L_{T}^\infty H_x^s\cap L_{T}^4L_x^\infty}
\]
and
\[
\nrm{\abs{u_2}^2u_2}_{L_{T}^2 H_x^s}
\le C(s, \|u_2\|_{L_{T}^\infty H^s\cap L_{T}^4L_x^\infty}).
\]
Similar estimates hold when we replace $u_1$ and $u_2$ with $\bar u_1$ and $\bar u_2$, respectively.
Thus, Lemma \ref{strest} and \eqref{gauge diff} yield
\begin{equation}\label{pr diff25}
\begin{aligned}
&\nrm{e^{-\Lambda_1}u_1-e^{-\Lambda_2}u_2}_{L_{T}^\infty H_x^s} \\
&\lesssim \nrm{e^{-\Lambda_1(0)}\phi_1-e^{-\Lambda_2(0)}\phi_2}_{H^s} \\
&\quad + \nrm{e^{-\Lambda_1}\left(\abs{u_1}^2u_1 - \abs{u_2}^2u_2\right)}_{L_{T}^1 H_x^s}
+ \nrm{e^{-\Lambda_1}\left(\abs{u_1}^2\bar u_1 - \abs{u_2}^2\bar u_2\right)}_{L_{T}^1 H_x^s} \\
&\quad +\nrm{\left(1-e^{\Lambda_1-\Lambda_2}\right)e^{-\Lambda_1}\abs{u_2}^2u_2}_{L_{T}^1 H_x^s}
+\nrm{\left(1-e^{\Lambda_1-\Lambda_2}\right)e^{-\Lambda_1}\abs{u_2}^2\bar u_2}_{L_{T}^1 H_x^s} \\
&\le C(s, \|u_1\|_{L_{T}^\infty X^s\cap L_{T}^4L_x^\infty}, \|u_2\|_{L_{T}^\infty X^s\cap L_{T}^4L_x^\infty}) \left(\nrm{\phi_1-\phi_2}_{X^{s}} + {T}^{\fr{1}{2}}\nrm{u_1-u_2}_{L_{T}^\infty X^s\cap L_{T}^4L_x^\infty}\right).
\end{aligned}
\end{equation} 
Hence, \eqref{pr diff21}--\eqref{pr diff25} yield
\begin{align*}
&\nrm{u_1-u_2}_{L_{T}^\infty X^s\cap L_{T}^4 L_x^\infty} \\
&\le C(s, \|u_1\|_{L_{T}^\infty X^s\cap L_{T}^4L_x^\infty}, \|u_2\|_{L_{T}^\infty X^s\cap L_{T}^4L_x^\infty})\left(\nrm{\phi_1-\phi_2}_{X^s} + {T}^{\frac{1}{2}}\nrm{u_1-u_2}_{L_{T}^\infty X^s\cap L_{T}^4 L_x^\infty}\right).
\end{align*}
This completes the proof.
\end{proof}
From Lemma \ref{diff2}, we immediately obtain the uniqueness of the solution when $s>\frac{1}{2}$.
In particular, for solutions obtained in Proposition \ref{existence time}, we have the following estimate.
\begin{prop}\label{diff3}
Let $s'>\frac{3}{2}$, $0\le s\le s'$, $\phi\in X^{s'}(\br)$, and $u_1\in C([0, T_{0,1}]; X^{s'}(\br))$ and $u_2\in C([0, T_{0,2}]; X^{s'}(\br))$ be solutions to \eqref{nls} with initial data $\phi_1$ and $\phi_2$ which are obtained in Proposition \ref{existence time}.
Then, there exist $C=C(s, \|\phi_1\|_{X^s}, \|\phi_2\|_{X^s})>0$ and $T^*=T^*(s, \nrm{\phi_1}_{X^s}, \nrm{\phi_2}_{X^s})\in (0, \min(T_{0,1}, T_{0,2})]$ such that
\begin{align*}
\nrm{u_1-u_2}_{L_{T^*}^\infty X^s \cap L_{T^*}^4L_x^\infty}
\le C\nrm{\phi_1-\phi_2}_{X^{s}}.
\end{align*} 
\end{prop}
\begin{proof}
From Proposition \ref{apriori}, there exist $C=C(s, \|\phi_1\|_{X^s})$ and $T_1^{*}=T_1^{*}(s, \nrm{\phi_1}_{X^s})\in (0, T_{0, 1})$ such that
\[
\nrm{u_1}_{L_{T_1^*}^\infty X^s\cap L_{T_1^*}^4 L_x^\infty}
\le C.
\]
We obtain a similar estimate for $u_2$.
From Lemma \ref{diff2}, there exists $C=C(s, \|\phi_1\|_{X^s}, \|\phi_2\|_{X^s})>0$ such that
\[
\nrm{u_1-u_2}_{L_T^\infty X^s \cap L_T^4L_x^\infty}
\le C\left(\nrm{\phi_1-\phi_2}_{X^{s}} +T^{\frac{1}{2}}\nrm{u_1-u_2}_{L_T^\infty X^s \cap L_T^4L_x^\infty}\right) .
\]
for any $T\in (0, \min({T_1^*}, {T_2^*})]$.
Taking $T^*=T^*(s, \snrm{\phi_1}_{X^s}, \snrm{\phi_2}_{X^s})\in (0, \min({T_1^*}, {T_2^*})]$ as
$C{T^*}^{\frac{1}{2}}\le \frac{1}{2}$, we obtain the desired estimate.
\end{proof}

\subsection{Proof of well-posedness in $X^{s}(\br)$ with $s\ge 0$}
In this subsection, we prove the local well-posedness of \eqref{nls} in $X^s(\br)$ with $s\ge 0$. 
First, we prove the case $s>\frac{1}{2}$.
We omit the proof of the uniqueness because we can prove it by using Lemma \ref{diff2} and $\snrm{u}_{L_T^4L_x^\infty} \lesssim T^{\frac{1}{4}}\snrm{u}_{L_T^\infty H_x^s}$ for $s>\frac{1}{2}$.
\begin{proof}[Proof of Theorem \ref{WP in $X^s$} of the case $s>\frac{1}{2}$]
(\textit{Existence of a solution})
Let $\phi\in X^s(\br)$. 
Then, we obtain $\phi_n:=P_{\le n}\phi\in X^{s+1}(\br)$ and $\phi_n\to \phi$ in $X^s(\br)$. 
Indeed, we have
\[
\nrm{P_{\ge n+1}\phi}_{X^s} = \nrm{P_{\ge n+1}\phi}_{H^s} + \nrm{\px^{-1}P_{\ge n+1}\phi}_{L^\infty}\to 0.
\]
From Proposition \ref{existence time}, there exist $T=T(\nrm{\phi}_{X^0})>0$ and solutions $u_n\in \ctx{s+1}$ to \eqref{nls} with initial data $\phi_n$ for any large $n$.
Also, Proposition \ref{diff3} yields that there exist $T^{*}=T^{*}(s, \nrm{\phi}_{X^s})\in (0, T]$ and $C=C(s, \|\phi\|_{X^s})$ such that for any large $m$ and $n$, it holds that
\[
\nrm{u_m-u_n}_{L_{T^{*}}^\infty X^s}
\le C\nrm{\phi_m-\phi_n}_{X^s}.
\]
Thus, there exists $u$ which is a limit of $u_n$ in $C([0, {T^{*}}]; X^s(\br))$, and an apprximation argument yields that $u$ is a solution to \eqref{nls}.

(\textit{Continuity of the flow map})
Let $\phi, \phi_n\in X^s(\br)$, $\nrm{\phi}_{X^s}, \nrm{\phi}_{X^s}\le R$, and $\phi_n\to \phi$ in $X^s(\br)$. 
Then, there exist $ T^*=  T^*(s, R)$ and solutions $u, u_n\in C([0,  T^*]; X^s(\br))$ to \eqref{nls} with initial data $\phi, \phi_n$, respectively. 
From Proposition \ref{diff3}, there exist $C = C (s, R)>0$ and $\tilde T^* = \tilde T^* (s, R)\in (0, T^*]$ such that
\[
\nrm{u_n-u}_{L_{\tilde T^*}^\infty X^s}
\le C\nrm{\phi_n-\phi}_{X^s}.
\]
Thus, we obtain the Lipschitz continuity of the flow map.
\end{proof}
Next, we prove the case $0\le s\le \frac{1}{2}$.
In the following, we denote
\[
B_{X^s}(R):=\{\phi\in X^s(\br)\mid\nrm{\phi}_{X^s}\le R\}.
\]
\begin{proof}[Proof of Theorem \ref{WP in $X^s$} of the case $0\le s\le \frac{1}{2}$]
(\textit{Existence of the extension})
Let $R>0$ and $\phi\in B_{X^s}(R)$. 
Then, we obtain $\phi_n:=\frac{\|\phi\|_{X^s}}{\|P_{\le n}\phi\|_{X^s}}P_{\le n}\phi\in X^1(\br)\cap B_{X^s}(R)$ and $\phi_n\to \phi$ in $X^s(\br)$. 
Since $\nrm{\phi_n}_{X^0}\le R$, Proposition \ref{existence time} that there exist $T=T(R)>0$ and solutions $u_n\in \ctx{1}$ to \eqref{nls} with initial data $\phi_n$.
Also, Proposition \ref{diff3} yields that there exist $T^{*}=T^{*}(s, R)\in (0, T]$ and $C=C(R)$ such that for any $m$ and $n$, it holds that
\[
\nrm{u_m-u_n}_{L_{T^{*}}^\infty X^s\cap L_{T^{*}}^4L_x^\infty}
\le C\nrm{\phi_m-\phi_n}_{X^s}.
\]
Thus, there exists $u$ which is a limit of $u_n$ in $C([0, {T^{*}}]; X^s(\br))\cap L_{T^{*}}^4L_x^\infty$, and an apprximation argument yields that $u$ satisfies \eqref{sol}.

(\textit{Well-definedness of the extension})
Let $R>0$ and $\phi\in B_{X^s}(R)$. 
We also assume $\phi_1^{(n)}, \phi_2^{(n)}\in X^1(\br)\cap B_{X^s}(R)$,  $\phi_1^{(n)}, \phi_2^{(n)}\to \phi$, and $u_1^{(n)}$ and $u_2^{(n)}$ are solutions of \eqref{nls} with initial data $\phi_1^{(n)}$ and $\phi_2^{(n)}$, respectively.
Moreover, we assume $u_1^{(n)}\to u_1$ and $u_2^{(n)}\to u_2$ in $C([0, {T^{*}}]; X^s(\br))\cap L_{T^{*}}^4L_x^\infty$, respectively.
From Proposition \ref{diff3}, by taking ${T^{*}}={T^{*}}(s, R)$ smaller if necessary, we have
\[
\nrm{u_1^{(n)}-u_2^{(n)}}_{L_{T^{*}}^\infty X^s\cap L_{T^{*}}^4L_x^\infty}
\le C(R)\nrm{\phi_1^{(n)}-\phi_2^{(n)}}_{X^s}
\]
for all $n$.
Taking $n\to\infty$, we obtain $u_1=u_2$ in $C([0, {T^{*}}]; X^s(\br))\cap L_{T^{*}}^4L_x^\infty$.

(\textit{Lipschitz continuity of the extended map})
Let $R>0$, $\phi, \phi_n\in B_{X^s}(R)$, and $\phi_n\to \phi$ in $X^s(\br)$. 
We also assume $\phi^{(k)}, \phi_n^{(k)}\in X^1(\br)\cap B_{X^s}(R)$,  $\phi^{(n)}\to \phi, \phi_n^{(k)}\to \phi_n$ in $X^s(\br)$.
Moreover, we assume that  $u^{(k)}$ and $u_n^{(k)}$ are solutions of \eqref{nls} with initial data $\phi^{(k)}$ and $\phi_n^{(k)}$, respectively, and $u^{(k)}\to u$ and $u_n^{(k)}\to u_n$ in $C([0, {T^{*}}]; X^s(\br))\cap L_{T^{*}}^4L_x^\infty$, respectively.
From Proposition \ref{diff3}, by taking ${T^{*}}={T^{*}}(s, R)$ smaller if necessary, we have
\[
\nrm{u^{(k)}-u_n^{(k)}}_{L_{T^{*}}^\infty X^s\cap L_{T^{*}}^4L_x^\infty}
\le C(s, R)\nrm{\phi^{(k)}-\phi_n^{(k)}}_{X^s}
\]
for all $k$.
Taking $k\to\infty$, we obtain 
\[
\nrm{u-u_n}_{L_{T^{*}}^\infty X^s\cap L_{T^{*}}^4L_x^\infty}
\le C(s, R)\nrm{\phi-\phi_n}_{X^s}.
\]
Thus, we obtain the Lipschitz continuity of the extended map.
\end{proof}

\section{The special coefficient case}
Unless otherwise noted, we assume 
\[
2\lambda+\bar \mu=0
\]
in this section.
 We remark that $\operatorname{Re}\left(2\lambda z+\mu \bar z\right)=0$ for all $z\in\bc$.
In particular, we have $\operatorname{Re}\Lambda=0$. 

From Proposition \ref{prop gauge}, we have
\begin{equation}\label{gauge special case}
\cl\left(e^{-\Lambda}u\right) =\abs{\mu}^2\abs{e^{-\Lambda}u}^2e^{-\Lambda}u.
\end{equation}
Therefore, $e^{-\Lambda}u$ is a solution to the cubic nonlinear Schr\"odinger equation. 
In particular, $e^{-\Lambda}u$ has conservation laws such as
\begin{equation}\label{conservation law1}
\nrm{u(t)}_{L^2} = \nrm{e^{-\Lambda(t)}u(t)}_{L^2} =  \nrm{e^{-\Lambda(0)}u(0)}_{L^2} = \nrm{u(0)}_{L^2},
\end{equation}
\begin{equation}\label{conservation law2}
\nrm{\px\left(e^{-\Lambda}u\right)(t)}_{L^2}^2 + \abs{\mu}^2\nrm{e^{-\Lambda(t)}u(t)}_{L^4}^4 
= \nrm{\px\left(e^{-\Lambda}u\right)(0)}_{L^2}^2 + \abs{\mu}^2\nrm{e^{-\Lambda(0)}u(0)}_{L^4}^4.
\end{equation}
First, we prove Theorem \ref{special case $X^s$} by using \eqref{conservation law1}.
\begin{prop}\label{existence time2}
We do not assume $2\lambda + \bar \mu=0$.
Let $s'>\frac{1}{2}$.
Then, there exist $T_0=T_0(\nrm{\phi}_{X^0})\in (0, 1)$ and a solution $u\in C([0, T_0]; X^{s'}(\br))$.
Also, for any $0\le s\le \frac{1}{2}$, there exist $T=T(s, \nrm{\phi}_{X^0})\in (0, T_0]$ and $C(s)>0$ such that
\[
\nrm{u}_{L_T^\infty X^s\cap L_T^4 L_x^\infty}
\le Ce^{C\nrm{\phi}_{X^0}}(1+\nrm{\phi}_{X^0}^2)\nrm{\phi}_{X^s}.
\] 
\end{prop}
\begin{proof}
The proof of the first half is similar to the proof of Proposition \ref{existence time}.
Hence, we provide the proof of the second half.
From the definition of $T_0$ in \eqref{T0}, the solution $u\in C([0, T_0]; X^{s}(\br))$ satisfies
\[
\nrm{u}_{L_{T_0}^\infty X^0\cap L_{T_0}^4 L_x^\infty} 
\le N_0,
\]
where $N_0$ is defined by \eqref{T0}.
Since Lemma \ref{apriori0} yields
\begin{align*}
&\nrm{u}_{L_{T}^\infty X^s\cap L_{T}^4 L_x^\infty} \\
&\le C_5e^{C_5(\nrm{\phi}_{X^0}+{T}^{\frac{1}{2}}\nrm{u}_{L_{T}^\infty L_x^2}^2)}\left(1+\nrm{u}_{L_{T}^\infty L_x^2}\right) \\
&\quad \times\left(\left(1+\nrm{\phi}_{L^2}\right)\nrm{\phi}_{X^s} + {T}^{\frac{1}{2}}\left(1+\nrm{u}_{L_{T}^\infty L_x^2}\right)\left(1 + \nrm{u}_{L_{T}^4 L_x^\infty}\right)^2\nrm{u}_{L_{T}^\infty H_x^s}\right)
\end{align*}
for any $T\in (0, T_0]$, a continuity argument yields
\[
\nrm{u}_{L_{T}^\infty X^s\cap L_{T}^4 L_x^\infty}
\le Ce^{C\nrm{\phi}_{X^0}}(1+\nrm{\phi}_{X^0}^2)\nrm{\phi}_{X^s},
\]
where 
\[
C = 3C_5(s), \quad T=\min\left(T_0, \frac{\nrm{\phi}_{X^0}^2}{N_0^4}, \frac{1}{C^2e^{2C\nrm{\phi}_{X^0}}(1+N_0)^8}\right).
\]
This completes the proof.
\end{proof}

\begin{proof}[Proof of Theorem \ref{special case $X^s$}]
Lemma \ref{smooth apriori step4} and \eqref{conservation law1} yield 
\[
\nrm{\int_\minfty^x u(t, y)\, dy}_{L_x^\infty}
\lesssim (1+\abs{t})\nrm{\phi}_{X^0} + \abs{t}^{\fr{1}{2}}\nrm{\phi}_{L^2}
\]
for all $t$.
Thus, we have
\begin{equation}\label{pr theorem2}
\nrm{u(t)}_{X^0} \lesssim (1+\abs{t})\nrm{\phi}_{X^0} + \abs{t}^{\fr{1}{2}}\nrm{\phi}_{L^2}.
\end{equation}
Hence, Theorem \ref{WP in $X^s$}, Proposition \ref{existence time2}, and \eqref{pr theorem2} yield Theorem \ref{special case $X^s$}.
\end{proof}
The conservation laws \eqref{conservation law1} and \eqref{conservation law2} and the Gagliardo-Nirenberg inequality yield that there exists $C=C(\nrm{\phi}_{H^1})>0$ such that
\begin{equation}\label{aprori $H^1$}
\nrm{u(t)}_{H^{1}}\le C
\end{equation}
for all $t$. 
In the following, we consider the well-posedness in $H^{s}(\br)$ for $s\ge 1$.
\subsection{Estimate for difference in $L^2(\br)$}
When we consider the well-posedness in the Sobolev space, we cannot use the decomposition as $u_1-u_2 = e^{\Lambda_1}(e^{-\Lambda_1}u_1-e^{-\Lambda_2}u_2) +(e^{\Lambda_1-\Lambda_2}-1)u_2$ since we have to use the $L^\infty_{T, x}$-norm of the primitives of solutions.
However, the condition $2\lambda+\bar \mu=0$ enjoys good structures as we saw at the beginning of this section. 
Also, we can obtain the estimate for the $L^2$-norm of the difference between two solutions.
\begin{prop}\label{diff energy}
Let $s\ge 1$ and $u_1, u_2\in \cth{s}$ be solutions to \eqref{nls} with initial data $\phi_1, \phi_2$, respectively. 
Also, we assume $\px u_1, \px u_2 \in L^1([0, T]; L^\infty(\br))$. Then, there exists  $C>0$ such that
\[
\nrm{u_1(t)-u_2(t)}_{L^2}^2\le\nrm{\phi_1-\phi_2}_{L^2}^2\exp\left(C\int_0^t\left(\nrm{\px u_1(t')}_{L^\infty} + \nrm{\px u_2(t')}_{L^\infty}\right)\,dt'\right)
\]
for any $t\in[0, T]$.
\end{prop}
\begin{proof}
Let $w(t):=u_1(t)-u_2(t)$.
Then, we have
\[
\pt w
%=-\frac{i}{2}\px^2 w-i\px \left(\lambda(u_1^2-u_2^2)+\mu(\abs{u_1}^2-\abs{u_2}^2)\right)
=-\frac{i}{2}\px^2 w-i\px \left(\lambda w^2+ (2\lambda u_2+\mu\bar u_2)w+\mu u_1\bar w\right).
\]
It holds that
\[
\abs{\int_\br\px (w^2)\bar w\, dx} +\abs{\int_\br\px (u_1\bar w)\bar w\, dx}
\lesssim \left(\nrm{\px u_1(t)}_{L^\infty} +\nrm{\px u_2(t)}_{L^\infty}\right)\nrm{w(t)}_{L^2}^2.
\]
Noting that $\overline{2\lambda u_2+\mu\bar u_2}=-(2\lambda u_2+\mu\bar u_2)$, we have
\begin{align*}
&\operatorname{Im}\left(\int_\br\px ((2\lambda u_2+\mu\bar u_2)w)\bar w\, dx\right) \\
&=\operatorname{Im}\left(\int_\br\px (2\lambda u_2+\mu\bar u_2)\abs{w}^2\, dx\right)
+\operatorname{Im}\left(\int_\br(2\lambda u_2+\mu\bar u_2)\px w\cdot\bar w\, dx\right) \\
&=\operatorname{Im}\left(\int_\br\px (2\lambda u_2+\mu\bar u_2)\abs{w}^2\, dx\right)
-\operatorname{Im}\left(\int_\br\px ((2\lambda u_2+\mu\bar u_2)w)\bar w\, dx\right).
\end{align*}
Thus, we obtain
\[
\abs{\operatorname{Im}\left(\int_\br\px ((2\lambda u_2+\mu\bar u_2)w)\bar w\, dx\right)}
\lesssim \nrm{\px u_2(t)}_{L^\infty}\nrm{w(t)}_{L^2}^2.
\]
Therefore, Gronwall's inequality yields the desired estimate.
\end{proof}
\subsection{A priori estimate using frequency envelope}
We obtained the bound of the difference in $L^{2}(\br)$. To prove the well-posedness in the Sobolev space, we use the frequency envelope.
See, for example, \cite[Definition 1]{Tao2001}, \cite[Definition 5.1]{Tao2004}.
In this paper, we add the character $\delta$-, since we have to take $\delta>0$ smaller than $s$ if $s>0$ is small.
\begin{defi}[frequency envelope]
Let $s\in\br$ and $\delta>0$. We say that a nonnegative sequence $\{c_k\}\in \ell^2$ is a $\delta$-frequency envelope in $H^s(\br)$ if it has the following properties:
\begin{align*}
&(i)\quad \sum_{k=0}^\infty c_k^2 \lesssim \nrm{\phi}_{H^s}^2, \\
&(ii)\quad c_j\le 2^{\delta\abs{j-k}}c_k, \\
&(iii)\quad \nrm{P_k\phi}_{H^s}\le c_k.
\end{align*}
\end{defi}
The a priori estimate which uses the frequency envelope is the following.
\begin{prop}\label{apriori gauge fr}
Let $s\ge \tilde s> 0$, $0<\delta<\min(\fr{1}{100}, \fr{\tilde s}{2})$, $r, R>0$, $\phi\in X^{s+1}(\br)$, $\nrm{\phi}_{H^{\tilde s}}\le r$, and $\nrm{\phi}_{H^s}\le R$. 
Also, we assume that $\{Rc_k\}$ is a $\delta$-frequency envelope of $\phi$ in $H^s(\br)$ and $u\in C([0, \infty); X^{s+1}(\br))$ is a solution to \eqref{nls}.  
Then, there exist $T^*=T^*(s, \tilde s, r)>0$ and $C=C(s)>0$ such that
\begin{equation*}
\sup_{k\in\bz_{\ge 0}}2^{sk}c_k^{-1}\nrm{P_k u}_{S_{T^*}} + \sup_{k\in\bz_{\ge 0}}2^{sk}c_k^{-1}\nrm{P_k \left(e^{\Lambda}u\right)}_{S_{T^*}}
\le C\left(R+R^{4(s+1)(s+2)}\right).
\end{equation*}
\end{prop}

To prove this proposition, we use commutator estimates for the Littlewood-Paley decomposition by Tao \cite[Lemma 2]{Tao2001}.
The following lemma is a special case of it.
In the following, $[\cdot, \cdot]$ denotes the commutator.
For $f:\br\to\bc$ and $y\in\br$, we define $T_yf(x):=f(x+y)$.
\begin{lem}\label{comm est}
Let $1\le p\le \infty$. Then, we have 
\begin{equation} \label{comm est ineq1}
\nrm{[P_k, f]g}_{L^{p}} \lesssim 2^{-k}\sup_{y\in\br}\nrm{T_y\left(\px f\right)g}_{L^{p}},
\end{equation}
\begin{equation} \label{comm est ineq2}
\nrm{[P_k, P_{\le k-3}f]g}_{L^{p}} \lesssim 2^{-k}\sup_{y\in\br}\nrm{T_y\left(\px P_{\le k-3} f\right)P_{[k-2, k+2]}g}_{L^{p}}
\end{equation}
for $k\in\bz$
\end{lem}
\begin{proof}
For \eqref{comm est ineq1}, see \cite[Lemma 2]{Tao2001}.
We have \eqref{comm est ineq2} from \eqref{comm est ineq1} and eqalities $P_k(P_{\le k-3}f\cdot g)=P_k(P_{\le k-3}f\cdot P_{[k-2, k+2]}g)$ and $P_k=P_kP_{[k-2, k+2]}$.
\end{proof}
We first consider the estimate for the gauge-transformed initial data.

\begin{lem}\label{est of gauge ini}
Let $s>0$ and $0<\delta<\min(\fr{1}{100}, \fr{s}{2})$.
Then, for any $R>0$, $\phi\in X^s(\br)$ with $\nrm{\phi}_{H^s}\le R$, and $\{Rc_k\}$ which is a $\delta$-frequency envelope of $\phi$ in $H^s(\br)$, we have
\[
\sup_{k\in\bz_{\ge 0}}2^{sk}c_k^{-1}\nrm{P_k\left(e^{-\Lambda(0)}\phi\right)}_{L^2} + \sup_{k\in\bz_{\ge 0}}2^{sk}c_k^{-1}\nrm{P_k\left(e^{-\Lambda(0)}\bar\phi\right)}_{L^2} \lesssim R+R^{2s+2},
\]
where $\Lambda(0):= \lambda\int_\minfty^x\phi(y)\, dy + \mu\int_\minfty^x\bar \phi(y)\, dy$.
\end{lem}

\begin{proof}
First, we consider the case $0< s\le \fr{1}{2}$. Then, we have
\begin{equation}\label{fr decomp}
\nrm{P_k\left(e^{-\Lambda(0)}\phi\right)}_{L^2}
\le \nrm{P_k\left(e^{-\Lambda(0)}P_{\le k-2}\phi\right)}_{L^2} +  \nrm{P_k\left(e^{-\Lambda(0)}P_{\ge k-1}\phi\right)}_{L^2}.
\end{equation}
The second term on the right-hand side of \eqref{fr decomp} is estimated from the properties of the $\delta$-frequency envelope:
\[
\nrm{P_k\left(e^{-\Lambda(0)}P_{\ge k-1}\phi\right)}_{L^2}
\le \sum_{l\ge k-1}\nrm{P_{l}\phi}_{L^2}
\le \sum_{l\ge k-1}2^{-sl +\delta(l-k)}Rc_k
\lesssim  2^{-sk}Rc_k.
\]
Also, Lemma \ref{comm est} and Bernstein's inequality yield
\begin{align*}
\nrm{P_k\left(e^{-\Lambda(0)}P_{\le k-2}\phi\right)}_{L^2}
&=\nrm{[P_k,e^{-\Lambda(0)}]P_{\le k-2}\phi}_{L^2} \\
&\lesssim 2^{-k} \sup_{y\in\br}\nrm{T_y\left(\left(\px \Lambda(0)\right)e^{-\Lambda(0)}\right)P_{\le k-2}\phi}_{L^2} \\
&\lesssim 2^{-k}\nrm{\phi}_{L^2}\sum_{0\le l\le k-2}\nrm{P_{l}\phi}_{L^\infty} \\
&\lesssim 2^{-k}R\sum_{0\le l\le k-2}2^{\fr{l}{2}-sl+\delta(k-l)}Rc_k \\
&\lesssim 2^{-\fr{k}{2}}R^2c_k.
\end{align*}
We obtain the same estimate when we replace $\phi$ with $\bar\phi$. 
Thus, we have the desired estimate. 

Next, we prove the case $s>\fr{1}{2}$. 
We have
\begin{equation}\label{fr decomp2}
\nrm{P_k\left(e^{-\Lambda(0)}\phi\right)}_{L^2}
\le \nrm{[P_k, e^{-\Lambda(0)}]\phi}_{L^2} +  \nrm{e^{-\Lambda(0)}P_k\phi}_{L^2}.
\end{equation}
From $|e^{-\Lambda}|=1$, we have $\nrm{e^{-\Lambda(0)}P_k\phi}_{L^2}\le 2^{-sk}Rc_k$.
Thus, it suffices to consider the estimate for the first term on the right-hand side of \eqref{fr decomp2}. 
From the triangle inequality, we have
\begin{equation}\label{pr est of gauge ini-1}
\nrm{[P_k, e^{-\Lambda(0)}]\phi}_{L^2}
\le \nrm{[P_k, P_{\le k-3}e^{-\Lambda(0)}]\phi}_{L^2} + \nrm{[P_k, \left(1-P_{\le k-3}\right)e^{-\Lambda(0)}]\phi}_{L^2}
=:\I +\II.
\end{equation}
It follows from Lemma \ref{comm est} that
\begin{equation}\label{pr est of gauge ini0}
\begin{aligned}
\I
&\lesssim 2^{-k}\sup_{y\in\br}\nrm{T_y\left(P_{\le k-3}\left(\px \Lambda(0)\cdot e^{-\Lambda(0)}\right)\right)P_{[k-2, k+2]}\phi}_{L^2} \\
&\lesssim 2^{-k}\nrm{\phi}_{L^\infty}\nrm{P_{[k-2, k+2]}\phi}_{L^2} \\
&\lesssim 2^{-(s+1)k}R^2c_k.
\end{aligned}
\end{equation}
Next, we consider the estimate for $\II$.
We note that the same argument as the case $0<s\le \frac{1}{2}$ yields
\begin{equation}\label{pr est of gauge ini0.5}
\nrm{P_k\left(e^{-\Lambda(0)}\phi\right)}_{L^2} 
+ \nrm{P_k\left(e^{-\Lambda(0)}\bar\phi\right)}_{L^2} 
\lesssim 2^{-\frac{1}{2}k}\left(R+R^2\right)c_k.
\end{equation}
Thus, Lemma \ref{comm est}  yields
\begin{equation}\label{pr est of gauge ini1}
\begin{aligned}
\II
&\lesssim 2^{-k}\sup_{y\in\br}\nrm{T_y\left(P_{\ge k-2}\left(\px \Lambda(0)\cdot e^{-\Lambda(0)}\right)\right)\phi}_{L^2} \\
&\lesssim 2^{-k}\sum_{l\ge k-2} \left(\nrm{P_l\left(e^{-\Lambda(0)}\phi\right)}_{L^2} + \nrm{P_l\left(e^{-\Lambda(0)}\bar\phi\right)}_{L^2}\right)\nrm{\phi}_{L^\infty}\\
&\lesssim 2^{-k} \sum_{l\ge k-2} 2^{-\fr{1}{2}l}\left(R + R^2\right)c_l \cdot R \\
&\lesssim 2^{-\fr{1+1}{2}k}\left(R + R^{1+2}\right)c_k.
\end{aligned}
\end{equation}
We also obtain a similar estimate when we replace $\phi$ with $\bar \phi$. 
Therefore, \eqref{fr decomp2}--\eqref{pr est of gauge ini0}, and \eqref{pr est of gauge ini1} yield 
\begin{equation}\label{pr est of gauge ini2}
\nrm{P_k\left(e^{-\Lambda(0)}\phi\right)}_{L^2} 
+ \nrm{P_k\left(e^{-\Lambda(0)}\bar\phi\right)}_{L^2} 
\lesssim 2^{-\frac{1+1}{2}k}\left(R+R^{1+2}\right)c_k.
\end{equation}
Hence, \eqref{pr est of gauge ini2} and the same argument as \eqref{pr est of gauge ini1} yield
\[
\II
\lesssim 2^{-\fr{2+1}{2}k}\left(R + R^{2+2}\right)c_k.
\]
Repeating this argument $[2s]$ times, we have
\begin{equation}\label{pr est of gauge ini3}
\II
\lesssim 2^{-\frac{[2s]+1}{2}k}\left(R+R^{2[s]+2}\right)c_k.
\end{equation}
Since $2s\le [2s] +1$, we obtain from \eqref{fr decomp2}--\eqref{pr est of gauge ini0}, and \eqref{pr est of gauge ini3} that
\begin{equation*}
\nrm{P_k\left(e^{-\Lambda(0)}\phi\right)}_{L^2} 
+ \nrm{P_k\left(e^{-\Lambda(0)}\bar\phi\right)}_{L^2} 
\lesssim 2^{-sk}\left(R+R^{2s+2}\right)c_k.
\end{equation*}
This completes the proof.
\end{proof}

To prove Proposition \ref{apriori gauge fr}, we consider an a priori estimate for $e^{-\Lambda}u$. 
In the following, we define 
\[
S_T:=L_T^\infty L_x^2\cap L_T^4 L_x^\infty
\]
for  $T>0$.
\begin{prop}\label{apriori gauge fr2}
Let $s>0$, $0<\delta<\min(\fr{1}{100}, \fr{s}{2})$, $R>0$, $\phi\in X^{s+1}(\br)$, and $\nrm{\phi}_{H^s}\le R$. 
Also, we assume that $\{Rc_k\}$ is a $\delta$-frequency envelope of $\phi$ in $H^s(\br)$ and $u\in C([0, \infty); X^{s+1}(\br))$ is a solution to \eqref{nls}.  
Then, there exist $T^*=T^*(s, R)>0$ and $C=C(s)>0$ such that
\begin{equation*}
\sup_{k\in\bz_{\ge 0}}2^{sk}c_k^{-1}\nrm{P_k \left(e^{-\Lambda}u\right)}_{S_{T^*}}
\le C\left(R+R^{2s+2}\right).
\end{equation*}
\end{prop}
In the following, we denote $v:=e^{-\Lambda}u$. 
We define
\begin{equation*}%\label{def of M}
M(T)=\sup_{k\in\bz_{\ge 0}}2^{sk}c_k^{-1}\nrm{P_kv}_{S_{T}}
\end{equation*}
for $T>0$.
We note that $M$ is well-defined since Bernstein's inequality yield
\[
2^{sk}c_k^{-1}\nrm{P_kv}_{S_{T}}
\lesssim 2^{(s+\fr{1}{2} +\delta)k}c_0^{-1}(1+{T}^\fr{1}{4})\nrm{P_k v}_{L_{T}^\infty L_x^2} 
\lesssim \nrm{v}\thx{\infty}{s+1}.
\]
Also, we obtain that $M$ is continuous. 
\begin{lem}\label{M}
Under the assumption of Proposition \ref{apriori gauge fr2}, $M$ is continuous
on $[0, T]$ for all $T>0$. 
\end{lem}
\begin{proof}
(\textit{Step 1})
We prove that
$F_k(t):= 2^{sk}c_k^{-1}\nrm{P_kv(t)}_{L^2_x}$ is equicontinuous on $[0, T]$.
This follows from
\begin{align*}
&\sup_k2^{sk}c_k^{-1}\abs{\nrm{P_kv(t)}_{L^2} - \nrm{P_kv(t_0)}_{L^2}} \\
&\le \sup_k2^{sk}c_k^{-1}\nrm{P_k\left(v(t)-v(t_0)\right)}_{L^2} \\
&\lesssim \nrm{v(t) - v(t_0)}_{H^{s+1}}.
\end{align*}

(\textit{Step 2})
We show that $\tilde F_k(t):= \sup_{\tilde t \in [0, t]}F_{k}(\tilde t)$ is also equicontinuous.
We prove it by a contradiction.
We assume that
there exist $t_0\in [0, T]$ and $\vep>0$ such that for any $b'>0$, there exist $k'\in\bn_0$ and $t'\in [0, T]$ such that $\abs{t' - t_0}<b'$ and $|\tilde F_{k'}(t') - \tilde F_{k'}(t_0)|\ge \vep$.
From Step 1, we can take $b>0$ such that
\begin{equation}\label{pr step21}
\sup_{k}\abs{F_{k}(t_1)- F_{k}(t_0)}\leq \frac{\vep}{2}
\end{equation}
for all $\abs{t_1 - t_0} < b$.
By the assumption there exist $k=k(b)\in \bn_0$ and $t=t(b)\in [0, T]$ which satisfy
$\abs{t-t_0}<b$ and $|\tilde F_{k}(t) - \tilde F_{k}(t_0)|\ge \vep$.

In the following, we only consider the case $t>t_0$, however the case $t<t_0$ is similarly handled.
From the definition of $\tilde F_{k}$, there exists $t_2\in [0, t]$ such that
$F_{k}(t_2) >\tilde F _{k}(t) - \frac{\vep}{2}$.
Thus, we have
$F_{k}(t_2)
>\tilde F _{k}(t_0) + \frac{\vep}{2} 
>\tilde F _{k}(t_0)$, which yields
$t_2\in (t_0, t]\subset (t_0, t_0+b)$.
However, we also have
$
F_{k}(t_2) 
>\tilde F_{k}(t_0)+\frac{\vep}{2} 
\ge F_{k}(t_0)+\frac{\vep}{2}
$.
This contradicts \eqref{pr step21}.

(\textit{Step 3})
We prove that
$\lim_{t\to t_0}\sup_{k} |2^{sk}c_{k}^{-1}(\|P_kv\|_{L_t^4([0, t]L_x^\infty)} - \|P_kv\|_{L_t^4([0, t_0]L_x^\infty)})|= 0$ for all $t_0\in[0, T]$. 
Here, we consider $\|P_kv\|_{L_t^4([0, t_0]L_x^\infty)}=0$ when $t_0=0$.
We may assume $t\ge t_0$.
Then, the triangle inequality yields
\begin{align*}
\sup_{k}2^{sk}c_{k}^{-1}\abs{\nrm{P_kv}_{L_t^4([0, t]L_x^\infty)} - \nrm{P_kv}_{L_t^4([0, t_0]L_x^\infty)}} 
&\le \sup_{k}2^{sk}c_{k}^{-1}\nrm{P_kv}_{L^4([t_0, t]; L_x^\infty)} \\ 
&\lesssim (t - t_0)^{\frac{1}{4}}\nrm{v}_{L_T^\infty H_x^{s+1}}.
\end{align*}
Hence, we have the desired equality.

(\textit{Step 4})
Step 2 and Step 3 mean the continuity of the $L_T^\infty L_x^2$ and $L_T^4 L_x^\infty$ parts of $M$, respectively.
This completes the proof.
\end{proof}
\begin{proof}[Proof of Proposition \ref{apriori gauge fr2}]
For $T, {\tilde C}>0$, We define a subset of $(0, T]$ as follows:
\[
N_{\tilde C}(T):=\{T'\in(0, T]\mid M(T')\le {\tilde C}\left(R+R^{2s+2}\right)\}.
\]
We note that $N_{\tilde C}(T)$ is not empty if ${\tilde C}$ is sufficiently large, since Lemmas \ref{est of gauge ini} and \ref{M} yield $\limsup_{T\searrow 0}M(T)\lesssim R + R^{2s+2}$.
Thus, it suffices to show that there exists $C = C(s)>0$ such that if $T'\in N_{\tilde C}(T)$, then we have
\begin{equation}\label{desired est of fr}
M(T')\le  C\left(R + R^{2s+2}\right) +  C{T'}^{\fr{1}{2}}\left({\tilde C}\left(R + R^{2s+2}\right)\right)^3.
\end{equation}
Indeed, taking ${\tilde C}=3 C$ and ${T}=\tilde C^{-6}\left(R + R^{2s+2}\right)^{-4}$, we obtain the desired estimate by a continuity argument.

Lemma \ref{strest} and \eqref{gauge special case} yield
\[
\nrm{P_k v}_{S_{T'}}
\le  C \left(\nrm{P_k v(0)}_{L^2} + \nrm{P_k\left(v^2\bar v\right)}_{L_{T'}^1L_x^2}\right).
\]
We may consider the estmate of $\|P_k(v^3)\|_{L_{T'}^1L_x^2}$ instead of $\|P_k(v^2\bar v)\|_{L_{T'}^1L_x^2}$.
It holds that $P_k((P_{\le k-4}v)^3)=0$.
Thus, we obtain the estimate as follows:
\begin{equation}\label{pr apriori gauge fr2}
\begin{aligned}
\nrm{P_k\left(v^3\right)}_{L_{T'}^1L_x^2}
&\le C{T'}^{\fr{1}{2}}\nrm{P_{\ge k-3}v}_{L_{T'}^\infty L_x^2}\nrm{v}_{L_{T'}^4 L_x^\infty}^2 \\
&\le C{T'}^{\fr{1}{2}}\Bigl(\sum_{l\ge k-3}\nrm{P_l v}_{S_{T'}}\Bigr)\Bigl(\sum_{l\ge0}\nrm{P_l v}_{S_{T'}}\Bigr)^2 \\
&\le C{T'}^{\fr{1}{2}}\Bigl(\sum_{l\ge k-3}2^{-sl+\delta\abs{k-l}}{\tilde C}(R+R^{2s+2})c_k\Bigr)\Bigl(\sum_{l\ge 0}2^{-sl}{\tilde C}(R+R^{2s+2})c_l\Bigr)^2 \\
&\le  C 2^{-sk}{T'}^{\fr{1}{2}}\left({\tilde C}\left(R + R^{2s+2}\right)\right)^3c_k.
\end{aligned}
\end{equation}
Therefore, \eqref{pr apriori gauge fr2} and Lemma \ref{est of gauge ini} yield the desired estimate \eqref{desired est of fr}.
\end{proof}

Since we proved Proposition \ref{apriori gauge fr2},
the same argument as this proposition can weaken the condition for $T^*$.

\begin{prop}\label{apriori gauge fr2.5}
Let $s\ge \tilde s> 0$, $0<\delta<\min(\fr{1}{100}, \fr{\tilde s}{2})$, $r, R>0$, $\phi\in X^{s+1}(\br)$, $\nrm{\phi}_{H^{\tilde s}}\le r$, and $\nrm{\phi}_{H^s}\le R$. 
Also, we assume that $\{Rc_k\}$ is a $\delta$-frequency envelope of $\phi$ in $H^s(\br)$ and $u\in C([0, \infty); X^{s+1}(\br))$ is a solution to \eqref{nls}.  
Then, there exist $T^*=T^*(s, r)>0$ and $C=C(s)>0$ such that
\begin{equation*}
\sup_{k\in\bz_{\ge 0}}2^{sk}c_k^{-1}\nrm{P_k \left(e^{-\Lambda}u\right)}_{S_{T^*}}
\le C\left(R+R^{2s+2}\right).
\end{equation*}
\end{prop}

\begin{proof}
It suffices to show that there exist $T^*=T^*(s, \tilde s, r)$, $C=C(s)>0$, and $ C' = C'(s, \tilde s)>0$ such that if $T'\in N_{\tilde C}(T^*)$, then we have
\[
M(T')
\le C(R + R^{2s+2})+C'{T'}^{\frac{1}{2}}(r+r^{2\tilde s+2})^2\tilde C(R+R^{2s+2}).
\]
Indeed, by taking $\tilde C= 3C$ and ${T'}=\min(T^*, C^2{(C'\tilde C)^{-2}(r+r^{2\tilde s+2})^{-4}})$, a continuity argument yields the desired estimate. 
From Proposition \ref{apriori gauge fr2}, there exists $T^*=T^*(\tilde s, r)>0$ such that
\[
\sum_{l\ge0}\nrm{P_l v}_{S_{T'}}
\lesssim r+r^{2\tilde s+2}
\]
for $T'\in (0, T^*]$, where the implicit constant depends only on $\tilde s$.
Thus, a similar argument as \eqref{pr apriori gauge fr2} yields
\[
\nrm{P_k\left(|v|^2v\right)}_{L_{T'}^1L_x^2}
\lesssim 2^{-sk}{T'}^{\fr{1}{2}}(r+r^{2\tilde s+2})^2\tilde C(R + R^{2s+2})c_k.
\]
Hence, we obtain
\[
M({T'})
\le  C(R+R^{2s+2}) + C'{T'}^{\frac{1}{2}}(r+r^{2\tilde s+2})^2\tilde C(R+R^{2s+2}).
\]

Thus, we obtain the desired estimate.
\end{proof}

From Proposition \ref{apriori gauge fr2.5}, we can obtain an a priori estimate for the solution to \eqref{nls} by the similar argument as Lemma \ref{est of gauge ini}. 

\begin{comment}
However, it is necessary to use the estimate for $e^{\Lambda}u$ since there is a term that contains $\overline{e^{\Lambda}u}$ when we consider $\px e^{-\Lambda}$. 
Hence, we use the bound of $e^{\Lambda}u$ in Proposition \ref{apriori gauge fr} .
\end{comment}

\begin{proof}[Proof of Proposition \ref{apriori gauge fr}]
In the following, let $T^*=T^*(s, \tilde s, r)>0$ be such that Proposition \ref{apriori gauge fr2.5} holds. 
We denote $\tilde R :=R+R^{2s+2}$. 

First, we prove
\begin{equation}\label{fr pr0}
\nrm{P_k u}_{S_{T^*}} + \nrm{P_k \left(e^{\Lambda}u\right)}_{S_{T^*}}
\lesssim 2^{-\min(\frac{1}{2}, s)k}(\tilde R+{\tilde R}^2).
\end{equation}
From $u=e^{\Lambda}e^{-\Lambda}u$, we have 
\begin{equation}\label{fr pr 1}
\nrm{P_ku}_{S_{T^*}}
\le \nrm{P_k\left(e^{\Lambda}P_{\le k-3}\left(e^{-\Lambda}u\right)\right)}_{S_{T^*}} +  \nrm{P_k\left(e^{\Lambda}P_{\ge k-2}\left(e^{-\Lambda}u\right)\right)}_{S_{T^*}}.
\end{equation}
The second term on the right-hand side of \eqref{fr pr 1} is estimated from Proposition \ref{apriori gauge fr2}:
\[
\nrm{P_k\left(e^{\Lambda}P_{\ge k-2}\left(e^{-\Lambda}u\right)\right)}_{S_{T^*}}
\lesssim \sum_{l\ge k-2}\nrm{P_{l}\left(e^{-\Lambda}u\right)}_{S_{T^*}}
\lesssim 2^{-sk}\tilde Rc_k.
\]
Also, Lemma \ref{comm est}, Proposition \ref{apriori gauge fr2}, and Bernstein's inequality yield
\begin{align*}
\nrm{P_k\left(e^{\Lambda}P_{\le k-3}\left(e^{-\Lambda}u\right)\right)}_{S_{T^*}}
&=\nrm{[P_k,e^{\Lambda}]P_{\le k-3}\left(e^{-\Lambda}u\right)}_{S_{T^*}} \\
&\lesssim 2^{-k} \sup_{y\in\br}\nrm{T_y\left(\px \Lambda\cdot e^{\Lambda}\right)P_{\le k-3}\left(e^{-\Lambda}u\right)}_{S_{T^*}} \\
&\lesssim 2^{-k}\underbrace{\nrm{u}_{S_{T^*}}}_{=\nrm{e^{-\Lambda}u}_{S_{T^*}}}\sum_{0\le l\le k-3}\nrm{P_l\left(e^{-\Lambda}u\right)}\txx{\infty} \\
&\lesssim 2^{-k}\tilde R\sum_{0\le l\le k-3}2^{\fr{l}{2}-sl+\delta(k-l)}\tilde Rc_k \\
&\lesssim 2^{-\fr{k}{2}}\tilde R^2c_k.
\end{align*}
Thus, we obtain the desired estimate for $\nrm{P_k u}_{S_{T^*}}$. 
From $e^{\Lambda}u=e^{2\Lambda}e^{-\Lambda}u$, a similar argument as $\nrm{P_k u}_{S_{T^*}}$ yields the desired bound.
Thus, we obtain \eqref{fr pr0}.
If $0< s\le \frac{1}{2}$, then this estimate completes the proof.
In the following, we assume $s>\frac{1}{2}$.

We decompose as
\begin{equation}\label{fr decomp3}
\nrm{P_ku}_{S_{T^*}}
\le \nrm{[P_k, e^{\Lambda}]e^{-\Lambda}u}_{S_{T^*}} +  \nrm{e^{\Lambda}P_k\left(e^{-\Lambda}u\right)}_{S_{T^*}}
\end{equation}
The second term on the right-hand side of \eqref{fr decomp3} is bounded by $2^{-sk}C\tilde Rc_k$. 
Thus, we need to estimate the first term on the right-hand side of \eqref{fr decomp3}. 
We decompose this term as follows:
\begin{align*}
\nrm{[P_k, e^{\Lambda}]\left(e^{-\Lambda}u\right)}_{S_{T^*}}
&\le \nrm{[P_k, P_{\le k-3}e^{\Lambda}]\left(e^{-\Lambda}u\right)}_{S_{T^*}} + \nrm{[P_k, \left(1-P_{\le k-3}\right)e^{\Lambda}]\left(e^{-\Lambda}u\right)}_{S_{T^*}} \\
&=:\I +\II.
\end{align*}
From Lemma \ref{comm est} and Proposition \ref{apriori gauge fr2}, it follows that
\begin{align*}
\I
&\lesssim 2^{-k}\sup_{y\in\br}\nrm{T_y\left(P_{\le k-3}\left(\px \Lambda\cdot e^{\Lambda}\right)\right)P_{[k-2, k+2]}\left(e^{-\Lambda}u\right)}_{S_{T^*}} \\
&\lesssim 2^{-k}\underbrace{\nrm{u}\txx{\infty}}_{=\nrm{e^{-\Lambda}u}\txx{\infty}\lesssim \tilde R}\nrm{P_{[k-2, k+2]}\left(e^{-\Lambda}u\right)}_{S_{T^*}} \\
&\lesssim 2^{-(s+1)k}\tilde R^2c_k.
\end{align*}
Also, Lemma \ref{comm est} and \eqref{fr pr0} yield
\begin{equation}\label{apriori fr step1}
\begin{aligned}
\II
&\lesssim 2^{-k}\sup_{y\in\br}\nrm{T_y\left(P_{\ge k-2}\left(\px \Lambda\cdot e^{\Lambda}\right)\right)\left(e^{-\Lambda}u\right)}_{S_{T^*}} \\
&\lesssim 2^{-k}\sum_{l\ge k-2} \Bigl(\nrm{P_l\left(e^{\Lambda}u\right)}_{S_{T^*}}
+\underbrace{\nrm{P_l\left(e^{\Lambda}\bar u\right)}_{S_{T^*}}}_{=\nrm{P_l\left(e^{-\Lambda}u\right)}_{S_{T^*}}}\Bigr)\nrm{e^{-\Lambda}u}\txx{\infty}\\
&\lesssim 2^{-k} \sum_{l\ge k-2} 2^{-\fr{1}{2}l}\left(\tilde R + \tilde R^2\right)c_l \cdot \tilde R \\
&\le 2^{-(1+\fr{1}{2})k}\left(\tilde R + \tilde R^3\right)c_k.
\end{aligned}
\end{equation}
Thus, we obtain 
\begin{equation}\label{apriori fr step2}
\nrm{P_ku}_{S_{T^*}}
\lesssim 2^{-\min\{(1+\fr{1}{2}), s\}k}\left(\tilde R + \tilde R^3\right)c_k.
\end{equation}

Next, we consider $\nrm{P_{k}\left(e^{\Lambda}u\right)}_{S_{T^*}}$.
We obtain
\begin{equation}\label{fr decomp4}
\nrm{P_k\left(e^{\Lambda}u\right)}_{S_{T^*}}
\le \nrm{[P_k, e^{\Lambda}]u}_{S_{T^*}} +  \nrm{e^{\Lambda}P_ku}_{S_{T^*}}.
\end{equation}
The second term on the right-hand side of \eqref{fr decomp4} is bounded by using \eqref{apriori fr step2}.
We decompose
\begin{align*}
\nrm{[P_k, e^{\Lambda}]u}_{S_{T^*}}
&\le \nrm{[P_k, P_{\le k-3}e^{\Lambda}]u}_{S_{T^*}} + \nrm{[P_k, \left(1-P_{\le k-3}\right)e^{\Lambda}]u}_{S_{T^*}} \\
&=:\III +\IV.
\end{align*}
From Lemma \ref{comm est} and \eqref{fr pr0}, we have
\begin{align*}
\III
&\lesssim 2^{-k}\sup_{y\in\br}\nrm{T_y\left(P_{\le k-3}\left(\px \Lambda\cdot e^{\Lambda}\right)\right)P_{[k-2, k+2]}u}_{S_{T^*}} \\
&\lesssim 2^{-k}\underbrace{\nrm{u}\txx{\infty}}_{=\nrm{e^{-\Lambda}u}\txx{\infty}\lesssim  \tilde R}\nrm{P_{[k-2, k+2]}u}_{S_{T^*}} \\
&\lesssim 2^{-\min\{(2+\fr{1}{2}), s+1\}k}\left(\tilde R + \tilde R^{2+2}\right)c_k.
\end{align*}
Similarly, Lemma \ref{comm est}, Proposition \ref{apriori gauge fr2}, and \eqref{fr pr0} yield
\begin{align*}
\IV
&\lesssim 2^{-k}\sup_{y\in\br}\nrm{T_y\left(P_{\ge k-2}\left(\px \Lambda\cdot e^{\Lambda}\right)\right)u}_{S_{T^*}} \\
&\lesssim 2^{-k}\sum_{l\ge k-2} \left(\nrm{P_l\left(e^{\Lambda}u\right)}_{S_{T^*}}
+\nrm{P_l\left(e^{\Lambda}\bar u\right)}_{S_{T^*}}\right)\nrm{u}\txx{\infty}\\
&\lesssim 2^{-k} \sum_{l\ge k-2} 2^{-\fr{1}{2}l}\left(\tilde R + \tilde R^{2}\right)c_l \cdot \tilde R \\
&\le 2^{-(1+\fr{1}{2})k}\left(\tilde R + \tilde R^{3}\right)c_k.
\end{align*}
Therefore, we obtain
\begin{equation}\label{apriori fr step3}
\nrm{P_k(e^{\Lambda}u)}_{S_{T^*}}
\lesssim 2^{-\min\{(1+\fr{1}{2}), s\}k}\left(\tilde R + \tilde R^{2+2}\right)c_k.
\end{equation}
Thus, \eqref{apriori fr step2} and \eqref{apriori fr step3} yield
\begin{equation}\label{apriori fr step4}
\nrm{P_ku}_{S_{T^*}} + \nrm{P_k(e^{\Lambda}u)}_{S_{T^*}}
\lesssim 2^{-\min\{(1+\fr{1}{2}), s\}k}\left(\tilde R + \tilde R^{2+2}\right)c_k.
\end{equation}
The inequality \eqref{apriori fr step4} and a similar argument as \eqref{apriori fr step1} yields
\begin{equation*}
\nrm{P_ku}_{S_{T^*}}
\lesssim 2^{-\min\{(2+\fr{1}{2}), s\}k}\left(\tilde R + \tilde R^5\right)c_k.
\end{equation*}
Hence, we obtain from similar arguments as the estimate for $\III$ and $\IV$ that
\[
\nrm{P_k(e^{\Lambda}u)}_{S_{T^*}}
\lesssim 2^{-\min\{(2+\fr{1}{2}), s\}k}\left(\tilde R + \tilde R^{4+2}\right)c_k.
\]
In particular, we have
\[
\nrm{P_ku}_{S_{T^*}}  +\nrm{P_k(e^{\Lambda}u)}_{S_{T^*}}
\lesssim 2^{-\min\{(2+\fr{1}{2}), s\}k}\left(\tilde R + \tilde R^{4+2}\right)c_k.
\]
From $[s]+1 +\frac{1}{2}\ge s$, we obtain by repeating this argument $[s]+1$ times that
\[
\nrm{P_ku}_{S_{T^*}}  +\nrm{P_k(e^{\Lambda}u)}_{S_{T^*}}
\lesssim 2^{-sk}\left(\tilde R + \tilde R^{2s+4}\right)c_k.
\]
Thus, we have the desired estimate.
\end{proof}
When $s>1$, we can control the $L_T^4\dot W_x^{1, \infty}$-norm of the solution by the triangle inequality.
However, the case $s=1$ needs a small modification to control it.
\begin{cor}\label{apriori gauge fr3}
Let $s\ge \tilde s> 0$, $0<\delta<\min(\fr{1}{100}, \fr{\tilde s}{2})$, $r, R>0$, $\phi\in X^{s+1}(\br)$, $\nrm{\phi}_{H^{\tilde s}}\le r$, and $\nrm{\phi}_{H^s}\le R$. 
Also, we assume that $\{Rc_k\}$ is a $\delta$-frequency envelope of $\phi$ in $H^s(\br)$ and $u\in C([0, \infty); X^{s+1}(\br))$ is a solution to \eqref{nls}.  
Then, there exist $T^*=T^*(s, r)>0$ and $C=C(s)>0$ such that
\[
\nrm{P_{\ge k}\left(e^{-\Lambda}u\right)}_{L_{T^*}^\infty H_x^s \cap L_{T^*}^4 \dot W_x^{1, \infty}}
\le C\left(R+R^{6(s+1)}\right)\left(\sum_{l\ge k}c_l^2\right)^{\fr{1}{2}}
\]
for all $k\in\bn_0$.
In particular, we have
\[
\nrm{\px u}_{L_{T^*}^4 L_x^\infty}
\le C(s, R).
\]
\end{cor}
\begin{proof}
We take $T^*=T^*(s, r)$ sufficiently small such that Proposition \ref{apriori gauge fr2.5} holds.
As in the proof of Proposition \ref{apriori gauge fr2.5}, we obtain
\begin{align*}
&\nrm{P_{\ge k}\left(e^{-\Lambda}u\right)}_{L_{T^*}^\infty H_x^s \cap L_{T^*}^4 \dot W_x^{1, \infty}}\\
&\lesssim \nrm{P_{\ge k}\left(e^{-\Lambda(0)}\phi\right)}_{H^s} + T^*\nrm{e^{-\Lambda}u}_{L_{T^*}^\infty H_x^s}^2\nrm{P_{\ge k-4}\left(e^{-\Lambda}u\right)}_{L_{T^*}^\infty H_x^s}.
\end{align*}
Thus, Lemma \ref{est of gauge ini}, Proposition \ref{apriori gauge fr2.5}, and the property of the frequency envelope yield
\begin{align*}
&\nrm{P_{\ge k}\left(e^{-\Lambda}u\right)}_{L_{T^*}^\infty H_x^s \cap L_{T^*}^4 \dot W_x^{1, \infty}} \\
&\lesssim \left(R+R^{2s+2}\right)\left(\sum_{l\ge k}c_l^2\right)^{\fr{1}{2}} + \left(R+R^{2s+2}\right)^3\left(\sum_{l\ge k}c_l^2\right)^{\fr{1}{2}}.
\end{align*}
Therefore, we obtain the desired estimate.
\end{proof}

\subsection{Proof of well-posedness in $H^{s}(\br)$}

In this subsection, we prove Theorem \ref{special case $H^s$}. 
From \eqref{aprori $H^1$}, it suffices to show the local well-posedness and that the existence time of solution depends only on $s$ and $H^1$-norm of the initial data.
We note that the uniqueness of the solution follows from Proposition \ref{diff energy}. 
Thus, we only prove the existence of the solution and the continuity of the flow map. 

\begin{proof}[Proof of Theorem \ref{special case $H^s$}]
(\textit{Existence of a solution})
In this proof, we also show that the existence time depends only on $s$ and $H^1$-norm of the initial data.
Let $\phi\in H^s(\br)$, $\nrm{\phi}_{H^1}\le r$, $\nrm{\phi}_{H^s}\le R$, and $\hat \phi_n := 1_{\fr{1}{n}\le \abs{\xi}\le n}\hat \phi$. 
Also, let $u_n\in C([0, \infty); X^{s+1}(\br))$ be a solution to \eqref{nls} with initial data $\phi_n \in X^{s+1}(\br)$. 
Moreover, we define $\Lambda_n$ in the same way as \eqref{def of gauge}.
It suffices to prove that $\{u_n\}$ is a Cauchy sequence in $C([0, T^*]; H^s(\br))\cap  L_{T^*}^1 \dot W_x^{1, \infty}$ for some small $T^*=T^*(s, r)$, since an approximation argument yields that its limit satisfies \eqref{sol}.

The triangle inequality and Bernstein's inequality yield
\begin{align*}
\nrm{u_n-u_m}_{L_{T^*}^\infty H_x^s\cap L_{T^*}^1 \dot W_x^{1, \infty}}
&\lesssim 2^{(s+\fr{1}{2})k}\nrm{u_n-u_m}_{L_{T^*}^\infty L_x^2} \\
&\quad + \nrm{P_{\ge k+1}u_n}_{L_{T^*}^\infty H_x^s\cap L_{T^*}^1 \dot W_x^{1, \infty}} + \nrm{P_{\ge k+1}u_m}_{L_{T^*}^\infty H_x^s\cap L_{T^*}^1 \dot W_x^{1, \infty}}
\end{align*}
for all $k, m, n\in\bn_0$.
From Proposition \ref{diff energy}, we obtain
$\nrm{u_n-u_m}_{L_{T^*}^\infty L_x^2}\to 0$ as $n, m\to\infty$.
Thus, we have
\begin{align*}
&\limsup_{n, m\to\infty}\nrm{u_n-u_m}_{L_{T^*}^\infty H_x^s\cap L_{T^*}^1 \dot W_x^{1, \infty}} \\
&\lesssim \limsup_{n\to\infty}\nrm{P_{\ge k+1}u_n}_{L_{T^*}^\infty H_x^s\cap L_{T^*}^1 \dot W_x^{1, \infty}} + \limsup_{m\to\infty}\nrm{P_{\ge k+1}u_m}_{L_{T^*}^\infty H_x^s\cap L_{T^*}^1 \dot W_x^{1, \infty}}
\end{align*}
for all $k\in\bn_0$.
We consider the estimate for $\limsup_{n\to\infty}\nrm{P_{\ge k+1}u_n}_{L_{T^*}^\infty H_x^s\cap L_{T^*}^1 \dot W_x^{1, \infty}}$.
First, we obtain from Proposition \ref{apriori gauge fr} that
\[
\nrm{P_{\ge k+1}u_n}_{L_{T^*}^\infty H_x^s} 
\lesssim \left(R+R^{4(s+1)(s+2)}\right)\Bigl(\sum_{l\ge k+1}c_l^2\Bigr)^{\fr{1}{2}}
\]
for some $T^*=T^*(s, r)>0$.
Here, $\{Rc_k\}$ is the $\delta$-freqency envelope of $\phi$ in $H^s(\br)$.
Next, from $u_n=e^{\Lambda_n}e^{-\Lambda_n}u_n$, we have
\[
\nrm{\px P_{\ge k+1}u_n}_{L_{T^*}^4 L_x^{\infty}}
\lesssim  \nrm{P_{\ge k+1}\left(u_n\px\Lambda_n\right)}_{L_{T^*}^4 L_x^{\infty}} 
+ \nrm{P_{\ge k+1}\left(e^{\Lambda_n}\px \left(e^{-\Lambda_n}u_n\right)\right)}_{L_{T^*}^4 L_x^{\infty}}.
\]
Proposition \ref{apriori gauge fr} and the Sobolev embedding yield
\[
\nrm{P_{\ge k+1}\left(u_n\px\Lambda_n\right)}_{L_{T^*}^4 L_x^{\infty}} 
\lesssim \nrm{u_n}_{L_{T^*, x}^\infty}\nrm{P_{\ge k-2}u_n}_{L_{T^*}^4 L_x^{\infty}}
\lesssim \left(R+R^{4(s+1)(s+2)}\right)^2\Bigl(\sum_{l\ge k+1}c_l^2\Bigr)^{\fr{1}{2}}.
\]
We consider the estimate for $\nrm{P_{\ge k+1}\left(e^{\Lambda_n}\px \left(e^{-\Lambda}u_n\right)\right)}_{L_{T^*}^4 L_x^{\infty}}$.
We have
\begin{align*}
&\nrm{P_{\ge k+1}\left(e^{\Lambda_n}\px \left(e^{-\Lambda_n}u_n\right)\right)}_{L_{T^*}^4 L_x^{\infty}} \\
&\le \nrm{[P_{\ge k+1}, e^{\Lambda_n}]\px \left(e^{-\Lambda_n}u_n\right)}_{L_{T^*}^4 L_x^{\infty}} + \nrm{\px P_{\ge k+1}\left(e^{-\Lambda_n}u_n\right)}_{L_{T^*}^4 L_x^{\infty}},
\end{align*}
and the second term on the right-hand side is estimated  by Corollary \ref{apriori gauge fr3}:
\[
\nrm{\px P_{\ge k+1}\left(e^{-\Lambda_n}u_n\right)}_{L_{T^*}^4 L_x^{\infty}}
\lesssim \left(R+R^{6(s+1)}\right)\Bigl(\sum_{l\ge k+1}c_l^2\Bigr)^{\fr{1}{2}}.
\]
Also, Proposition \ref{apriori gauge fr}, Lemma \ref{comm est}, and Corollary \ref{apriori gauge fr3} yield
\begin{align*}
\nrm{[P_{\ge k+1}, e^{\Lambda_n}]\px \left(e^{-\Lambda}u_n\right)}_{L_{T^*}^4 L_x^{\infty}}
&\lesssim \sum_{l\ge k+1}2^{-l}\sup_{y\in\br}\nrm{\px\Lambda_n\cdot e^{\Lambda_n}T_y\left(\px \left(e^{-\Lambda}u_n\right)\right)}_{L_{T^*}^4 L_x^{\infty}} \\
&\le\sum_{l\ge k+1}2^{-l}\sup_{y\in\br}\nrm{u_n}_{L_{T^*, x}^{\infty}}\nrm{\px \left(e^{-\Lambda}u_n\right)}_{L_{T^*}^4 L_x^{\infty}} \\
&\lesssim 2^{-k}\left(R+R^{4(s+1)(s+2)}\right)\left(R+R^{6(s+1)}\right).
\end{align*}
Hence, we obtain 
\begin{align*}
&\limsup_{n, m\to\infty}\nrm{u_n-u_m}_{L_{T^*}^\infty H_x^s\cap L_{T^*}^1 \dot W_x^{1, \infty}} \\
&\lesssim \left(R+R^{4(s+1)(s+2)}\right)\left(\sum_{l\ge k+1}c_l^2\right)^{\fr{1}{2}} + 2^{-k}\left(R+R^{4(s+1)(s+2)}\right)\left(R+R^{6(s+1)}\right).
\end{align*}
Taking $k\to\infty$, we obtain that $\{u_n\}$ is a Cauchy sequence in $C([0, T^*]; H^s(\br))\cap  L_{T^*}^1 \dot W_x^{1, \infty}$.
We note that the equality $P_{\ge k+1}u = (u-u_n) - P_{\le k}(u-u_n) + P_{\ge k+1}u_n$ 
yields
\begin{equation}\label{est of u}
\begin{aligned}
&\nrm{P_{\ge k+1}u}_{L_{T^*}^\infty H_x^s\cap L_{T^*}^1 \dot W_x^{1, \infty}} \\
&\lesssim \left(R+R^{4(s+1)(s+2)}\right)\left(\sum_{l\ge k+1}c_l^2\right)^{\fr{1}{2}} + 2^{-k}\left(R+R^{4(s+1)(s+2)}\right)\left(R+R^{6(s+1)}\right).
\end{aligned}
\end{equation}

(\textit{Continuity of the flow map})
Let $\phi, \phi_n\in H^{s}(\br)$, $\nrm{\phi}_{H^s}\le R$, and $\phi_n\to \phi$ in $H^s(\br)$. 
Also, let $\vep\in(0, R)$ and we take $N\in\bn_0$ such that $\nrm{\phi-\phi_n}_{H^s}\le \vep$ for all $n\ge N$.
Thus, taking $T^*=T^*(s, R)$ smaller if necessary, we may assume $u, u_n\in C\left([0, T^*]; H^s(\br)\right)\cap L_{T^*}^1 \dot W_x^{1, \infty}$ is a solution to \eqref{nls} with initial data $\phi, \phi_n$, respectively. 
Let $\{\vep d_k^{(n)}\}$ be a frequency envelope of $\phi-\phi_n$ in $H^s(\br)$.
Then, $R(c_k + \fr{\vep}{R}d_k^{(n)})$ is a $\delta$-frequency envelope of $\phi_n$.
Therefore, Proposition \ref{diff energy} and \eqref{est of u} yield
\begin{align*}
&\limsup_{n\to\infty}\nrm{u-u_n}_{L_{T^*}^\infty H_x^s\cap L_{T^*}^1 \dot W_x^{1, \infty}} \\
&\le \limsup_{n\to\infty}\nrm{P_{\le k}(u-u_n)}_{L_{T^*}^\infty H_x^s\cap L_{T^*}^1 \dot W_x^{1, \infty}} \\
&\quad + \nrm{P_{\ge k+1}u}_{L_{T^*}^\infty H_x^s\cap L_{T^*}^1 \dot W_x^{1, \infty}} 
+ \limsup_{n\to\infty}\nrm{P_{\ge k+1}u_n}_{L_{T^*}^\infty H_x^s\cap L_{T^*}^1 \dot W_x^{1, \infty}} \\
&\lesssim \left(R+{R}^{4(s+1)(s+2)}\right)\left(\sum_{l\ge k+1}c_l^2\right)^{\fr{1}{2}} + \vep\left(R+{R}^{4(s+1)(s+2)}\right) \\
&\quad + 2^{-k}\left(R+R^{4(s+1)(s+2)}\right)\left(R+{R}^{6(s+1)}\right)
\end{align*}
for all $k\in\bn_0$.
Thus, by taking $k\to\infty$, we obtain
\[
\limsup_{n\to\infty}\nrm{u-u_n}_{L_{T^*}^\infty H_x^s\cap L_{T^*}^1 \dot W_x^{1, \infty}}
\lesssim \vep\left(R+{R}^{(2s+3)^2}\right).
\]
Since $\vep\in(0, R)$ is taken arbitrarily, $u_n$ converges to $u$ in $C\left([0, {T^*}]; H^s(\br)\right)\cap L_{T^*}^1 \dot W_x^{1, \infty}$.
\end{proof}
\begin{appendix}
\section{On the limit of regularized solutions}\label{appendixA}
As we mentioned in Remark \ref{remark}, the limit of the regularized solution exists uniquely without depending on the way of the approximation.
\begin{prop}
Under the same condition as in the proof of Proposition \ref{existence of a smooth solution}, we have $\lim_{(\vep, \eta)\to (0, 0)}u_{\vep, \eta}=u$ in $C([0, \tilde T^{**}]; X^s(\br))$, where $ u:=\lim_{\eta\searrow 0}u_{\eta^3, \eta}$.
\end{prop}
\begin{proof}
In the following, by taking $\tilde T^{**}=\tilde T^{**}(s, \nrm{\phi}_{X^s})$ sufficiently small if necessary, we may assume Propositions \ref{smooth apriori}, \ref{smooth apriori2}, \ref{smooth est of diff}, and \ref{apriori} hold. 
We have 
\[
\nrm{u_{\vep, \eta}-u}_{L_{\tilde T^{**}}^{\infty}X^s}
\le \nrm{u_{\vep, \eta}-u_{\vep, 0}}_{L_{\tilde T^{**}}^{\infty}X^s} + \nrm{u_{\vep, 0}-u}_{L_{\tilde T^{**}}^{\infty}X^s}
\]
From Proposition \ref{smooth est of diff} and Lemma \ref{BS} (iii), it suffices to show that
\[
\lim_{\vep\to +0}\nrm{u_{\vep, 0}-u}_{L_{\tilde T^{**}}^{\infty}X^s} =0.
\]
In the following, we note that $u$ satisfies \eqref{mild-sol} in $H^{s-1}(\br)$

First, we consider the estimate for the primitive.
We have
\begin{align*}
&\nrm{\int_\minfty^x\left(u_{\vep, 0}(t, y) -u(t, y)\right)\,dy}_{L_{\tilde T^{**}, x}^{\infty}} \\
&\le \nrm{\int_\minfty^x\left(U_{\vep}(t) -U(t)\right)\phi(y)\,dy}_{L_{\tilde T^{**}, x}^{\infty}}\\
&\quad + \nrm{\int_0^tU_\vep(t-t')\left(\lambda \left(u_{\vep, 0}^2-u^2\right) + \mu\left(\abs{ u_{\vep, 0}}^2-\abs{u}^2\right)\right)(t')\,dt'}_{L_{\tilde T^{**}, x}^{\infty}} \\
&\quad + \nrm{\int_0^t\left(U_\vep(t-t')-U(t-t')\right)\left(\lambda u^2 +\mu\abs{u}^2\right)(t')\,dt'}_{L_{\tilde T^{**}, x}^{\infty}} \\
&=: \I_1 + \I_2 + \I_3.
\end{align*}
For $\I_1$, we can apply \eqref{ineq for appendix} when considering $\vep_2=0$:
\[
\I_1
\lesssim \vep^{\frac{1}{2}} {\tilde T^{**\frac{1}{2}}}\nrm{\phi}_{H^1}.
\]
It follows from a similar argument as Lemma \ref{est of Duhamel} that
\[
\I_2 \lesssim {\tilde T^{**\frac{1}{2}}}\left(\nrm{u_{\vep, 0}}_{L_{\tilde T^{**}}^{\infty}L_x^2} + \nrm{u}_{L_{\tilde T^{**}}^{\infty}L_x^2}\right)\nrm{u_{\vep, 0}-u}_{L_{\tilde T^{**}}^{\infty}L_x^2}.
\]
As in the proof of Lemma \ref{smooth est of diff step1}, we have
\[
\I_3 
\lesssim {\tilde T^{**\frac{9}{8}}}\vep^{\fr{7}{8}}\nrm{u}_{L_{\tilde T^{**}}^{\infty}H_x^1}^2.
\]

Next, we consider the estimate for the $L_{\tilde T^{**}}^\infty H_x^s$-norm.
Lemma \ref{seki}, \eqref{est of exp f2}, and \eqref{pr diff2 ineq} yield
\begin{align*}
&\nrm{u_{\vep, 0} - u}_{L_{\tilde T^{**}}^{\infty}H_x^s} \\
&\le \nrm{u_{\vep, 0}(1-e^{-\Lambda_{\vep, 0}+\Lambda})}_{L_{\tilde T^{**}}^{\infty}H_x^s}
+ \nrm{e^{\Lambda}\left(e^{-\Lambda_{\vep, 0}}u_{\vep, 0} - e^{-\Lambda}u\right)}_{L_{\tilde T^{**}}^{\infty}H_x^s} \\
&\le C(s, \|u\|_{L_{\tilde T^{**}}^{\infty}H_x^{s}}, \|u_{\vep, 0}\|_{L_{\tilde T^{**}}^{\infty}H_x^{s}})\left(\nrm{u_{\vep, 0} -u}_{L_{\tilde T^{**}}^{\infty}X^{[s]}} + \nrm{e^{-\Lambda_{\vep, 0}}u_{\vep, 0} - e^{-\Lambda}u}_{L_{\tilde T^{**}}^{\infty}H_x^s}\right).
\end{align*}
Since $[s]$ is an integer, a direct calculation yields
\begin{align*}
&\nrm{u_{\vep, 0} - u}_{L_{\tilde T^{**}}^{\infty}H_x^{[s]}} \\
&\le \nrm{u_{\vep, 0}(1-e^{-\Lambda_{\vep, 0}+\Lambda})}_{L_{\tilde T^{**}}^{\infty}H_x^{[s]}}
+ \nrm{e^{\Lambda}\left(e^{-\Lambda_{\vep, 0}}u_{\vep, 0} - e^{-\Lambda}u\right)}_{L_{\tilde T^{**}}^{\infty}H_x^{[s]}} \\
&\le C(s, \|u\|_{L_{\tilde T^{**}}^{\infty}H_x^{s}}, \|u_{\vep, 0}\|_{L_{\tilde T^{**}}^{\infty}H_x^{s}} )\left(\nrm{u_{\vep, 0} -u}_{L_{\tilde T^{**}}^{\infty}X^{[s]-1}} + \nrm{e^{-\Lambda_{\vep, 0}}u_{\vep, 0} - e^{-\Lambda}u}_{L_{\tilde T^{**}}^{\infty}H_x^{[s]}}\right).
\end{align*}
Thus, we obtain 
\begin{equation}\label{append1}
\begin{aligned}
&\nrm{u_{\vep, 0} - u}_{L_{\tilde T^{**}}^{\infty}H_x^{[s]}} \\
&\le C(s, \|u\|_{L_{\tilde T^{**}}^{\infty}H_x^{s}}, \|u_{\vep, 0}\|_{L_{\tilde T^{**}}^{\infty}H_x^{s}} )\left(\nrm{u_{\vep, 0} -u}_{L_{\tilde T^{**}}^{\infty}X^{[s]-1}} + \nrm{e^{-\Lambda_{\vep, 0}}u_{\vep, 0} - e^{-\Lambda}u}_{L_{\tilde T^{**}}^{\infty}H_x^{s}}\right) \\
\end{aligned}
\end{equation}
The $L_{\tilde T^{**}}^{\infty}H_x^{s}$-norm of $e^{-\Lambda_{\vep, 0}}u_{\vep, 0} - e^{-\Lambda}u$ is estimated by using Lemma \ref{strest}, \eqref{ep-gauge trans}, and \eqref{gauge trans} as follows:
\begin{align*}
\nrm{e^{-\Lambda_{\vep, 0}}u_{\vep, 0} - e^{-\Lambda}u}_{L_{\tilde T^{**}}^{\infty}H_x^{s}} 
&\le\nrm{\left(U_\vep(t)-U(t)\right)\left(e^{-\Lambda(0)}\phi\right)}_{L_{\tilde T^{**}}^{\infty}H_x^{s}} \\
&\quad+ \nrm{\int_0^tU_{\vep}(t-t')\left(\left(e^{-\Lambda_{\vep, 0}} - e^{-\Lambda}\right)N_{\vep}^{(3)}(u_{\vep, 0})\right)(t')\,dt'}_{L_{\tilde T^{**}}^{\infty}H_x^{s}} \\
&\quad+\nrm{\int_0^tU_{\vep}(t-t')\left(e^{-\Lambda} \left(N_{\vep}^{(3)}(u_{\vep, 0}) - N_{\vep}^{(3)}(u)\right)\right)(t')\,dt'}_{L_{\tilde T^{**}}^{\infty}H_x^{s}} \\
&\quad+\nrm{\int_0^tU_{\vep}(t-t')\left(e^{-\Lambda} \left(N_{\vep}^{(3)}(u) - N_{0}^{(3)}(u)\right)\right)(t')\,dt'}_{L_{\tilde T^{**}}^{\infty}H_x^{s}} \\
&\quad+ \nrm{\int_0^t\left(U_{\vep}(t-t')- U(t-t')\right)\left(e^{-\Lambda}N_{0}^{(3)}(u)\right)(t')\,dt'}_{L_{\tilde T^{**}}^{\infty}H_x^{s}} \\
&\quad+ \vep\nrm{\int_0^tU_{\vep}(t-t')\left(e^{-\Lambda_{\vep, 0}}N^{(2)}(u_{\vep, 0})\right)(t')\,dt'}_{L_{\tilde T^{**}}^{\infty}H_x^{s}} \\
&=: \II_1 + \II_2 + \II_3 + \II_4 + \II_5 + \II_6.
\end{align*}
We consider the estimate for $\II_1$ and $\II_5$ later.
From Lemma \ref{seki} and \eqref{pr diff2 ineq}, we have
\[
\II_2\le C(s, \|u\|_{L_{\tilde T^{**}}^{\infty}X^{s}}, \|u_{\vep, 0}\|_{L_{\tilde T^{**}}^{\infty}X^{s}} ){\tilde T^{**}}\nrm{u_{\vep, 0}-u}_{L_{\tilde T^{**}}^{\infty}X^{s}}.
\]
For $\II_3$ and $\II_4$, \eqref{est of exp f2} yields
\[
\II_3\le C(s, \|u\|_{L_{\tilde T^{**}}^{\infty}X^{s}}, \|u_{\vep, 0}\|_{L_{\tilde T^{**}}^{\infty}X^{s}} ){\tilde T^{**}}\nrm{u_{\vep, 0}-u}_{L_{\tilde T^{**}}^{\infty}H_x^{s}},
\]
\[
\II_4\le C(s, \|u\|_{L_{\tilde T^{**}}^{\infty}X^{s}} )\vep {\tilde T^{**}}.
\]
Also, Lemma \ref{schauder} and \eqref{est of exp f2} yield
\[
\II_6\le C(s, \|u_{\vep, 0}\|_{L_{\tilde T^{**}}^{\infty}X^{s}})\left(\vep^{\fr{1}{2}} {\tilde T^{**\fr{1}{2}}}+\vep {\tilde T^{**}}\right).
\]

The term $\|u_{\vep, 0}-u\|_{L_{\tilde T^{**}}^\infty H_x^{[s]-1}}$ in \eqref{append1} is estimated as follows by using \eqref{mild-sol}:
\begin{align*}
\nrm{u_{\vep, 0}-u}_{L_{\tilde T^{**}}^{\infty}H_x^{[s]-1}} 
&\le \nrm{\left(U_\vep(t)-U(t)\right)\phi}_{L_{\tilde T^{**}}^{\infty}H_x^{[s]-1}} \\
&\quad +\nrm{\int_0^tU_\vep(t-t')\px\left(\lambda\left(u_{\vep, 0}^2-u^2\right)+\mu\left(\abs{u_{\vep, 0}}^2-\abs{u}^2\right)\right)(t')}_{L_{\tilde T^{**}}^{\infty}H_x^{[s]-1}}\\
&\quad + \nrm{\int_0^t\left(U_\vep(t-t') - U(t-t')\right)\px\left(\lambda u^2 +\mu\abs{u}^2\right)(t')}_{L_{\tilde T^{**}}^{\infty}H_x^{[s]-1}} \\
&=: \III_1 + \III_2 + \III_3.
\end{align*}
From Lemma \ref{est of diff free sol}, we have
\[
\III_1
\lesssim \vep^{\frac{1}{2}}{\tilde T^{**\frac{1}{2}}}\nrm{\phi}_{H^{[s]}}.
\]
A direct calculation yields
\[
\III_2\le C(s, \|u_{\vep, 0}\|_{L_{\tilde T^{**}}^{\infty}H_x^{[s]}}, \|u_{\vep, 0}\|_{L_{\tilde T^{**}}^{\infty}H_x^{[s]}}){\tilde T^{**}}\nrm{u_{\vep, 0}-u}_{L_{\tilde T^{**}}^{\infty}H_x^{[s]}}.
\]
For $\III_3$, we consider later.

By taking $\tilde T^{**} = \tilde T^{**}(s, \nrm{\phi}_{X^s})$ sufficiently small, we obtain
\begin{align*}
\limsup_{\vep\to+0}\nrm{u_{\vep, 0}-u}_{L_{\tilde T^{**}}^{\infty}X^s} 
\le C(s, \|\phi\|_{X^s})\limsup_{\vep\to+0}\left(\II_1 + \II_5 + \III_3\right),
\end{align*}
where we used $\lim_{\vep\to+0}(\I_3 + \II_4 + \II_6 + \III_1)=0$.
Hence, the remainder of the proof is to show $\limsup_{\vep\to+0}\left(\II_1 + \II_5 + \III_3\right)=0$. 

We consider the term $\II_1$. 
For $t\in [0, \tilde T^{**}]$, we define 
\[
f_\vep(t):= \nrm{\left(U_\vep(t)-U(t)\right)e^{-\Lambda(0)}\phi}_{H^{s}}.
\] Then, we have $\lim_{\vep\to 0}f_\vep(t)=0$ for all $t$. Therefore, it suffices to show that $\{f_\vep\}$ is equicontinuous on $[0, {\tilde T^{**}}]$. Let $\delta>0$. We can take some $n\in\bn$ such that $\nrm{P_{\ge n+1}\left(e^{-\Lambda(0)}\phi\right)}_{H^s} <\delta$. Thus, we have
\begin{align*}
&\limsup_{t\to t_0}\sup_{\vep\in(0, 1)}\abs{f_\vep(t)-f_\vep(t_0)} \\
&\le \limsup_{t\to t_0}\nrm{\left(U(t)-U(t_0)\right)\left(e^{-\Lambda(0)}\phi\right)}_{H^{s}} +\limsup_{t\to t_0}\sup_{\vep\in(0, 1)}\nrm{\left(U_\vep(t)-U_\vep(t_0)\right)\left(e^{-\Lambda(0)}\phi\right)}_{H^{s}} \\
&\le \limsup_{t\to t_0}\sup_{\vep\in(0, 1)}\nrm{\left(U_\vep(t)-U_\vep(t_0)\right)P_{\le n}\left(e^{-\Lambda(0)}\phi\right)}_{H^{s}} +2\delta \\
&\lesssim \limsup_{t\to t_0}2^{2(n+1)}\abs{t-t_0}\nrm{P_{\le n}\left(e^{-\Lambda(0)}\phi\right)}_{H^s} + \delta \\
&=\delta.
\end{align*}
Since $\delta>0$ is taken arbitrarily, we obtain $\lim_{t\to t_0}\sup_{\vep\in(0, 1)}|f_\vep(t)-f_\vep(t_0)| =0$. Therefore, we obtain $\lim_{\vep\to +0}\II_1=0$.

Finally, we consider the terms $\II_5$ and $\III_3$. 
Lebesgue's dominated convergence theorem yields
\[
g_{\vep}(t):=\nrm{\int_0^t\left(U_\vep(t-t')-U(t-t')\right)\left(e^{-\Lambda}N_{0}^{(3)}(u)\right)(t')\,dt'}_{H^s}\to0
\]
as $\vep\to+0$ for all $t\in[0, {\tilde T^{**}}]$.
Hence, it is sufficient to show that $\{g_{\vep}\}$ is equicontinuous on $[0, \tilde T^{**}]$.
In the following, we may assume $t\ge t_0$.
It holds that
\begin{align*}
&\int_0^t\left(U_\vep(t-t')-U(t-t')\right)\left(e^{-\Lambda}N_{0}^{(3)}(u)\right)(t')\,dt' \\
&\quad-\int_0^{t_0}\left(U_\vep(t_0-t')-U(t_0-t')\right)\left(e^{-\Lambda}N_{0}^{(3)}(u)\right)(t')\,dt' \\
&=\int_{t_0}^t\left(U_\vep(t-t')-U(t-t')\right)\left(e^{-\Lambda}N_{0}^{(3)}(u)\right)(t')\,dt' \\
&\quad+ \left(U_\vep(t-t_0)-1\right)\int_0^{t_0}U_\vep(t_0-t')\left(e^{-\Lambda}N_{0}^{(3)}(u)\right)(t')\,dt' \\
&\quad - \left(U(t-t_0)-1\right)\int_0^{t_0}U(t_0-t')\left(e^{-\Lambda}N_{0}^{(3)}(u)\right)(t')\,dt'.
\end{align*}
We have
\[
\lim_{t\to t_0}\sup_{\vep\in(0, 1)}\nrm{\int_{t_0}^t\left(U_\vep(t-t')-U(t-t')\right)\left(e^{-\Lambda}N_{0}^{(3)}(u)\right)(t')\,dt'}_{H^s} =0
\]
and
\[
\lim_{t\to t_0}\nrm{\left(U(t-t_0)-1\right)\int_0^{t_0}U(t_0-t')\left(e^{-\Lambda}N_{0}^{(3)}(u)\right)(t')\,dt'}_{H^s}=0.
\]
Also, for any $\delta>0$, we can take some $n\in\bn$ such that
\[
\sup_{\vep\in(0, 1)}\nrm{P_{\le n}\left(\int_0^{t_0}U_{\vep}(t_0-t')\left(e^{-\Lambda}N_0^{(3)}(u)\right)(t')\,dt'\right)}\le \delta.
\]
Therefore, a similar argument as in the estimate for $\II_1$ yields that
\[
\lim_{t\to t_0}\sup_{\vep\in(0, 1)}\nrm{\left(U_\vep(t-t_0)-1\right)\int_0^{t_0}U_\vep(t_0-t')\left(e^{-\Lambda}N_{0}^{(3)}(u)\right)(t')\,dt'}_{H^s} =0.
\]
Thus, $g_\vep$ is equicontinuous, and hence we obtain $\lim_{\vep\to 0}\II_5 =0$. 
We can treat the term $\III_3$ as $\II_5$.
This completes the proof.
\end{proof}
\end{appendix}
\section*{Acknowledgement}
The author would like to express his deep appreciation to Mamoru Okamoto for many discussions and very valuable comments. This work was supported by JST SPRING, Grant Number JPMJSP2138.

\providecommand{\bysame}{\leavevmode\hbox to3em{\hrulefill}\thinspace}
\providecommand{\MR}{\relax\ifhmode\unskip\space\fi MR }
\providecommand{\MRhref}[2]{%
  \href{http://www.ams.org/mathscinet-getitem?mr=#1}{#2}
}
\providecommand{\href}[2]{#2}

\end{document}